\documentclass[envcountsect, envcountsame]{llncs}

\usepackage{graphicx}

\usepackage[title]{appendix}

\usepackage{amsmath}
\usepackage{amsfonts}
\usepackage{amssymb}

\usepackage[fleqn]{nccmath}

\usepackage{varioref}

\usepackage{hyperref}

\usepackage[usenames,dvipsnames]{xcolor}
\usepackage{graphicx}
\usepackage{tikz-cd}
\usepackage[all]{xy}

\spnewtheorem*{parametersel}{Parameter selection procedure}{\itshape}{\rmfamily}
\spnewtheorem*{parametercon}{Parameter constraints}{\itshape}{\rmfamily}
\spnewtheorem*{actionwithm}{Acting on the starting surface}{\itshape}{\rmfamily}
\spnewtheorem*{parameterex}{Example parameters}{\itshape}{\rmfamily}
\spnewtheorem*{generators}{Generators}{\itshape}{\rmfamily}
\spnewtheorem*{startingsurface}{Starting surface}{\itshape}{\rmfamily}
\spnewtheorem*{orientationrep}{Representing the orientation}{\itshape}{\rmfamily}
\spnewtheorem*{orientrepresent}{Representing the orientation}{\itshape}{\rmfamily}

\usepackage{thm-restate}

\def\defi{\textsf}
\def\epsilon{\varepsilon}
\def\rho{\varrho}
\def\theta{\vartheta}
\def\phi{\varphi}
\def\tilde{\widetilde}

\def\char{\text{char}}
\DeclareMathOperator{\Cl}{Cl}
\DeclareMathOperator{\coker}{coker}
\DeclareMathOperator{\degrd}{degrd}
\DeclareMathOperator{\disc}{disc}
\DeclareMathOperator{\End}{End}
\DeclareMathOperator{\Frac}{Frac}
\DeclareMathOperator{\Gal}{Gal}
\DeclareMathOperator{\GL}{GL}
\DeclareMathOperator{\Hom}{Hom}
\DeclareMathOperator{\Cent}{Cent}
\DeclareMathOperator{\im}{im}
\DeclareMathOperator{\id}{id}
\DeclareMathOperator{\Jac}{Jac}
\DeclareMathOperator{\Pic}{Pic}
\DeclareMathOperator{\sgn}{sgn}
\DeclareMathOperator{\tr}{tr}
\DeclareMathOperator{\Tr}{Tr}
\DeclareMathOperator{\diag}{diag}
\DeclareMathOperator{\nrm}{n}

\newcommand{\Bpinf}{B_{p, \infty}}

\ExplSyntaxOn
\NewDocumentCommand{\definealphabet}{mmmm}
 {%
  \int_step_inline:nnn { `#3 } { `#4 }
   {
    \cs_new_protected:cpx { #1 \char_generate:nn { ##1 }{ 11 } }
     {
      \exp_not:N #2 { \char_generate:nn { ##1 } { 11 } }
     }
   }
 }
\ExplSyntaxOff

\def\C{\mathbb{C}}
\def\F{\mathbb{F}}
\def\i{\mathbf{i}}
\def\M{\mathcal{M}}
\def\O{\mathcal{O}}
\def\Q{\mathbb{Q}}
\def\Z{\mathbb{Z}}

\def\ZZ{\mathbb{Z}}
\definealphabet{bb}{\mathbb}{A}{Z}

\definealphabet{bf}{\mathbf}{A}{Z}

\def\Ac{\mathcal{A}}
\def\Bc{\mathcal{B}}
\definealphabet{cal}{\mathcal}{A}{Z}

\definealphabet{frak}{\mathfrak}{A}{Z}
\definealphabet{frak}{\mathfrak}{a}{z}

\newcommand{\dg}{\dagger}
\newcommand{\Id}{\mathrm{Id}}

\title{Post-quantum commutative group actions from oriented superspecial abelian varieties}

\author{Wouter Castryck\inst{1} \and Bogdan Adrian Dina\inst{2,3} \and\\ P\'eter Kutas \inst{2,4} \and Gioella Lorenzon\inst{1}}

\institute{KU Leuven, Belgium \and Eötvös Loránd University, Hungary \and  University of Haifa, Israel \and University of Birmingham, United Kingdom}

\authorrunning{Castryck, Dina, Kutas, Lorenzon}

\begin{document}

\pagestyle{plain}

\maketitle

\begin{abstract}
  We study the cryptographic use of Shimura class group actions on superspecial principally polarized abelian varieties  equipped with a CM orientation, within the context of post-quantum cryptography. Besides giving an accessible introduction to the topic, illustrated with many examples, we provide tools for representing CM orientations through higher-dimensional isogenies and for computing the Shimura class group action in the restricted effective group action model.
  Here, key auxiliary results are a version of Kani's lemma ``with CM" and a trick that often allows to halve the ambient dimension. We also give a method for finding superspecial principally polarized abelian surfaces that are oriented by a given CM order, which is extracted from an algebraic framework for studying CM orientations using the Ibukiyama--Katsura--Oort correspondence. We also analyze
  obstructions to transitivity of the Shimura class group action, related to the characteristic-zero notion of CM type. 
  
\end{abstract}

\section{Introduction}

Commutative cryptographic group actions generalize  exponentiation in finite cyclic groups and  inherit some of its versatility in cryptographic design, such as non-interactive key exchange protocols \`a la Diffie--Hellman and threshold signatures, see~\cite{egarega} for an overview.  At the same time, the concept  is general enough to allow for (conjecturally) quantum-resistant instantiations, unlike exponentiation which falls prey to Shor's algorithm.
This generality is, however, relative: the only known fruitful source of post-quantum commutative group actions comes from the theory of complex multiplication (CM). 
Moreover, nearly all known instantiations work in dimension $g =1$, i.e., using elliptic curves, where the groups are ideal-class groups of imaginary quadratic orders, acting on sets of oriented elliptic curves via isogenies. Their cryptographic use was proposed independently by Couveignes~\cite{couveignes} and Rostovtsev--Stolbunov~\cite{rostovtsev-stolbunov}, using ordinary elliptic curves over finite fields, but the focus has now shifted to supersingular elliptic curves for the sake of flexibility and efficiency.
This is implemented, e.g., in~\cite{csidh,scallop,coral}. 

CM theory in arbitrary dimension $g \geq 1$ was developed by Shimura and Taniyama mid-$20$th century~\cite{Shimura}. Unlike the elliptic curve case, the acting group is not simply the ideal-class group of an imaginary quadratic order, but 
a slightly more complicated class group attached to a CM order of rank $2g$. The complication comes from the fact that it acts on principally polarized abelian varieties (PPAVs); the action needs to be compatible with the polarization data.

Given the growing role of higher-dimensional abelian varieties in isogeny-based cryptography, and given the central place of CM theory in the theory of PPAVs, it is a natural goal to
study how Shimura class group actions can be put to cryptographic use. Our purpose is to initiate this study for superspecial PPAVs, which are the appropriate generalizations to higher dimension of supersingular elliptic curves. Sanon's PhD thesis~\cite{SanonThesis} already gave a high-level description of a key exchange protocol from Shimura class groups acting on absolutely simple, ordinary abelian surfaces over a finite field $\F_q$ oriented by their full endomorphism ring.
We are unaware of other literature exploring such constructive uses, let alone in the superspecial setting,  although~\cite{robert_todo} mentions the possibility of extending the module-action framework to modules over CM orders, which would generalize our discussion (which corresponds to rank-$1$ modules). This being said, the Shimura class group action has been on cryptographers' radars for a long time, in tangential studies such as \cite{martindale2018isogeny,bisson2015computing,bisson2017polarised,lauter_robert,brooks-jetchev-wesolowski,IONICA2020385}.

When compared to elliptic curves, the cases $g > 1$ offer two potential advantages: \emph{(i)} one can reach larger class groups for a given field size, especially when working over a prime field $\F_p$ as in CSIDH~\cite{csidh}, which is not possible in large-discriminant frameworks such as~\cite{orientexpress,Houben}, and \emph{(ii)} there potentially exist more prime ideals of a given small norm $\ell$, allowing for a key space that is more concentrated on small primes. But we mainly view this work as a didactic exploration of CM orientations on superspecial PPAVs and the corresponding Shimura class group actions, with a focus on their cryptographic use and spiced up with many examples. We also note that, already  in dimension $g = 1$, orientations and CM actions have proven to be important tools beyond the constructive use of cryptographic group actions \cite{DBLP:conf/asiacrypt/FeoGFKLPSW21,DBLP:journals/cic/MerdyW25,love2020supersingular}.  

\subsubsection{Contributions and paper organization.} Our contributions are organized as follows. 
First, after reviewing CM orientations and Shimura class groups in Section~\ref{sec:background}, with further background to be found in the appendices, we introduce some recurring examples of superspecial CM orientations in Section~\ref{sec:superspecial}. Next, in Section~\ref{sec:evaluating_the_action} we discuss how to set up and evaluate Shimura class group actions within the $\mathbf{g}$-REGA model from~\cite{egarega}. This includes finding the generating set $\mathbf{g}$ and representing the orientation; two auxiliary tools here are an enhanced version of Kani's lemma  incorporating the CM orientation, and a dimension-halving trick often allowing the user to keep working in dimension $g$ (which seems a new observation even for $g = 1$). Along the way we describe the structure of Shimura class groups of CM orders in biquadratic CM fields (i.e., $C_2 \times C_2$ fields), and we break a security measure proposed by Sanon in Remark~\ref{rmk:sanon_break}.
Then in Section~\ref{sec:algebra}, we give a purely algebraic treatment of CM orientations on superspecial principally polarized abelian surfaces (PPASs), within the framework of the Ibukiyama--Katsura--Oort (IKO) correspondence from~\cite{ibukiyama1986supersingular}. We show how to evaluate the Shimura class group action on the IKO-side, and how to produce superspecial PPASs that are oriented by a prescribed CM order; this has applications to the $2$-dimensional version of the ``one endomorphism" problem (often called \textsf{OneEnd} in the literature~\cite{oneend}), on which we elaborate in Appendix~\ref{app:oneend}.  Next, in Section~\ref{sec:orbits}, we delve into the orbits of the Shimura class group action, which is a rich and much more intricate discussion than in the case $g = 1$, led by the characteristic-zero notion of CM type. An interesting by-product of our discussion is a very efficient method for deciding whether two given oriented supersingular elliptic curves lie in the same orbit under the usual ideal-class group action, see Example~\ref{ex:decidingorbits}.   We conclude in Section~\ref{sec:security} with some security considerations.

\subsubsection{Appendices.} To keep the main body contained, a considerable amount of material was deferred to the appendices. Besides additional background in Appendices~\ref{app:background},~\ref{app:bass} and proofs in Appendices~\ref{app:proofevalsection},~\ref{app:almostiso},~\ref{app:algebraproofs}, they contain several results  worth highlighting. 
Most notably, building on Section~\ref{sec:evaluating_the_action}, we present a key exchange protocol built from superspecial CM-oriented PPASs, all defined over $\F_p$, where the real multiplication part is generated by the $p$-th power Frobenius endomorphism. 
The starting orientation is called R2-D4 and described in Appendix~\ref{app:R2families}, while concrete parameters tailored for efficient key exchange are discussed in Appendix~\ref{app:key_exch}. 
Some compression techniques for CM-oriented superspecial PPASs over $\F_p$ can be found in Appendix~\ref{app:compress}.
Along the way, for use in R2-D4 and similar orientations,   we give explicit descriptions of the endomorphism rings of certain well-known supersingular elliptic curves in Appendix~\ref{sec:endrings}, thereby fixing and completing corresponding statements in the existing literature. Away from key exchange, we recall that 
Appendix~\ref{app:oneend} contains results on the $2$-dimensional counterpart of \textsf{OneEnd}. 
In Appendix~\ref{app:reduction} we slightly extend a result due to Goren~\cite{Goren1997}, which describes modulo which unramified primes $p$ a given CM-oriented PPAS in characteristic zero has a superspecial reduction; concretely, unlike Goren, we also cover the $C_2 \times C_2$ case, i.e., the case of CM by a biquadratic field.

\subsubsection{Acknowledgements.}
The authors would like to thank Eyal Goren, the reviewers of a prior version of this paper, and the authors of~\cite{supersingularitytesting,C2C2attack} for helpful input and feedback. For the statement and proof of Lemma~\ref{lem:conductor_lemma} we strongly benefited from interactions with Claude Opus 5.5 (Anthropic PBC). All authors are supported by  the European Union under the Horizon Europe programme with grant agreement No.\ 101309092 (PQC-SCAN), as well as by the CELSA alliance through the MaCro project. Castryck and Lorenzon are also supported by the European Research Council (ERC) under the Horizon 2020 programme, grant agreement No.\ 101020788 (ISOCRYPT), by the Research Council KU Leuven grant C14/24/099, as well as by Cybersecurity Research Flanders with reference number VR20192203. Kutas is partly supported by EPSRC through grant number EP/V011324/1, and together with Dina, he is also supported by the Hungarian Ministry of Innovation and Technology NRDI Office within the framework of the Quantum Information National Laboratory Program.  Kutas is also supported  by NRDI grant ``EXCELLENCE-151343". Lorenzon is funded by Research Foundation – Flanders (FWO) under a PhD Fellowship fundamental research (project number 1139225N). 

\section{Background} \label{sec:background}

We review the notions of principally polarized abelian varieties (PPAVs) equipped with a CM orientation and the Shimura class group action. More elementary background 
can be found in Appendix~\ref{app:background}, to which we have deferred all proofs.

\subsection{CM orientations} \label{subsec:CMorientations}

Henceforth $K$ denotes a \defi{CM field} of degree $2g$, i.e., a totally imaginary quadratic extension of a totally real degree-$g$ number field $K_0 \supset \Q$. Observe that $K$ comes equipped with an automorphism $x \mapsto \overline{x}$ with fixed field $K_0$ that we can think of as complex conjugation. We write $K_0^+$ for the set of totally positive elements in $K_0$, i.e., whose minimal polynomial over $\Q$ has positive real roots only.

\begin{definition}\label{def:orient}
    A \defi{$K$-orientation} on a $g$-dimensional PPAV $(A, \lambda)$ is a homomorphism $\iota : K \hookrightarrow \End^0(A)$ such that $\iota(\overline{\alpha}) = \iota(\alpha)^\dagger$ for all $\alpha \in K$, where ${}^\dagger$ denotes the Rosati involution. We say that $\iota$ is moreover a \defi{primitive $R$-orientation}, for some order $R \subset K$, if $\iota^{-1} \End(A) = R$.
\end{definition}
\noindent Since $\End(A)$ is closed under the Rosati involution, primitive $R$-orientations can only occur for orders that are self-conjugate, i.e., $\overline{R} = R$. \footnote{E.g., the order $\Z[\sqrt{2} + \sqrt{-1}] \subset \Q(\sqrt{2}, \sqrt{-1})$ does not satisfy $\overline{R} = R$.} We call such an order a \defi{CM order}. A sufficient condition is that $R$ has maximal real multiplication, i.e., $R \cap K = \Z_{K_0}$.
We remark that
  we use the terminology of ``orientations" for the sake of analogy with cryptographic customs in the case $g = 1$, but we note that many mathematical texts just talk about PPAVs ``with CM". 

Consider two $K$-oriented PPAVs $(A_1, \lambda_1, \iota_1), (A_2, \lambda_2, \iota_2)$. Let 
$\mu  \in  K_0^+ \, \cap \, \iota_1^{-1}\End(A_1)$, 
i.e., $\mu \in K$ is such that $\iota_1(\mu)$ is a totally positive endomorphism on $A_1$ (e.g., we can choose $\mu = N \in \Z_{>0}$). Consider an $\iota_1(\mu)$-isogeny $\varphi : (A_1, \lambda_1) \to (A_2, \lambda_2)$. Then we say that this polarized isogeny is an \defi{oriented isogeny} if for all 
$\alpha \, \in \, \iota_1^{-1} \End(A_1) \, \cap \, \iota_2^{-1} \End(A_2)$
we have $\varphi \circ \iota_1(\alpha) = \iota_2(\alpha) \circ \varphi$. We mark this by writing $\varphi : (A_1, \lambda_1, \iota_1) \to (A_2, \lambda_2, \iota_2)$. If both orientations are primitive for the same order $R \subset K$ then we say that it concerns a \defi{horizontal} oriented isogeny. An \defi{oriented isomorphism} is an oriented $1$-isogeny.

\begin{definition}
  Let $K$ be a CM field of degree $2g$. We write $\M_k(K)$ for the set of $K$-oriented PPAVs $(A, \lambda, \iota)$ over $k$, up to oriented isomorphism. For a CM order $R \subset K$, we write $\M_k(R)$ for the subset of primitively $R$-oriented PPAVs. 
\end{definition}

\begin{remark}
  If $g= 1$ then we suppress the canonical principal polarization from the notation and denote the elements of $\M_k(K)$ as $(E, \iota)$. The set $\M_k(R)$ may be empty. E.g., this is the case if $g = 1$, $p > 0$ and $p \mid [\Z_K : R]$, see Section~\ref{sec:orbits}.
\end{remark}

\subsection{The Shimura class group} \label{sec:shimura_cl_group}

To any CM order $R$ in a CM field $K$ we attach its \defi{Shimura class group}, denoted $\frakC(R)$. The  nomenclature in the existing literature is heterogeneous: the Shimura class group also goes under the name \defi{polarized class group}, while our main reference~\cite{pila} even just talks about ``a Picard group", denoting it as $\Pic_\ast R$.  Explicitly, consider pairs $(\fraka, \mu)$ where $\fraka$ is an invertible fractional $R$-ideal such that $\fraka \bar{\fraka} = \mu R$, 
and $\mu \in K_0^+$ (i.e., $\mu$ is a totally positive element of $K_0$). Endowed with component-wise multiplication, such pairs form a group we denote by $I_R$. 
Pairs $(\mu R, \mu \bar{\mu})$ with $\mu \in K^\times$ form a subgroup we denote by $P_R$.
The Shimura class group of $R$ is the quotient 
\begin{equation}\label{def:Shimura-Cl}
    \frakC(R) := I_R / P_R.
\end{equation} 
We denote by $\sim$ the equivalence induced by $P_R$ on pairs $(\fraka, \mu) \in I_R$. 
If $g = 1$ then the existence of $\mu \in K_0^+ = \Q_{>0}$ such that $\fraka \bar{\fraka} = \mu R$
is an automatic condition and it determines $\mu$ uniquely (namely, it must be the norm of $\fraka$). Thus, in this case, it is immediate that $\frakC(R) \cong \Cl(R)$.

\begin{lemma} \label{lem:different_mus}
  For a given invertible fractional ideal $\fraka \subset R$ there are at most $2^{g-1}$ Shimura classes $[(\fraka, \mu)] \in \frakC(R)$.
\end{lemma}

\begin{example}
If $K_0 = \Q(\sqrt{2})$ then $\varepsilon_1 = 1 + \sqrt{2}$ is not totally positive, so there is one class only. If $K_0 = \Q(\sqrt{3})$ then $\varepsilon_1 = 2 + \sqrt{3}$ is totally positive, so here there will be two classes, typically, but there are exceptions: e.g., if $K = K_0(\alpha)$ with $\alpha = \sqrt{-\varepsilon_1}$ then clearly $\varepsilon_1 = \alpha \bar \alpha$ so, here too, there is one class only. 
\end{example}

As explained in the proof of Theorem~\ref{thm:action} below, if $\fraka \subset R$ then we can act with the Shimura ideal $(\fraka, \mu)$ on any primitively $R$-oriented PPAV $(A, \lambda, \iota)$ by means of an $\iota(\mu)$-isogeny. Since all existing formulae for isogeny evaluation in higher dimension restrict to $N$-isogenies for some positive integer $N$, in practice we will restrict to Shimura ideals for which $\mu = N \in \Z_{> 0}$, even though in Corollary~\ref{cor:evaluate} we will explain how the KLPT${}^2$ algorithm from~\cite{DBLP:conf/crypto/CastryckDKLPT25} can be used, in principle, to by-pass this restriction in the case of an oriented superspecial PPAS with ``known" endomorphism ring. The Shimura classes of the form $[(\fraka, N)]$ with $\fraka \subset R$ and $N \in \Z_{>0}$ form a subgroup $\frakC(R)_{\Z} \subset \frakC(R)$ whose index is bounded by the following statement, for which we provide a simple proof; this also follows from~\cite[Lem.\,6]{bisson2017polarised}. 

\begin{lemma} \label{lem:Shimura_integer_norm_subgroup}
We have $\frakC(R)^2 \subset \frakC(R)_\Z$, therefore $[\frakC(R) : \frakC(R)_\Z] \mid \#\frakC(R)[2]$.
\end{lemma}

The Shimura class group fits into the following exact sequence:
\begin{equation} \label{eq:shimura_exact_sequence}
      1 \xrightarrow{~} \frac{(R^\times_{0})^+}{N_{K/K_0}(R^\times)} \xrightarrow{u~\mapsto~[(R, u)]} \frakC(R) \xrightarrow{[(\fraka, \mu)]~\mapsto~[\fraka]} \Cl(R) \xrightarrow{N_{K/K_0}} \Cl^+(R_0),
\end{equation}
where $R_0 = R \cap K_0$, $(R^\times_{0})^+$ is the subgroup of totally positive units, and $\Cl^+(R_0)$ is the narrow class group (which, as we recall, is the group of invertible fractional ideals modulo the subgroup of principal fractional ideals admitting a totally positive generator). For instance, this implies that $\#\frakC(R)[2] \leq \#\Cl(R)[2] \cdot U$ with $U = [ (R^\times_{0})^+ : N_{K/K_0}(R^\times) ] \leq 2^{g-1}$, where the last bound arises as in the proof of Lemma~\ref{lem:different_mus}.
Under mild conditions, the last map is surjective.
For example, this is true if $R = \Z_K$ and at least one prime ideal of $R_0$ ramifies in $K$~\cite[Proof of Thm.\,3.1]{BrokerGruenewaldLauter}
(that theorem also states that if $g = 2$ then this can only fail if $\Gal(K, \Q) \cong C_2 \times C_2$).
So the \defi{Shimura class number} $\frakh(R)= \# \frakC(R)$ can often be computed as $h(R) U / h^+(R_0)$. An algorithm for computing Shimura class groups was described in~\cite[\S4]{EngeThome2014} (implemented in PARI/GP).

We are now ready to state our main theorem of interest: 
\begin{theorem} \label{thm:action}
  Assume that the set $\M_k(R)$ is non-empty. There is a free action 
  \[ \frakC(R) \times \M_k(R) \to \M_k(R) : \left( \, [(\fraka, \mu)],  (A, \lambda, \iota) \, \right) \mapsto [(\fraka, \mu)] \star  (A, \lambda, \iota). \]
  It can be assumed w.l.o.g.\ that $\fraka \subset R$ and has norm coprime to $p$, in which case
  the image arises as the codomain of an $\iota(\mu)$-isogeny with kernel $\cap_{x \in \fraka} \ker \iota(x)$.
\end{theorem}
\noindent At this point we make no claims on the orbits of the action. This is discussed in Section~\ref{sec:orbits}.

\section{First examples of superspecial CM orientations} \label{sec:superspecial}

The focus of our paper lies on the case of superspecial abelian varieties (and surfaces in particular).
The easiest way to generate a $g$-dimensional PPAV that is oriented by a CM field $K$ of degree $2g$ is by considering $A = E^g$, for some supersingular elliptic curve $E$, equipped with the product polarization $\lambda$. 
The curve $E$ admits an isomorphism $\theta : \O \to \End(E)$ from a maximal order $\O$ in the unique definite quaternion algebra $B_{p, \infty}$ over $\Q$ that is ramified at $p$. By component-wise application we obtain an isomorphism
$\theta : M_g(\O) \to \End(A)$
that we denote by the same letter. Therefore giving a $K$-orientation on $(A, \lambda)$ is the same as giving an embedding $\iota : K \hookrightarrow M_g(B_{p, \infty})$
such that $\iota(\bar \alpha) = \iota(\alpha)^*$ for all $\alpha \in K$; here ${}^*$ denotes conjugate-transpose, which under $\theta$ indeed corresponds to the Rosati involution on $A$ according to Example~\ref{ex:conjugate_transpose}. For a CM order $R \subset K$, a primitive $R$-orientation is then a $K$-orientation mapping $R$ into $M_g(\O)$, in such a way that $\iota(\alpha) \notin M_g(\O)$ for any $\alpha \in K \setminus R$.

\begin{example} \label{ex:example_orientation}
  Let $p \equiv 3 \bmod 4$ and consider the elliptic curve $E : y^2 = x^3 - x$ over $\F_{p^2}$, which is supersingular. Fix a square root $\i \in \F_{p^2} \setminus \F_p$ of $-1$. Consider $A = E^2$ equipped with the product polarization $\lambda$. It is well-known that
  $\theta \ : \ i \mapsto ( \, (x, y) \mapsto (-x, \i y) \, ), \ j \mapsto \pi_p$
  establishes an isomorphism between the maximal order
  $\O = \Z\langle 1, i, (1+j)/2, (i + k)/2 \rangle \subset \Bpinf$ and $\End(E)$.
  Assume that $m = p + 4$ and $d = (p + 1)/4$ are coprime, square-free integers. 
  We claim that
  \[ \iota \ : \ \frac{1 + \sqrt{-m}}{2} \mapsto \begin{pmatrix} \frac{1 + j}{2} & i \\ i & \frac{1 + j}{2} \end{pmatrix}, \ \sqrt{d} \mapsto \begin{pmatrix} 0 & \frac{1+j}{2} \\ \frac{1-j}{2} & 0 \end{pmatrix}  \]
  is a primitive  $R$-orientation on $(A, \lambda)$, where $R = \Z[(1 + \sqrt{-m})/2, \sqrt{d}]$; as above we identify elements of $M_2(\O)$ with their images in $\End(A)$ under $\theta$. If we denote the two matrices by $M$ and $D$, then in order to verify that this indeed concerns a CM orientation, the key properties to check are that
  \begin{itemize}
      \item $M$ and $D$ commute,
      \item $M^2 - M + (m+1)/4 \Id_2 = 0$ and $D^2 - d \Id_2 = 0$,
      \item $M^\ast = 1 - M$ and $D^\ast = D$,
  \end{itemize}
  which is easily done.
  To verify primitivity, note that if $d \not \equiv 1 \bmod 4$ then $R = \Z_K$ (this follows, e.g., from~\cite[p.\,68]{lang_ANT}), in which case primitivity is automatic. To cover the case $d \equiv 1 \bmod 4$ it is enough to observe that $\iota((1 + \sqrt{d})/2) \notin M_2(\O)$.

  As a toy example, letting $p = 103$ we find $m = 107$ and $d = 26$. Using that the fundamental unit $\sqrt{26} + 5$ of  $\Q(\sqrt{26})$ is not totally positive and that $107 R_0$ ramifies in $K$,\footnote{It can in fact be argued that the existence of a ramifying prime follows from the fact that the fundamental unit is not totally positive.} from~\eqref{eq:shimura_exact_sequence} we see that the sequence
  \[       1 \xrightarrow{~}  \frakC(R) \xrightarrow{[(\fraka, \mu)]~\mapsto~[\fraka]} \Cl(R) \xrightarrow{N_{K/K_0}} \Cl^+(R_0) \xrightarrow{~} 1  \]
  is exact. With the use of Magma or SageMath one then verifies that $\frakC(R)$ has order $42$ and is generated by the classes of
    \[ \mathfrak{t} = \left( (3, \sqrt{-107} - 1), 3 \right), \qquad  \mathfrak{s} = \left( (7, \sqrt{-107} \sqrt{26} - 2), 7 \right) \]
  (and cyclically generated by the class of their product $\mathfrak{s} \mathfrak{t}$). Then, for instance, using the $(3, 3)$-isogeny formulae from~\cite{BFT} one  calculates that
  $ [\frakt] \star (A, \lambda, \iota) = (\Jac(H), \lambda_1, \iota_1 ) $
    where $H : y^2 = 102x^6 + 25x^5 + 57x^4 + 4x^3 + 41x^2
    + 46x + 14$ and $\lambda_1$ is the canonical principal polarization on $\Jac(H)$, and where $\iota_1$ denotes the push-forward orientation (which we do not spell out here).
    See Section~\ref{subsec:representation} for more details on how to evaluate the Shimura class group action in practice.
\end{example}

\begin{example} \label{ex:secondorientation}
  Now assume that moreover $d = a^2 + b^2$ for certain integers $a, b \in \Z$ (e.g., $26 = 5^2 + 1^2$ in the toy example). Since $d$ is square-free we can assume that $b$ is odd. Then a similar reasoning shows that
  \[ \iota' \ : \ \frac{1 + \sqrt{-m}}{2} \mapsto \begin{pmatrix} i + \frac{1 + j}{2} & 0 \\ 0 & i + \frac{1 + j}{2} \end{pmatrix}, \ \sqrt{d} \mapsto \begin{pmatrix} a & b \\ b & -a \end{pmatrix}  \]
  is another $R$-orientation on $(A, \lambda)$. It is always primitive because using that $b$ is odd one sees that $\iota'((1 + \sqrt{d})/2) \notin M_2(\O)$.  In Example~\ref{ex:example_orientation_revisited} we will argue that $(A, \lambda, \iota)$ and $(A, \lambda, \iota')$ cannot lie in the same orbit under the action of $\frakC(R)$.
\end{example}

\begin{example} \label{ex:jacexample}
Away from products, let $p \neq 2, 5$ and consider the genus-$2$ curve $C : y^2 = x^5 - 1$ together with its Jacobian $A = \Jac(C)$, equipped with the canonical polarization. This PPAS is superspecial if and only if $p \equiv 4 \bmod 5$. For the CM order $R = \Z[\zeta_5]$ we have a primitive orientation $\iota : R \hookrightarrow \End(A)$ sending $\zeta_5$ to the automorphism of $A$ induced by $ C \to C : (x, y) \mapsto (\xi_5 x, y)$; here $\zeta_5$, $\xi_5$ denote primitive $5$-th roots of unity in $\C$ and $k$, respectively. 
\end{example}

\begin{example}[R2-D2] \label{ex:R2D2}
We fix a supersingular elliptic curve $E / \F_{p^2}$ for which $\pi_{p^2} = -p$ and let $(A, \lambda)$ be its Weil restriction, i.e.,  
\[ A = W_{\F_{p^2} / \F_p}(E) \cong_{\bar \F_p} E \times E^{(p)} \]
equipped with the product polarization.
We choose $E$ such that it comes equipped with an explicit isogeny 
$\tau : E \to E^{(p)}$, satisfying $\hat \tau \circ \pi_p = \pi_p \circ \tau$,  
of square-free degree $d$ not a multiple of $p$. Such curves were studied in~\cite{ChenuSmith} and a  method for finding large examples can be found in~\cite[\S5]{orientexpress}.
Then, an orientation is given by
\[ \iota : \Q(\sqrt{d}, \sqrt{-p}) \to \End^0(A) : \sqrt{d} \mapsto \begin{pmatrix} 0 & \hat \tau \\ \tau & 0 \end{pmatrix}, \sqrt{-p} \mapsto \Pi_p = \begin{pmatrix} 0 & \pi_p \\ \pi_p & 0 \end{pmatrix}, \]
where we note that $\Pi_p$ is the $p$-th power Frobenius endomorphism on $A$.
These matrices are fixed, resp.\ negated, under the Rosati involution (= conjugate-transpose), they square to $d$, resp.\ $-p$, and they commute, so this is a well-defined field embedding.
It typically concerns a primitive $R$-orientation with $R = \Z[\sqrt{d}, \sqrt{-p}]$. Indeed, clearly $(1 + \sqrt{d})/2, (1 + \sqrt{-p})/2$ do not map to endomorphisms, so the only concern is whether $(1 + \sqrt{-dp})/2$ maps to an endomorphism. This concern is void if $-dp \not \equiv 1 \bmod 4$ (true for the parameters from~\cite{orientexpress}) because then $(1 + \sqrt{-dp})/2$ is not an algebraic integer. In general, one can verify that the orientation is primitive if and only if
 $(1 + \pi_p \circ \tau)/2 \notin \End(E)$.
\end{example}

In the above examples, the primitivity of the orientation was easily established. In general, however,   finding the primitive order $R = \iota^{-1} \End(A)$ from a description of $\iota$ on a suborder $R' \subsetneq R$ may be a very hard problem. Indeed, this is similar to computing the endomorphism ring of an ordinary abelian surface $A$ over a finite field $\F_q$ from its Frobenius order $\Z[\pi_q, \hat{\pi}_q]$, for which no efficient algorithms are available~\cite{bisson2015computing}.

\begin{remark}
    R2-D2 is part of a larger family of orientations called R2, discussed in Appendix~\ref{app:R2families}. The suffix D2 refers to the Galois group $D_2 \cong C_2 \times C_2$ of $K = \Frac(R)$ over $\Q$.
\end{remark}

\begin{remark}
    In Section~\ref{sec:algebra} we will develop the theory of CM orientations on the algebraic side of the IKO correspondence. Within that framework, we will discuss an easy technique for producing superspecial PPASs together with a prescribed CM orientation, see Section~\ref{sec:constructing_orientations}. 
    In Appendix~\ref{app:R2families} we will put Example~\ref{ex:R2D2} inside the larger family of explicit $\F_p$-rational orientations that we call R2.
\end{remark}

\section{Evaluating the Shimura class group action} \label{sec:evaluating_the_action}

In this section we describe the main auxiliary tools needed for cryptographic applications of the Shimura ideal class group action on oriented superspecial PPAVs in the $\mathbf{g}$-REGA model (e.g., non-interactive key exchange).

\subsection{On the generating set $\mathbf{g}$} \label{subsec:genset}

 As mentioned in Section~\ref{sec:shimura_cl_group}, we will act with $\frakC(R)_{\Z}$ in practice. Recall that this may be a strict subgroup of the full Shimura class group (but when choosing parameters carefully, equality may be enforced).
 We refer the reader to Section~\ref{sec:actingIKO} for how to act with elements of $\frakC(R) \setminus \frakC(R)_{\Z}$ in dimension $g = 2$, but this method requires knowledge that is not available to Bob when extracting the shared secret from Alice's public key, and vice versa.

 Similar to class group actions like CSIDH~\cite{csidh} and SCALLOP~\cite{scallop}, we choose the elements of our generating set $\mathbf{g}$ to be  representable by Shimura ideals $(\fraka, \ell)$ for small odd primes $\ell \nmid \disc R$,\footnote{This is a convenient assumption but it can be relaxed, to allow for horizontal isogenies coming from ramified primes; since this adds technicality while not adding much to $\frakC(R)$ (the corresponding ideal classes have very small order), we omit this.} and we collect enough (yet polylogarithmically many) such ideals for these to generate $\frakC(R)_{\Z}$ with overwhelming probability (heuristically). Note: the inverse of $[(\fraka, \ell)]$ is represented by $(\overline{\fraka}, \ell)$. In Appendix~\ref{app:proofevalsection} we prove that, in dimension $g = 2$, the options for $\fraka$ are the following.
 
 \begin{lemma} \label{lem:possible_splittings}
   For a given CM order R of degree $4$ and a prime number $\ell \nmid \disc R$
   there exists an ideal $\fraka \subset R$ such that $(\fraka, \ell) \in I_R$ if and only if $\ell R$ splits as either
   \begin{equation} \label{eq:goodpatterns} \ell R = \frakp_1 \overline{\frakp}_1 \frakp_2 \overline{\frakp}_2  \qquad \text{or} \qquad \ell R = \frakp \overline{\frakp} 
   \end{equation}
   for degree-$1$ prime ideals $\frakp_1, \frakp_2$, resp., a degree-$2$ prime ideal $\frakp$. The corresponding options for $\fraka$ are $\frakp_1 \frakp_2, \frakp_1 \overline{\frakp}_2, \overline{\frakp}_1 \frakp_2, \overline{\frakp}_1 \overline{\frakp}_2$, resp., $\frakp, \overline{\frakp}$.
 \end{lemma}

 \noindent One sees that if $\ell$ splits completely then we obtain two pairs $\frakl, \frakl^{-1}, \frakl', \frakl'^{-1} \in \frakC(R)_\Z$, while if $\ell$ splits into two complex-conjugate primes, we end up with one pair $\frakl, \frakl^{-1} \in \frakC(R)_{\Z}$.  E.g., this means that if we consider a set of primes $\ell_1, \ldots, \ell_r$ that split completely along with the corresponding set $\mathbf{g} = \{ \frakl_1, \frakl_1', \ldots, \frakl_r, \frakl_r' \}$, then the number of expressions of the form
 \begin{equation} \label{eq:keys}
   \frakl_1^{e_1} \frakl_1'^{e_1'} \cdots \frakl_r^{e_r} \frakl_r'^{e_r'}, \qquad |e_i|, |e_i'| \leq B
 \end{equation}
 for some bound $B \in \Z_{\geq 0}$ equals $(2B+1)^{2r}$, potentially yielding the same amount of elements of $\frakC(R)_\Z$. 
 More generally, for arbitrary $g$, we end up with $2^{g-1}$ pairs in the completely split case, hence 
 \[ (2B+1)^{r2^{g-1}} \] many expressions.
This growing ``key space" for fixed $r$ and $B$ is one of the features of working in higher dimension.

\begin{remark} 
The asymptotic proportions of primes factoring as in~\eqref{eq:goodpatterns} 
depend on $\Gal(L, \Q)$ where $L$ denotes the Galois closure of $K$. Using Chebotarev's density theorem, one checks that these proportions
are given by $\frac{1}{4}$, resp.\ $\frac{1}{2}$, in the case $C_2 \times C_2$, by $\frac{1}{4}$, resp.\  $0$, in the case $C_4$, and by $\frac{1}{8}$, resp.\ $\frac{1}{4}$, in the case $D_4$.
In practice, see e.g.\ Appendix~\ref{app:key_exch}, we can bias this in our favor and choose parameters such that as many small primes $\ell$ as possible split completely. 
\end{remark}

In order to find the factorization of $\ell R$, one can look for an order $\Z[\gamma] \subset R$ (i.e., a monogenic suborder) such that every element of $R$ can be written as a $\Q$-linear combination of $1, \gamma, \ldots, \gamma^{2g - 1}$ without denominators that are $0$ mod $\ell$ (e.g., this is true as soon as $\ell \nmid [R : \Z[\gamma]]$). 
Writing $f(x) \in \Z[x]$ for the minimal polynomial of $\gamma$, the factorization of $\ell R$ can be obtained from that of $f(x) \bmod \ell$. For $g = 2$ the details are as follows; it is convenient to precompute\footnote{An elegant way of computing $r$ directly from $f(x)$ is through its cubic resolvent: in cases $C_4, D_4$ it concerns the unique integer root of the cubic resolvent, while in the $C_2 \times C_2$ case it can be characterized as the largest root; see~\cite{conrad_cubicquartic}.} $r := \gamma_1 \bar{\gamma}_1 + \gamma_2 \bar{\gamma}_2 \in \Z$, where
      $\gamma_1 = \gamma, \bar{\gamma}_1, \gamma_2, \bar{\gamma}_2 \in \C$ are the roots of $f(x)$.
    \begin{itemize}
      \item If $f(x) \equiv f_1(x) f_2(x) \bmod \ell$ for irreducible quadratic factors $f_i(x)$ such that $f_1(0) + f_2(0) \equiv r \bmod \ell$ then $\ell R = \frakp \overline{\frakp}$ with $\frakp = (\ell, f_1(\gamma))$. 
      \item If $f(x)$ factors completely mod $\ell$ then we label the roots as $a_1, b_1, a_2, b_2 $ of $f(x) \bmod \ell$ in  such a way  that $r \equiv a_1b_1 + a_2b_2 \bmod \ell$. Then we can conclude that $\ell R = \frakp_1 \overline{\frakp}_1 \frakp_2 \overline{\frakp}_2$ with $\frakp_i = (\ell, \gamma - a_i)$ for $i = 1, 2$.
\end{itemize}

\subsubsection{The $C_2 \times C_2$ case.}
There is an alternative  
way of describing an ``easy" generating set $\mathbf{g}$ in the case of a CM field $K$ with $\Gal(K, \Q) \cong C_2 \times C_2$, even though this time the group generated by $\mathbf{g}$ may be a strict subgroup of $\frakC(R)_{\Z}$ (but of very small index). 
The starting point is that $K$  has three quadratic subfields: the real field $K_0$ and two imaginary fields $K_1, K_2$. We let $R_i = R \cap K_i$ for $i = 1,2$.
It is then natural to consider the following map:
\[ \kappa_R : \Cl(R_1) \times \Cl(R_2) \to \frakC(R) : ([\fraka_1], [\fraka_2]) \mapsto [((\fraka_1 R)( \fraka_2 R), N_{K_1/\Q}(\fraka_1) N_{K_2/\Q}(\fraka_2))], \]
seen as taking values in $\frakC(R)_{\Z}$.
In Lemma~\ref{lem:almost_iso} below we argue that this group homomorphism is close to being an isomorphism (at least in the case $R = \Z_K$).
We can then build $\mathbf{g}$ from a set of ``easy" ideal classes 
\[ [\frakm_1 ], [\frakm_2], \ldots, [\frakm_s] \in \Cl(R_1), \qquad [\frakl_1 ], [\frakl_2], \ldots, [\frakl_t] \in \Cl(R_2),  \]
e.g., as chosen in CSIDH or SCALLOP,
by letting it  be the collection of images of $([\frakm_i], 1)$, $i = 1, \ldots, s$ and $(1, [\frakl_i])$, $i = 1, \ldots, t$. 

\begin{restatable}{lemma}{AlmostIso} \label{lem:almost_iso}
    Assume that $K$ is a $C_2 \times C_2$ CM field, and let $K_0$ be the real quadratic subfield and $K_1, K_2$ the imaginary quadratic subfields. 
    If $\Z_K^\times$ contains no roots of unity other than $\pm 1$,  then 
    \begin{align*} |\ker{\kappa_{\Z_K}}| \leq & \frac{1}{2q(K)}\prod_{\ell \text{ prime}}{e(\ell)}, \\ |\coker{\kappa_{\Z_K}}| \leq & \frac{1}{4q(K)} |\coker{N_{K/K_0}}| \prod_{\ell \text{ prime}}{e(\ell)},
    \end{align*}
    where $q(K) = [\Z_K^\times : \Z_{K_0}^\times] \in \{1,2\}$,  $e(\ell)$ is the ramification index of a prime $\ell$ in $K$, and $N_{K/K_0}: \Cl(\Z_K) \to \Cl^+(\Z_{K_0})$ denotes the ideal class norm map.
\end{restatable}

\begin{proof}
    The proof follows from results by Lemmermeyer \cite{lemmermeyer_kuroda} and 
    the exact sequence~\eqref{eq:shimura_exact_sequence}. We refer to Appendix \ref{app:almostiso}  for all the details. \qed
\end{proof}

As mentioned in Section \ref{sec:shimura_cl_group}, a sufficient condition for $|\coker N_{K/K_0}| = 1$ is that the extension $K/K_0$ ramifies at a prime of $K_0$. 
A theorem characterizing such $C_2 \times C_2$ fields can be found in \cite[Thm.\,2]{SpearmanWilliams}.
We note that when $R$ is not a maximal order, a similar bound on $\ker \kappa_R$ can also be proven (see Corollary \ref{cor:kerbound} in Appendix \ref{app:almostiso}). We expect that a similar bound on the cokernel also holds for $\kappa_R$, but the proof seems less straightforward.  

\subsection{Kani's lemma with CM} \label{sec:Kani_with_CM}

In  Section~\ref{subsec:representation} we will use the machinery of higher-dimensional isogeny representations for the purpose of representing orientations. Let us quickly revisit the case of a polarized isogeny $\varphi: A_1\to A_2$ between PPAVs of dimension $g$, of reduced degree $d \in \Z_{>0}$. The
higher-dimensional representation of $\varphi$ exploits Kani's lemma \cite{Kani1997} to encode $\varphi$ into a higher-dimensional isogeny $\Phi$
of smooth reduced degree $N$, so that $\Phi$, and hence $\varphi$, can be
evaluated using a chain of isogenies of small degree.
The classical method makes use of an auxiliary matrix $\alpha$ with entries in $\Z$, coming from an expression of $n := N - d$ as a sum of $u = 1, 2, 4$ squares (see \cite[Prop.~5.15]{robert-survey}); by Lagrange's theorem, $u=4$ always succeeds.

For the purpose of representing orientations, we wish to generalize this to the setting where $A_1$ and $A_2$ have an explicit $R$-orientation for some CM order $R$ in a CM field $K$ of degree $2g$.
Two things change. 
First, the auxiliary matrix may be taken with entries in $R$ rather than in $\mathbb{Z}$, 
and so the condition on $n$ becomes about relative norms of elements of $R$ rather than about sums of integer squares. 
Second, the $d$-isogeny $\varphi$ does not necessarily have $d \in \Z_{>0}$, as $d$ can be any element of $R_0^+ = R_0 \cap K_0^+$, i.e., any totally positive element of $R_0$. 

\begin{theorem}\label{thm:embedding}
Let $(A_1,\lambda_1,\iota_1)$ and $(A_2,\lambda_2,\iota_2)$ be $R$-oriented PPAVs of dimension $g$,
and let $\varphi\colon (A_1, \lambda_1, \iota_1)\to (A_2, \lambda_2, \iota_2)$ be an $R$-oriented $d$-isogeny with $d\in R_0^{+}$. Let
$n\in R_0^{+}$ be such that
$ N:=d+n\in\Z$,
and let $\alpha = (\alpha_{j,k})_{1 \leq j,k \leq u}\in M_u(R)$ be such that $\alpha^{*}\alpha=n\,\Id_u$, where $\alpha^*$ denotes the conjugate-transpose of $\alpha$. 
Write $\alpha_{A_i} = \iota_{A_i}(\alpha) := (\iota_{i} (\alpha_{j,k}))_{1 \leq j, k \leq u} \in\End(A_i^u)$. Then
\[
  \Phi=\begin{pmatrix}\varphi\,\Id_u&\alpha_{A_2}^{\dg}\\[2pt]-\alpha_{A_1}&\varphi^{\dg}\Id_u\end{pmatrix}
  \;\colon\;
  \bigl(A_1^u\times A_2^u,\ \lambda_1^{\times u}\times\lambda_2^{\times u}\bigr)
  \to
  \bigl(A_2^u\times A_1^u,\ \lambda_2^{\times u}\times\lambda_1^{\times u}\bigr)
\]
is an $N$-isogeny of PPAVs of dimension $2ug$; in particular
$\Phi^{\dg}\Phi=[N]$ and $\deg\Phi=N^{2ug}$. Moreover:
\begin{enumerate}
\item[(i)] $\varphi=\rho\circ\Phi\circ\jmath$, where $\jmath\colon A_1\to A_1^u\times A_2^u$ is the
inclusion of the first factor and $\rho\colon A_2^u\times A_1^u\to A_2$ is the projection onto the
first factor;
\item[(ii)] if the ideals $dR_0$ and $nR_0$ are coprime in $R_0$ and $N$ is coprime to the base field characteristic, then
\[
\begin{aligned}
  \ker\Phi
  &=\bigl\{\bigl(-\alpha_{A_1}^{\dg}(P),\,\varphi\,\Id_u(P)\bigr) \mid P\in A_1^u[N]\bigr\}\\
  &=\bigl\{\bigl(\iota_1(d) \Id_u (P),\,\alpha_{A_2}\varphi\,\Id_u(P)\bigr) \mid P\in A_1^u[N]\bigr\}.
\end{aligned}
\]
\end{enumerate}
\end{theorem}

\begin{remark}\label{rmk:meet-in-the-middle}
Point \emph{(ii)} of Theorem~\ref{thm:embedding} describes $\ker\Phi$ in terms of the action of $\varphi$ on $A_1[N]$, so the images of the $2g$ points in a
basis of $A_1[N]$ are enough to determine $\Phi$. Assuming that
$N$ is a square, less than that actually suffices: the action of $\varphi$ on $A_1[\sqrt N]$ already
determines $\Phi$, by the meet-in-the-middle technique of \cite[\S6.4]{sidhRobert}, which we detail further in Remark~\ref{rmk:app_meet-in-the-middle}.
\end{remark}

To apply Theorem~\ref{thm:embedding} one has to produce, for a prescribed $n\in R_0^{+}$, a matrix $\alpha \in M_u(R)$ such that $\alpha^* \alpha = n \Id_u$,
and one wants $u$ as small as possible since $\Phi$ is an isogeny between abelian varieties of dimension $2ug$. 

As mentioned, if $n \in \Z_{>0}$ then one can 
construct $\alpha \in M_u(\Z)$ with $u = 1, 2, 4$ in the cases where $n$ is a sum of $1, 2, 4$ integer squares.
A similar technique applies if $n \in R_0^+$ is a sum of squares of elements in $R_0$: in \cite[Lem.\,4.2]{robert-note}, Robert gives solutions for $\alpha \in M_{2^k}(R_0)$ when $n$ is a sum of $m$ squares, for some $k \geq \lceil \log_2m \rceil$ where equality can be assumed if $m \leq 8$. 
E.g., for $n=\sum_{i=0}^{m-1}\theta_i^{2}$ with $\theta_i\in R_0$ the solutions for $m = 1, 2,4$  are
\begin{equation} \label{eq:zarhin_matrices}
  \alpha=(\theta_0),\qquad
  \alpha=\begin{pmatrix}\theta_0&\theta_1\\-\theta_1&\theta_0\end{pmatrix}, \qquad
  \alpha=\begin{pmatrix}
    \theta_0&\theta_1&\theta_2&-\theta_3\\
    -\theta_1&\theta_0&\theta_3&\theta_2\\
    -\theta_2&-\theta_3&\theta_0&-\theta_1\\
    \theta_3&-\theta_2&\theta_1&\theta_0
  \end{pmatrix}.
\end{equation}
There is no guarantee that $n \in R_0^+$ can always be written as a sum of squares in $R_0$, and if it can, more than $4$ squares might be needed. A theorem by Siegel \cite[Thm.\,1]{siegel} states that all totally positive integral elements of a totally real number field $K_0$ are a sum of integral squares only when $K_0 = \Q$ or $K_0 = \Q(\sqrt{5})$. 
Peters showed that if a totally positive integral element $n$ of a real quadratic field is a sum of integral squares, then it is a sum of at most $5$ integral squares, and gave an explicit criterion to detect such $n$ \cite[Sz.\,2]{peters}. 

In the following, we give solutions for $\alpha \in M_u(R)$ in the cases where $n$ is a sum of $k+1$ relative norms of elements of $R$, in which case one can always take $u = 2^k$, and in the cases where $n$ is a sum of  $t$ relative norms of elements of $R$ and $m = 2^d$ squares in $R_0$, in which case one ends up with $u = 2^{k + t}$ with $k \geq d$ as above.
These solutions are more general than what is required for our applications in Section \ref{subsec:representation}, but they are of independent interest, and do specialize to meet our needs.

\begin{proposition}\label{prop:auxiliary-alpha}
Let $X,Y\in M_m(R)$ and set
\[
  \alpha=\begin{pmatrix}X&Y^{*}\\-Y&X^{*}\end{pmatrix}\in M_{2m}(R).
\]
\begin{enumerate}
\item[(a)] If $X^{*} X = X X^{*} =x\,\Id_m$ and $Y^{*} Y= Y Y ^{*} = y\,\Id_m$ for some $x, y \in R_0$, and if $XY=YX$, then
\[
  \alpha^{*}\alpha = \alpha\alpha^{*} = (x+y)\,\Id_{2m}.
\]
\item[(b)] Conversely, for arbitrary $X,Y\in M_m(R)$ and $n\in R_0$ one has
$\alpha^{*}\alpha=n\,\Id_{2m}$ if and only if
\[
  XY=YX,\qquad X^{*}X+Y^{*}Y=XX^{*}+YY^{*}=n\,\Id_m .
\]
\item[(c)] Let $k\in \Z_{\geq 0}$ and $\theta_0, \dots, \theta_k \in R$. Define $\alpha_0:=(\theta_0)\in M_1(R)$ and, for
$1\le i\le k$,
\[
  \alpha_i:=\begin{pmatrix}\alpha_{i-1} & \bar \theta_i\,\Id_{2^{i-1}} \\[2pt]
  -\theta_i\,\Id_{2^{i-1}} & \alpha_{i-1}^{*}\end{pmatrix}\in M_{2^{i}}(R).
\]
Then $\alpha_k^{*}\alpha_k=n\,\Id_{2^{k}}$ with $n=\sum_{i=0}^{k} N_{K/K_0}(\theta_i)$. In particular, if $n$ is a sum
of $k+1$ relative norms of elements of $R$ there is a matrix $\alpha \in M_{2^k}(R)$ such that $\alpha^*\alpha = n \Id_{2^k}$.

\item[(d)] Let $d, t \in 
\Z_{\geq 0}$ and let $m = 2^d$. Consider $\theta_0, \ldots, \theta_{m-1} \in R_0$. 
Let $\alpha_0 \in M_{2^k}(R_0)$ such that $\alpha_0^* \alpha_0 = \alpha_0 \alpha_0^* = n_1 \Id_{2^k}$ with $n_1 = \sum_{i=0}^{m-1}\theta_i^2$, for $k \geq d$ as above. 
Consider $\theta_1', \ldots, \theta_t' \in R$ and let $n_2 = \sum_{i=1}^t N_{K/K_0}(\theta_i')$. 
Define, for $1 \leq i \leq t$,
\[
  \alpha_i:=\begin{pmatrix}\alpha_{i-1} & \bar \theta_i'\,\Id_{2^{k+i-1}} \\[2pt]
  -\theta_i'\,\Id_{2^{k+i-1}} & \alpha_{i-1}^{*}\end{pmatrix}\in M_{2^{k+i}}(R).
\]
Then $\alpha_t^* \alpha_t = (n_1 + n_2)\Id_{2^{k+t}}$. In particular, if $n$ is a sum of $t$ relative norms of elements of $R$ and $m = 2^d$ squares of elements in $R_0$, there exists a matrix $\alpha \in M_{2^{k+t}}(R)$ such that $\alpha^*\alpha = n \Id_{2^{k+t}}$.
\end{enumerate}
\end{proposition}

Proposition~\ref{prop:auxiliary-alpha}(d) reduces to \cite[Lem.\,4.2]{robert-note} when $n_2 = 0$, which a priori only offers a solution $\alpha \in M_{2^k}(R_0)$ but it is clear that if $n_1$ is a sum of squares in $\Z$ then also $\alpha \in M_{2^k}(\Z)$ (as is apparent from~\eqref{eq:zarhin_matrices} for $k = d = 0, 1, 2$).

\subsection{Representing the orientation} \label{subsec:representation}

\subsubsection*{Usage.} Let $(A, \lambda)$ be a $g$-dimensional PPAV together with a primitive orientation $\iota : R \to \End(A)$ by a CM order $R$ of rank $2g$. 
Say we want to act on $(A, \lambda, \iota)$ with a Shimura ideal of the form $((\ell, f_1(\gamma)), \ell)$, for a suitable $\gamma \in R$ generating $K / \Q$ and a suitable degree-$g$ polynomial $f_1(x) \in \Z[x]$; that is, a Shimura ideal of the kind described in Section~\ref{subsec:genset}. Following the proof of Theorem~\ref{thm:action}, this amounts to computing the subgroup
\[ H := A[\ell] \cap A[f_1(\iota(\gamma))] \cong \left( \frac{\Z}{\ell \Z} \right)^g, \]
which is maximal isotropic with respect to the $\ell$-Weil pairing.
Generators for $H$ can be obtained from the kernel
of a matrix of $f_1(\iota(\gamma))$ acting on $A[\ell]$.
\begin{remark}
  This can be sped up by using that $H$ equals the image of $f_2(\iota(\gamma))$ on $A[\ell]$, where $f_2(x)$ denotes the cofactor of $f_1(x)$ in the factorization of the minimal polynomial $f(x)$ of $\gamma$. This generalizes an optimization that is well-known in the elliptic curve setting, see e.g.~\cite[\S3.2]{orientexpress}.
\end{remark} 
The result of the action is then the codomain of an $\ell$-isogeny with kernel $H$, equipped with the induced principal polarization and the push-forward orientation. For this approach, it is crucial to be able to evaluate $\iota(\gamma)$ on arbitrary $\ell$-torsion points. Thus, we need to represent the orientation in a way that is both compact and allows the user to evaluate it efficiently.

\begin{remark} \label{rmk:sanon_break}
  As a side remark, observe that we can act with 
  $((\ell, f_1(\gamma)), \ell)$ without knowing what $R$ is, as long as $\gamma \in R$. One could hope to thwart Kuperberg's attack~\cite{childs_jao_soukharev,kuperberg} on the vectorization problem on $\frakC(R)$ by keeping $R$ secret and just providing $\gamma$: as mentioned at the end of Section~\ref{sec:superspecial}, extracting the primitive order $R$ from a non-primitive orientation (by $\Z[\gamma]$ in this case) might indeed be an intractable problem.
  Interestingly, this countermeasure was proposed by Sanon~\cite[pp.\,107--108]{SanonThesis}, albeit in the context of ordinary PPAVs oriented by Frobenius. But it does not provide the hoped-for security:
  determining the stabilizer $S$ of the (non-free) action by $\Z[\gamma]$ is an abelian hidden subgroup problem, which is solved in polynomial time using Kitaev's algorithm~\cite{kitaev}. One then has a free action by $\frakC(\Z[\gamma])/S$ which can be targeted using Kuperberg's algorithm.
\end{remark}

\subsubsection{Higher-dimensional representation.} Let us first assume that we can efficiently evaluate endomorphisms coming from $R_0$. Then, in order to represent $\iota(\gamma)$, the routine approach, whose elliptic curve counterpart was first described in SCALLOP-HD~\cite{scallophd}, would be as follows: write $d = \overline{\gamma} \gamma \in R_0^+$ and suppose that we can find a smooth integer $N$, ideally a power of $2$, such that $N - d$ can be effectively written as a sum of $s$ squares of elements of $R_0$.\footnote{If $N-d$ is not a sum of squares then one can retry with a different $N$ or $\gamma$; see our paragraph on representing $R_0$ at the end of this section for a workaround.} Then using the auxiliary matrix 
\[ \alpha \in M_u(R_0), \qquad u = 2^{\lceil \log_2 s \rceil} \]
from Proposition~\ref{prop:auxiliary-alpha}(d), we can embed $\iota(\gamma)$ in an $N$-isogeny between PPAVs of dimension $2ug$ as outlined in Theorem~\ref{thm:embedding}. This is an endomorphism of $A^{2u}$, in fact. This endomorphism can be computed from interpolation data for $\iota(\gamma)$, i.e., from the point pairs $P, \iota(\gamma)(P)$ with $P$ running over a basis of $A[N]$. Recall from Remark~\ref{rmk:meet-in-the-middle} 
that if $N$ is a square then using Robert's meet-in-the-middle technique it  suffices to provide interpolation data on $A[\sqrt{N}]$. When evaluating the action, this interpolation data can be pushed through: if $\varphi : A \to A / H$ is a horizontal isogeny of degree coprime to $N$, then the point pairs $\varphi(P), \varphi(\iota(\gamma)(P))$ form a valid set of interpolation points for $(\varphi_\ast \iota)(\gamma)$.\footnote{In applications, this interpolation data must be masked using a change of basis for security reasons. In practice, one changes to a deterministic basis of $A[N]$ (or of $A[\sqrt{N}]$ in the case of meet-in-the-middle) for reasons of data compression, and this handles the security issue at once.}

\subsubsection*{Halving the ambient dimension.}
In the common scenario where $\gamma = \sqrt{-d}$ and $N - d$ is a sum of at most $3$ squares of elements of $R_0$,  we can reduce the ambient dimension to $ug$. This observation seems new even for $g = 1$. To see this, let $\alpha$ be as in~\eqref{eq:zarhin_matrices}, where in the case of $3$ squares we let $\theta_2 = 0$, i.e., $N - d = \theta_0^2 + \theta_1^2 + \theta_3^2$. Define
\begin{itemize}
   \item $\xi := \sqrt{-d} + \alpha$ if $u = 1$ (i.e., $N - d$ is the square of an element of $R_0$),
   \item $\xi := \sqrt{-d} \cdot  \diag(1, -1) + \alpha$ if $u = 2$ (sum of two squares),
   \item $\xi := \sqrt{-d} \cdot \diag(1, -1, 1, -1) + \alpha$ if $u = 4$ (sum of $3$ squares and $\theta_2 = 0$).
\end{itemize}
One then checks that $\iota_A(\xi)$
is a polarized endomorphism 
of $A^u$ of reduced degree $N$. If we can evaluate this endomorphism at any point of $A^u$, then we can also evaluate $\iota(\gamma)$ at any point of $A$, assuming we know how to evaluate $\iota_A(\alpha)$. It is easy to see that the interpolation points $\varphi(P), \varphi(\iota(\sqrt{-d})(P))$ can be used to determine the kernel of $\iota(\xi)$. The method remains compatible with meet-in-the-middle: if the interpolation data is restricted to $A[\sqrt{N}]$ then this remains good enough to compute ``half the kernel", i.e., $\ker(\iota(\xi)) \cap A^u[\sqrt{N}]$. But the data can also be used to compute half the kernel of the Rosati dual, i.e., $\ker(\iota(\xi^\ast)) \cap A^u[\sqrt{N}]$, allowing one to complete the path as wanted.
Observe that the case $u = 1$ avoids the need for working in higher dimension altogether; for $g = 1$ this essentially specializes to the non-HD versions of SCALLOP~\cite{scallop,pearlscallop}.

\subsubsection{Representing $R_0$.}
We now return to the problem of representing endomorphisms coming from $R_0$.
In dimension $g = 2$ where $K_0 = \Q(\sqrt{c})$ for some square-free $c \in \Z_{>0}$, one can proceed by finding a smooth-degree representation of the endomorphism $\iota(\sqrt{c})$, which is a $c$-isogeny, using the methods from the previous paragraph but where now $\alpha \in M_u(\Z)$. In the special case where $c = p$, this is actually not needed: here the representation comes for free, see the next Section~\ref{sec:key_exch_over_Fp}.
In arbitrary dimension, the most robust solution seems to be circumventing the problem and looking for some $\gamma \in R$ generating $K / \Q$ such that $\overline{\gamma}{\gamma} \in \Z$, so that the methods from the previous paragraph can be run over $\Z$ right away. The existence of such an element $\gamma$ is not a strong condition on $R$.\footnote{Note that on PPAVs over a finite field $\F_q$ one always has $\pi_q \hat{\pi}_q = q$, so in the simple, ordinary case (where $\pi_q$ generates the endomorphism algebra) one can simply work with Frobenius.} E.g., if $R = \Z_K$ then
any element of the form $(\beta / \bar{\beta}) N_{K / \Q}(\beta)$ does the job, as long as $\beta/\bar{\beta}$ generates $K / \Q$. Moreover, by ensuring that $N_{K / \Q}(\beta)$ is smooth, one gets for free that $\overline{\gamma} \gamma = N_{K / \Q}(\beta)^2$ is smooth as well, again avoiding all need to go to higher dimension.
However, this typically leads to generators $\gamma$ having a large norm (corresponding to $\iota(\gamma)$ having a large degree), and it gives little control over $[R : \Z[\gamma]]$, meaning we potentially lose many small primes $\ell$ that could have been used to compute the Shimura class group action.

\subsection{Instantiating key exchange over $\F_p$} \label{sec:key_exch_over_Fp}

In Appendix~\ref{app:key_exch} we describe concrete key exchange protocols built from Shimura class group actions on superspecial PPASs. These are built from the R2 family of orientations: R2-D2 was described in Example~\ref{ex:R2D2}, while R2-C4 and R2-D4 will be discussed in Appendix~\ref{app:R2families}. A common feature of these orientations is that all PPASs and isogenies involved when computing the action are defined over $\F_p$ and the $p$-th power Frobenius endomorphism $\Pi_p$ is in the image of $\iota$. This choice is motivated by CSIDH~\cite{csidh}, where working with the Frobenius orientation on supersingular elliptic curves comes with three distinguishing features: 
\emph{(i)} the fact that every curve $E$ in the orbit is defined over $\F_p$ makes for a more compact object; \emph{(ii)} the orientation does not need representation, because the Frobenius endomorphism is always known explicitly; \emph{(iii)} it is easy to choose parameters ensuring the splitting of many small primes $\ell$ (as well as control the field of definition of the points in a corresponding kernel).

In higher dimension, these features remain partially present: 
\begin{itemize} 
\item One can no longer avoid the need for representing the orientation altogether, but at least the Frobenius part of the orientation still comes for free. E.g., in the R2-D4 and R2-C4 orientations described in Appendix~\ref{app:R2families}, where $K_0 = \Q(\sqrt{p})$ and $\iota(\sqrt{p}) = \Pi_p$, this removes the need for representing an endomorphism of the form $\iota(\sqrt{c})$, $c \in \Z_{>0}$, overall making for a more compact orientation representation. Moreover, the $\F_p$-rationality can be exploited to compress this data further: this is explained in Appendix~\ref{app:compress}.
\item In all R2 orientations we have $\Q(\sqrt{\pm p}) \subset K$, so by choosing $\pm p$ to be a quadratic residue mod $\ell_i$ for small primes $\ell_1, \ldots, \ell_r$ 
we roughly double the chance that $\ell_i$ splits completely
(from $\frac{1}{4}$ to $\frac{1}{2}$ in cases $C_2 \times C_2, C_4$ and $\frac{1}{8}$ to $\frac{1}{4}$ in case $D_4$). Moreover, parameter generation is flexible enough to boost those chances further; e.g., see our approach to R2-D4 taken in Appendix~\ref{app:R2D4_key_exch}.
\end{itemize}

\begin{remark} \label{rmk:why_to_avoid_Fp} There may be reasons to stay away from Frobenius orientations, e.g., for more control over $\#\frakC(R)$ or more flexibility in choosing parameters, as in the SCALLOP family~\cite{scallop,scallophd,pearlscallop} (a general technique for constructing such orientations in dimension $g = 2$ will be outlined in Section~\ref{sec:constructing_orientations}).
In the $C_2 \times C_2$ case, there is another very good reason for avoiding Frobenius orientations: based on a preliminary version of the current paper, the authors of~\cite{C2C2attack} observed that the vectorization problem for R2-D2, which we recall is an orientation by a CM order of the form $\Z[\sqrt{d}, \sqrt{-p}]$, reduces to vectorization problems in $\Cl(\Z[\sqrt{-p}])$ and $\Cl(\Z[\sqrt{-dp}])$.
This relies on a  twisting trick akin to~\cite{rational_from_irrational,wesolowski_oriented} and, as far as we are aware,
does not obviously generalize to arbitrary $C_2 \times C_2$ orientations, because it depends on a property that is specific to cases where $p$ ramifies in $K$, namely that twists coming from ``inequivalent CM types" can nevertheless end up in the same orbit; see Section~\ref{sec:orbits}, and Example~\ref{ex:R2D2orbits} in particular. For the sake of reference, we have still included (a shortened version of) our R2-D2-based key exchange protocol in  Appendix~\ref{app:R2D2_keyexchange}.
\end{remark}

\section{CM orientations under the IKO correspondence}\label{sec:algebra}

The IKO framework, introduced in~\cite{ibukiyama1986supersingular},
gives a convenient Deuring-like dictionary between the geometry of superspecial PPASs over $\overline{\F}_p$ and their connecting isogenies on one hand, and the algebra of positive definite invertible matrices in $M_2(B_{p, \infty})$ and their ``connecting matrices" on the other hand.
The goal of this section 
is to incorporate CM orientations and Shimura class group actions in this framework and provide  applications. Omitted proofs can be found in Appendix~\ref{app:algebraproofs}.
    Note: in this section the letter $g$ is used to denote an IKO matrix, rather than the dimension of an abelian variety, which is assumed $2$ throughout.

\subsection{The IKO correspondence: a slight extension} \label{subsec:IKO_intro_and_extend}
Fix a supersingular elliptic curve $E$ in characteristic $p > 0$ along with an isomorphism $\theta : \O \to \End(E)$ from some maximal order $\O \subset B_{p, \infty}$, as at the beginning of Section~\ref{sec:superspecial}. By Deligne's theorem~\cite{goldstein_schacher} any superspecial PPAS is isomorphic to $(A, \lambda)$ for some principal polarization $\lambda: A \to A^\vee$ on $A = E^2$. Let $\lambda_0$ denote the product polarization. Then the automorphism $\lambda_0^{-1} \circ \lambda$ corresponds under $\vartheta$ (when applied component-wise) to some totally positive Hermitian norm-$1$ matrix, i.e., 
\[ g = \begin{pmatrix} s & r \\ \bar{r} & t \end{pmatrix}, \qquad s, t \in \Z_{>0}, \ r \in \O, \ st - r \bar r = 1; \]
see also Example~\ref{ex:IKO}. We call $g$ the  \defi{$g$-matrix} or IKO-matrix associated with $(A, \lambda)$.
To identify equivalent principal polarizations $\lambda, \lambda'$, i.e., for which $(A, \lambda) \cong (A, \lambda')$, one should work with $g$-matrices up to equivalence, where we say $g' \sim g$ if and only if $g'=u^*gu$ for some $u\in \GL_2(\mathcal{O})$. 

Since we are just considering $A$ with different principal polarizations, we always have $\End(A) \cong M_2(\mathcal{O})$ and the piece of data that varies with $\lambda$ is the Rosati involution. 
We know that the Rosati involution for the product polarization $\lambda_0$ is just conjugate-transpose $b \mapsto b^\ast$. In general, for a superspecial PPAS corresponding to a matrix $g$, the Rosati involution is $b\mapsto b^{\dagger} = g^{-1}b^*g$ which follows from the definition of the $g$-matrix and the definition of the Rosati involution. In that respect, superspecial surfaces can also be characterized by positive involutions on $M_2(\mathcal{O})$.

Isogenies  are incorporated in the IKO correspondence as follows. Let $\lambda_1, \lambda_2$ be principal polarizations on $A$ with $g$-matrices $g_1, g_2$ and let $\mu \in M_2(\O)$ be a matrix fixed by the Rosati involution with respect to $g_1$ that is totally positive, i.e., the (necessarily real) roots of its minimal polynomial over $\Q$ are positive.
Then $\gamma \in M_2(\O)$ corresponds (under $\theta$) to a $\theta(\mu)$-isogeny $(A, \lambda_1) \to (A, \lambda_2)$ if and only if
\begin{align*}
    \lambda_1 \circ \theta(\mu) = \widehat{\theta(\gamma)} \circ \lambda_2 \circ \theta(\gamma) & \Leftrightarrow \lambda_0^{-1} \circ \lambda_1 \circ \theta(\mu) = \lambda_0^{-1} \circ \widehat{\theta(\gamma)} \circ \lambda_0 \circ \lambda_0^{-1} \circ \lambda_2 \circ \theta(\gamma)  \\
    & \Leftrightarrow g_1 \mu = \gamma^\ast g_2 \gamma.
\end{align*} 
This (slightly) extends the corresponding calculation from~\cite[\S2.3]{DBLP:conf/crypto/CastryckDKLPT25}, which was restricted to $\mu = N \Id_2$ for positive integers $N$.

\begin{remark} \label{rmk:secondcomponentmatters}
    For $\gamma = \Id_2$ we have the following identity as a special case: $g_1 \Id_2 = \Id_2^\ast g_1 \Id_2$.
    An interesting aspect of this trivial equation is that it says both that $(A, \lambda_1)$ and $(A, \lambda_1)$ are $1$-isogenous (trivially), and that $(A, \lambda_0)$ and $(A, \lambda_1)$ are $\vartheta(g_1)$-isogenous.
    Since this reasoning applies to arbitrary $g_1$, this shows the great impact of the choice of $\mu$ (e.g., when specified as the second component of a Shimura ideal) on the resulting codomain as a PPAV.
\end{remark}

So this means that if $\lambda_1$ is a principal polarization on $A$ with $g$-matrix $g_1$ and if $H \subset A$ is a finite subgroup that is maximal with the property that
\[ e_{\vartheta(\mu)}(P, \lambda_1(Q)) = 1, \quad \forall P, Q \in H, \]
then a $g$-matrix for the induced principal polarization on $A \cong A / H$
is 
\[ g_2 := \gamma^{-\ast} g_1 \mu \gamma^{-1}, \]
with $\gamma \in M_2(\O)$ any matrix such that $\ker \theta(\gamma) = H$. This follows, e.g., from Proposition~\ref{prop:max_iso}. In particular, $g_2$ is necessarily a valid $g$-matrix, i.e., a positive-definite Hermitian matrix of norm $1$ contained in $M_2(\O)$.

\begin{remark} \label{rmk:verifygmatrix}
Most of these properties are easily confirmed:
\begin{itemize}
  \item Using $\mu^\dagger = g_1^{-1} \mu^\ast g_1$ one checks that $g_2^\ast = g_2$.
  \item By definition of a $\vartheta(\mu)$-isogeny, the degree of $\vartheta(\mu)$ equals the squared degree of $\vartheta(\gamma)$, hence in view of~\cite[Lem.\,2.10]{PIP} we see that $g_2$ has norm $1$.
  \item It concerns a positive-definite matrix: for this it suffices to see that $\gamma^\ast g_2 \gamma = g_1 \mu$ is positive-definite, and this in turn follows from the total positivity of $\mu$ and the fact that after base-changing to $B_{p, \infty} \otimes_{\Q} \C \cong M_2(\C)$ if needed one can rewrite $g_1 \mu = h(h^\ast \mu h) h^{-1}$ with $h h^\ast = g_1$. The existence of such an $h$ follows from Cholesky decomposition, which we have included as Lemma~\ref{lem:quaternion_cholesky} below.
\end{itemize} 
Verifying that $g_2 \in M_2(\O)$ in an elementary way is more intricate. For instances of $H$ that arise as the kernel of a Shimura ideal, such a verification can be found in Appendix~\ref{app:direct_action}; see also Remark~\ref{rmk:direct_action}.
\end{remark}

\begin{lemma}[Cholesky decomposition in $M_2(B_{p, \infty})$]
\label{lem:quaternion_cholesky}
Let $g \in M_2(B_{p, \infty})$ be a positive-definite Hermitian matrix, i.e., $g = g^\ast$ and its (necessarily real) eigenvalues are positive. Then there exists $h \in M_2(B_{p, \infty})$ such that $h h^\ast = g$.
\end{lemma}

\subsection{Orientations of $g$-matrices and the Shimura class group action} \label{sec:actingIKO}

Let $K$ be a quartic CM field and let $R \subset K$ be a CM order. We now describe $K$-oriented superspecial PPASs $(A, \lambda, \iota)$ within the framework of the IKO correspondence, where as above we assume $A = E^2$. Let $g$ be the $g$-matrix associated with $\lambda$ and let $\iota : R \hookrightarrow \End(A) \cong M_2(\O)$ be a primitive orientation.
Write $R_0 = \Z[a]$ for some real generator $a$ and $R = R_0[b_1, b_2]$ for generators $b_i \in R$.\footnote{In many cases one can do with a single generator $b_i$: this is true if and only if $R$ is a free $R_0$-module. E.g., for this it suffices that $R_0$ is maximal and $h(K_0)=1$.}
Upon identifying $a, b_1, b_2$ with their images under the orientation, we can think of them as commuting elements of $M_2(\mathcal{O})$ satisfying $g^{-1}a^*g=a$ and $g^{-1}b_i^*g \in R$ for $i = 1, 2$.
We refer to the tuple $g, a, b_1, b_2 \in M_2(\O)$ as \defi{algebraic data} associated with $(A, \lambda, \iota)$.

\begin{theorem}\label{thm:algN}
Assuming the generalized Riemann hypothesis (GRH), there is a polynomial-time algorithm which upon input of
\begin{itemize}
    \item algebraic data $g,a,b_1, b_2$ of an oriented superspecial PPAS $(A, \lambda, \iota)$,
    \item an integral Shimura ideal $(\fraka,\mu)$ of $\Z[a, b_1, b_2]$ such that $p \nmid N(\fraka)$,\footnote{We included the assumption $p \nmid N(\fraka)$ for consistency with Theorem~\ref{thm:action}. But, in fact, our algebraic treatment works equally well in the inseparable case.}
\end{itemize}
returns algebraic data of the codomain of the corresponding $\iota(\mu)$-isogeny, equipped with the induced principal polarization and orientation.
\end{theorem}

\begin{proof}
The endomorphisms that vanish on the kernel of $\fraka$ form a left ideal of $M_2(\mathcal{O})$ generated by the elements of $\fraka$. A generator $\gamma$ of this (principal) ideal can be found in polynomial time, using either~\cite{PIP} or~\cite[\S4]{page2026conversion}; the former algorithm is more efficient but the latter method is heuristic-free (modulo a dependency on GRH). As explained in the previous section, 
the $g$-matrix of the codomain is then $g' = \gamma^{-*}g\mu\gamma^{-1}$. 
From the requirement that $\gamma$ should represent an oriented isogeny, it is then clear that the push-forward orientation is
given by
$a' = \gamma a\gamma^{-1}$, $b_i' = \gamma b_i\gamma^{-1}$. All this data is straightforward to compute. \qed
\end{proof}

\begin{remark} \label{rmk:direct_action}
It is instructive to try and \emph{define} the Shimura class group action (on superspecial PPASs) via $g \mapsto \gamma^{-*}g\mu\gamma^{-1}$, $a \mapsto \gamma a \gamma^{-1}$, $b_i \mapsto \gamma b_i \gamma^{-1}$ and then give a direct matrix-arithmetic proof that this action is well-defined. The main tasks here are showing that $g', a', b_i' \in M_2(\O)$, that $a'$ is fixed under the Rosati involution induced by $g'$, and that $\Z[a', b'_1, b'_2]$ is stable under this Rosati involution. The details of such a direct approach can be found in Appendix~\ref{app:direct_action}. 
\end{remark}

All explicit formulae for isogeny evaluation (that we are aware of) restrict to $N$-isogenies for some $N \in \Z_{>0}$. In view of this, we have the following important consequence of Theorem~\ref{thm:algN}: if a $g$-matrix of $(A, \lambda)$ is known, then we can act with arbitrary Shimura ideal classes, even if they do not belong to $\frakC(R)_{\Z}$, in polynomial time. (Note: here we assume that $(A, \lambda)$ is given to us as a Jacobian with the canonical polarization, or as a product of elliptic curves with the product polarization.) We point out that knowing a $g$-matrix of $(A, \lambda)$ is polynomial-time equivalent 
to knowing the endomorphism ring of $A$, in the strong sense of~\cite[Def.\,2.14]{montessinos2026equivalent}; see~\cite[\S3]{montessinos2026equivalent} for details.

\begin{corollary}\label{cor:evaluate}
Let $R$ be a CM order in a quartic CM field and let $p$ be a prime number. There exists a heuristic polynomial-time algorithm for computing $[(\fraka, \mu)] \star (A, \lambda, \iota)$ for all $[(\fraka,\mu)] \in \frakC(R)$ and $(A, \lambda, \iota) \in \M_{\overline{\F}_p}(R)$, assuming that a $g$-matrix of $(A, \lambda)$ is known.
\end{corollary}
\begin{proof}
We assume that the orientation $\iota$ is given in higher-dimensional representation as in Section~\ref{subsec:representation}, but the argument should carry over to any other valid way of representing $\iota$. 
From a $g$-matrix of $(A, \lambda)$ we can compute an explicit isomorphism $\varphi : (A, \lambda) \to (E^2, \lambda_0 \circ g)$ using~\cite[Thm.\,5.3]{montessinos2026equivalent}. 
This is then also an oriented isomorphism between $(A, \lambda, \iota)$ and $(E^2, \lambda_0 \circ g, \varphi_\ast \iota)$,
and the given orientation data are easily pushed through and interpolated to obtain algebraic data $g, a, b_1, b_2$. Algebraic data $g', a', b_1', b_2'$ for $[(\fraka, \mu)] \star (A, \lambda, \iota)$ is then obtained using Theorem~\ref{thm:algN}.
Using the 
IKO-to-surface conversion algorithm from~\cite[Thm.\,5.1]{DBLP:conf/crypto/CastryckDKLPT25} we can realize this codomain as a Jacobian (or a product of elliptic curves with the product polarization), where an explicit isomorphism $\varphi'$ can be found through another use of~\cite[Thm.\,5.3]{montessinos2026equivalent}.
The algebraic data $a', b_1', b_2'$ can then be converted into interpolation data and pushed through $\varphi'$ to obtain the desired output.
 \qed 
\end{proof}
\begin{remark}
Corollary \ref{cor:evaluate} is heuristic due to several calls to the powersmooth $\text{KLPT}^2$ algorithm from \cite{DBLP:conf/crypto/CastryckDKLPT25}, which relies on some heuristics. Note that the improved $\text{KLPT}^2$ algorithm from \cite{page2026conversion} can likely be made heuristic-free (perhaps under GRH; this would work with very similar ideas as in the case of 1-dimensional KLPT~\cite{weso22}). 
\end{remark}

\subsection{Constructing $g$-matrices with a given orientation} \label{sec:constructing_orientations}
An important application of the algebraic viewpoint is the construction of superspecial PPASs with a prescribed orientation. Indeed, on the geometry side there is no systematic way to find orientations other than ad-hoc constructions, as in Section~\ref{sec:superspecial} or Appendix~\ref{app:R2families}, or CM reductions, which only provide orientations by orders with a relatively small discriminant. 

As an entry point, we consider the next lemma: 
\begin{lemma}\label{lem: cent}
Let $g$ be a $g$-matrix and let $a\in M_2(\mathcal{O}) \setminus \Q \Id_2$ be a symmetric element with respect to Rosati. Then its centralizer $\mathrm{Cent}(a) \subset M_2(B_{p,\infty})$ is isomorphic to the quaternion algebra $C = \mathbb{Q}(a)\otimes B_{p,\infty}$ over $\Q(a)$, and  Rosati acts as quaternion conjugation on C. 
\end{lemma}

\noindent  This shows that one can obtain orientations as follows. Take a symmetric element $a\in M_2(\mathcal{O})$ and consider the intersection $M_2(\mathcal{O})\cap \mathrm{Cent}(a)$. This intersection is invariant under Rosati as both $M_2(\mathcal{O})$ and $\mathrm{Cent}(a)$ are. Then taking quadratic subrings of $M_2(\mathcal{O})\cap \mathrm{Cent}(a)$ will yield orientations. 
This is reminiscent of the 1-dimensional case where orientations are simply quadratic subrings of maximal orders. We note that intersections of the form $M_2(\O) \cap \mathrm{Cent}(a)$ were intensively studied in \cite{nicole2005superspecial} where it was shown that they are either maximal orders or Eichler orders of level $p$ in the quaternion algebra $\mathbb{Q}(a)\otimes B_{p,\infty}$ as soon as $p$ is unramified in $\mathbb{Q}(a)$ \cite[Thm.\,2.5.27]{nicole2005superspecial}. 

Our goal is an algorithm for creating superspecial PPASs that are primitively oriented by a prescribed CM order $R$.
For convenience we assume $R = \Z_K$. We expect that the method below can be tweaked so that it covers a much more general class of CM orders, but this would add a layer of technicality that lies beyond our current scope.
Note that primitivity comes for free if $R = \Z_K$, so we are merely concerned with finding a $K$-oriented superspecial PPAS $(A, \lambda, \iota)$ such that $\Z_K \subset \iota^{-1} \End(A)$.

Our approach is to generalize the 1-dimensional blueprint, where
one takes an element in $B_{p,\infty}$ of prescribed minimal polynomial and then puts that element into a maximal order (after which one finds a corresponding elliptic curve using the effective Deuring correspondence). Here the situation is somewhat more complex as one has to embed an element of $M_2(B_{p, \infty})$ into a maximal order of $M_2(B_{p, \infty})$ in a way that is compatible with the Rosati involution; this compatibility is automatic in the $1$-dimensional case. 
First we prove two useful lemmas. The first lemma can be thought of as a base-point-free orientation framework. Indeed, it is sometimes inconvenient (especially when trying to generalize 1-dimensional proofs) to work with $g$-matrices, and rather use the following equivalent framework. 
\begin{lemma}\label{lemma:bpf}
Let $\Lambda$ be a maximal order in $M_2(B_{p,\infty})$ that is stable under con- jugate-transpose. Then there exists an isomorphism of rings with involutions between $\Lambda$
with conjugate-transpose 
and $M_2(\mathcal{O})$ with some Rosati involution induced by a $g$-matrix. This isomorphism can be computed in polynomial time.  
\end{lemma}
\begin{proof}
First, we find $\gamma$ such that $M_2(\mathcal{O})=\gamma^{-1}\Lambda \gamma$. Such a $\gamma$ can be found in polynomial time using \cite{page2026conversion} or \cite{PIP} as it is equivalent to a principal ideal problem. One can assume that the reduced norm of $\gamma$ is coprime to $p$.
As $M_2(\mathcal{O})$ is closed under conjugate-transpose one has that $(\gamma^{-1}\Lambda\gamma)^*=\gamma^{-1}\Lambda\gamma$. This can be rearranged as $\gamma\gamma^*\Lambda=\Lambda\gamma\gamma^*$ \ which means that $\gamma\gamma^*$ generates a two-sided ideal of $\Lambda$, hence $(\gamma\gamma^*)=(N)$ for some integer $N$. This implies that $\gamma\gamma^*/N$ has norm $1$ and hence the same holds for $\gamma^* \gamma / N$. Call this latter matrix $g$ and note that it indeed concerns a $g$-matrix. Then it is easy to see that $(\Lambda,*)$ is isomorphic as a ring with involution to  $(M_2(\mathcal{O}),\dagger_g)$ where $\dagger_g$ denotes the Rosati involution induced by $g$. 
\qed
\end{proof}

\noindent The second lemma is a very generic fact from the algorithmic theory of quaternion algebras (and might already be in the literature but we could not find a reference). Let $C$ be a quaternion algebra over a field $K$, not of characteristic 2. It is a well-known result that $C$ is split by a quadratic extension $K(\sqrt{d})$,
i.e., $C\otimes_K K(\sqrt{d}) \cong M_2(K(\sqrt{d}))$,
if and only if $C$ contains $K(\sqrt{d})$ as a subfield. The next lemma shows that if one can find an explicit isomorphism between $C\otimes_K K(\sqrt{d})$ and $M_2(K(\sqrt{d}))$, then one can find a subfield of $C$ that is isomorphic to $K(\sqrt{d})$: 
\begin{lemma}\label{lemma:split}
Let $K$ be a number field. Let $C$ be a quaternion algebra over $K$ given by a quaternion basis over $K$ and let $d \in K \setminus K^2$. Then one can decide whether $K(\sqrt{d})$ embeds into $C$ in quantum polynomial time. Furthermore, if it embeds, one can find an explicit embedding in quantum polynomial time. 
\end{lemma}

\begin{remark} \label{rmk:split_heuristic_free}
Let us emphasize that the quantum part is not necessary in practice, but it makes the lemma heuristic-free. Indeed, our proof of Lemma \ref{lemma:split} 
in Appendix~\ref{app:omitted_algebra_proofs} invokes a quantum algorithm from~\cite{kutas_montessinos} but this can be replaced
with a heuristic algorithm for solving quadratic forms.
\end{remark}

 We now proceed to finding $\Z_K$-oriented superspecial PPASs.  The following is a first step:

\begin{lemma}\label{lem:orient_fields}
Given a quartic CM field $K$, 
there exists a quantum polynomial-time algorithm that finds an embedding $\iota : K \hookrightarrow M_2(B_{p, \infty})$ such that $\iota(\bar \alpha) = \iota(\alpha)^\ast$ for all $\alpha \in K$, or decides that no such embedding exists.
\end{lemma}

\begin{proof}
We first
find an element $a\in M_2(B_{p,\infty})$ such that $a$ is symmetric with respect to conjugate-transpose and $\Q(a) \cong K_0$. 
We will essentially show how one can generate an element of given trace and norm for symmetric matrices. A symmetric matrix has the following form: 
$$
\begin{pmatrix}
 s&r \\
 \overline{r} &t
\end{pmatrix}
$$
where $s,t \in \Q$ and $r\in B_{p,\infty}$. Let the trace of $a$ be $M$ and its norm be $N$. This gives us two equations: 
$s+t=M, ~st- \nrm(r)=N$. 
Such $s,t,r$ can be found by reducing to a quadratic form in 6 variables that always admits a solution and can be solved with Castel's algorithm \cite{castel2013solving} in polynomial time (this is a variation of an algorithm due to Simon that does not require factoring).
Now we find the CM field $K$ inside $\mathrm{Cent}(a)$, or decide that it cannot be found, using Lemma \ref{lemma:split}
with $C = \Cent(a)$ and $K = \Q(a)$. \qed
\end{proof}

\begin{remark}[continuation of Remark~\ref{rmk:split_heuristic_free}] Here too the word ``quantum" can be omitted at the cost of introducing some heuristics.    
\end{remark}

Our main technical tool for converting this into an algorithm for finding $\Z_K$-orientations is the following:

\begin{lemma} \label{lem:conductor_lemma}
  Let $K \subset M_2(B_{p, \infty})$ be a quartic CM field such that complex conjugation in $K$ amounts to conjugate-transpose in $M_2(B_{p, \infty})$ and assume that the conductor ideal in $\Z_K$ of the order $\Z_K \cap M_2(\O)$
  is principal and generated by a totally positive $\mu \in \Z_{K_0}$ that is coprime to $p$. Let $\xi \in M_2(B_{p, \infty})$ be such that $\xi \xi^\ast = \mu$. Let $\gamma \in M_2(B_{p, \infty})$ be such that $M_2(\O) \gamma$ equals the left fractional ideal $M_2(\O) \Z_K$ generated by $\Z_K$. Let $\Lambda = (\gamma \xi)^{-1} M_2(\O) \gamma \xi$. Then $\Lambda^\ast = \Lambda$ and
  \[ \iota : \Z_K \to \Lambda : \alpha \mapsto \xi^{-1} \alpha \xi \]
  is an injective ring homomorphism such that $\iota(\bar \alpha) = \iota(\alpha)^\ast$ for all $\alpha \in \Z_K$.
\end{lemma}

\begin{corollary} \label{cor:orientation_finding}
    If we assume that $h^+(K_0)$ is polynomially bounded,\footnote{This is typically a mild assumption. By Lemma~\ref{lem:narrow_cl_number} we have $h^+(K_0) / h(K_0) \in \{1, 2, 4\}$. The Cohen--Lenstra heuristics~\cite{cohen_lenstra} predict that the odd part of $h(K_0)$ is extremely small, e.g., it is just $1$ for about $76\%$ of all real quadratic number fields. By an extension of the Cohen--Lenstra heuristics (see~\cite[\S1]{koymans_pagano} for a discussion) the even part of $h_0(K_0)$ comes almost entirely from $\Cl(K_0)[2]$, which by genus theory is roughly of size $2^t$ with $t$ the number of prime factors of $\disc K_0$. In view of the Erd\"os--Kac theorem this number is typically of size $\log \disc K_0$.} there exists
    a heuristic polynomial-time algorithm which upon input of a quartic CM field $K$ returns a primitively $\Z_K$-oriented 
    superspecial PPAS in characteristic $p$, or returns that
    no such surface exists. 
\end{corollary}

\begin{proof}
    Repeatedly run Lemma~\ref{lem:orient_fields} until the conductor ideal $\frakc$ of $\Z_K \cap M_2(\O)$ is generated (as an ideal in $\Z_K$) by some totally positive $\mu \in \Z_{K_0}$ whose norm is not a multiple of $p$. We expect this to happen with good probability. Indeed, since $\bar \frakc = \frakc$ we expect
    that $\frakc = \frakc_0 \Z_K$ for some ideal $\frakc_0 \subset \Z_{K_0}$; the only exception here is where $\frakc$ has a factor $\frakp^e$ with $e$ odd and $\frakp$ a prime ideal that is ramified over $K_0$. 
    We then expect $\frakc_0$ to be generated by some totally positive $\mu$ with probability about $1 / h^+(K_0)$. Using the previous technical lemma we can find an embedding of $\Z_K$ inside a self-conjugate maximal order $\Lambda \subset M_2(B_{p, \infty})$ in a way that is compatible with Rosati. Note that all ingredients can be computed efficiently, including the Cholesky decomposition from Lemma~\ref{lem:quaternion_cholesky} which involves solving quadratic equations over $\Q$ in $5$ variables, which can be done efficiently using Castel's algorithm~\cite{castel2013solving}.
    Now invoking Lemma \ref{lemma:bpf} we get an orientation for some $g$-matrix in $M_2(\O)$. This can be converted into a $K$-oriented superspecial PPAS  as in the proof of Corollary~\ref{cor:evaluate}. \qed
\end{proof}

\subsection{Applications} \label{sec:algebra_appl}

In Appendix~\ref{app:oneend} we discuss applications of the previously described algebraic machinery to the ``one endomorphism problem". In dimension 1 it is known that knowing one non-scalar endomorphism provides significant extra information: 
\begin{itemize}
    \item Given one non-scalar endomorphism one can compute the entire endomorphism ring in sub-exponential time \cite{DBLP:journals/cic/MerdyW25}.
    \item Given an endomorphism of small degree (polylogarithmic in  $p$), one can compute the full endomorphism ring in polynomial time \cite{love2020supersingular}.
    \item Given an oracle for computing non-scalar endomorphisms, one can compute the endomorphism ring in polynomial time~\cite{oneend}.
\end{itemize}
We investigate the first two problems in dimension 2, where we observe that adjustments have to be made, as there are non-scalar (even polarized) endomorphisms that will likely not provide enough information for computing the full endomorphism ring: an example is \[ \begin{pmatrix}
    1& 1\\ 1&-1
\end{pmatrix}\]  on  $E^2$ for any supersingular elliptic curve $E$.  The proper analogue of a non-scalar endomorphism in dimension 2 is one whose minimal polynomial defines a quartic CM field over $\mathbb{Q}$ (a ``generic" endomorphism satisfies this). 

\section{Orbits of the Shimura class group action} \label{sec:orbits}

It is classical that if $g=1$ and $p = 0$ then the action $\star : \Cl(R) \times \M_k(R) \to \M_k(R)$ is ``essentially" transitive, although our algebraic treatment of orientations comes with the subtlety, emphasized by Onuki~\cite{onuki}, that there are in fact two orbits. One can jump back and forth between these orbits via $(E,  \iota) \mapsto (E, \iota \circ \sigma)$, with $\sigma : \alpha \mapsto \overline{\alpha}$ denoting complex conjugation on $R$. 
In positive characteristic, on the other hand, we know that if $p > 0$ and $p \mid [\Z_K : R]$ then $\M_k(R)$ is empty, while if $p \nmid [\Z_K : R]$ then:
\begin{itemize}
\item[\emph{(i)}] If $p$ splits in $K$ then all curves in $\M_k(R)$ are ordinary and one can jump between the two orbits by quadratic twisting. 
\item[\emph{(ii)}] If $p$ ramifies in $K$ then all curves in $\M_k(R)$ are supersingular; here the action becomes transitive.
\item[\emph{(iii)}] If $p$ is inert in $K$ then all curves in $\M_k(R)$ are supersingular, but in this case there are again two orbits. Instead of $(E,  \iota) \mapsto (E, \iota \circ \sigma)$, there is a neat alternative way to jump between the orbits due to Onuki: it can be assumed that all curves are defined over $\F_{p^2}$, and then one can jump back and forth using the $p$-power Frobenius map $(E,  \iota) \mapsto (E^{(p)}, \iota^{(p)})$.
\end{itemize}
This can be assembled from~\cite[\S13, Thm.\,5\,\&\,Thm.\,12]{lang}
and~\cite[Prop.\,3.3]{onuki}.

In higher dimension this picture is much more subtle.\footnote{E.g., from Remark~\ref{rmk:nonempty_orbit_with_conductor_p} we see that $\M_k(R)$ may be non-empty even when $p \mid [\Z_K : R]$.} The literature does not seem to contain the ingredients needed for a complete answer, which in any case lies beyond the scope of our paper.
But we can sketch the overall picture, highlighting the main ingredient: the \defi{CM type}. Behind the scenes, this notion is already responsible for the (typically) two orbits in the elliptic curve case. Much of what follows is standard to specialists:
our main goal is to illustrate the theory with examples that are relevant to the paper at hand.

\subsection{An obstruction to transitivity}
Our starting point is
the ring homomorphism
$\rho_{(A, \lambda, \iota)} : R \to \End_k(D(A)) : \alpha \mapsto \iota(\alpha)^\ast$, where $\iota(\alpha)^\ast$ denotes the pull-back map under $\iota(\alpha)$ on the $g$-dimensional $k$-vector space $D(A) = H^0(A, \Omega^1_A)$ of regular $1$-forms on $A$.
This homomorphism gives $D(A)$ the structure of an $R$-module, in the sense of representation theory.

\begin{lemma} \label{lem:CMtypeobstruction}
  If $(A, \lambda, \iota), (A', \lambda', \iota') \in \M_k(R)$ lie in the same orbit under the action of $\frakC(R)$ then $D(A) \cong D(A')$ as $R$-modules.
\end{lemma}

\begin{proof}
  Every Shimura ideal class has a representative $(\fraka, \mu)$ such that $N(\fraka)$ is not a multiple of $p$.  Then the isogeny $\varphi_{\fraka} : A \to A'$ is separable, so the map $\varphi_{\fraka}^\ast : D(A') \to D(A)$ is an isomorphism of $k$-vector spaces. The diagram
\begin{center}
\begin{tikzcd}
D(A') \arrow[r, "\iota'(\alpha)^\ast"] \arrow[d, "\varphi_{\fraka}^\ast"'] 
& D(A') \arrow[d, "\varphi_{\fraka}^\ast"] \\
D(A) \arrow[r, "\iota(\alpha)^\ast"'] 
& D(A)
\end{tikzcd}
\end{center}
commutes for every $\alpha \in R$, because $\varphi_{\fraka}$ is oriented. This indeed amounts to saying that $\varphi^\ast_{\fraka}$ is an isomorphism of $R$-modules. \hfill $\qed$ 
\end{proof}

\begin{remark}
  The lemma generalizes to non-horizontal isogenies: for any two CM orders $R, R' \subset \Z_K$ such that $p \nmid [\Z_K : R][\Z_K : R']$ it holds that if $(A, \lambda, \iota) \in \M_k(R)$, $(A', \lambda', \iota') \in \M_k(R')$ are connected by a separable oriented isogeny, then $D(A), D(A')$ are isomorphic as $\Z_K$-modules (the assumption $p \nmid [\Z_K : R]$ naturally allows to extend the domain of $\rho_{(A, \lambda, \iota)}$ to $\Z_K$, and likewise for $R'$).
\end{remark}

\begin{example} \label{ex:decidingorbits}
  We start by shedding a different light on Onuki's motivating example from~\cite[p4]{onuki}. Let $R = \Z[i]$ be the ring of Gaussian integers, assume $p \neq 2$ and let $\i \in k$ be a square root of $-1$. Consider $E : y^2 = x^3 + x$ over $k$ along with the primitive $R$-orientations
  \[ \iota : i \mapsto \left( (x, y) \mapsto (-x, \i y) \right), \qquad \iota' : i \mapsto \left( (x, y) \mapsto (-x, -\i y) \right) \]
  on $E$. The space $D(E)$ is generated by $dx / y$, which under $\iota(i)$ pulls back to $d(-x)/\i y = (-1/\i)dx/y = \i dx/y$ while under $\iota'(i)$ it pulls back to $-\i dx / y$. Thus the representations map $i$ to the multiplication maps by $\i$ and $-\i$, respectively. Therefore $(E, \iota)$, $(E, \iota')$ must lie in different orbits. In general, this gives a very efficient method for testing whether, in case \emph{(iii)} above, two given oriented supersingular elliptic curves lie in the same orbit or not. In other words, the decisional version of the ``effective $R$-vectorization problem" from~\cite[Prob.\,2]{wesolowski_oriented} is very easy to solve.
\end{example}

\begin{example} \label{ex:example_orientation_revisited}
    Let $E_1, E_2$ be elliptic curves over $k$ and let $A = E_1 \times E_2$. Then $D(A) \cong D(E_1) \oplus D(E_2)$. So, upon picking a generator $\omega_i$ of $D(E_i)$ for $i = 1, 2$ we can view $(\omega_1, 0), (0, \omega_2)$ as a basis for $D(A)$. The elements of $\End(A)$ can be viewed as matrices 
    \[ \begin{pmatrix} \alpha_{11} & \alpha_{21} \\ \alpha_{12} & \alpha_{22} \end{pmatrix} , \qquad \alpha_{ij} \in \Hom(E_i, E_j). \]
    With respect to our basis of $D(A)$, the matrix of the pull-back operator along such a matrix $(\alpha_{ji})_{i,j}$ is given by $(c_{ij})_{i,j}$, where
    $c_{ij} \in k$ is the constant such that 
    $\alpha_{ij}^\ast \omega_j = c_{ij} \omega_i$.
    Let us have a look at what this implies for the two orientations $\iota, \iota'$ from Examples~\ref{ex:example_orientation} and~\ref{ex:secondorientation}. 
    
    Using that differentials pull back to zero under inseparable isogenies, i.e., $\pi_p^\ast \omega_i = 0$ for $i = 1, 2$, we find that $\rho_{(A, \lambda, \iota)}$ maps
    \[ \frac{1 + \sqrt{-m}}{2} \mapsto \begin{pmatrix} \frac{1}{2} & \i \\ \i & \frac{1}{2} \end{pmatrix}, \qquad \sqrt{d} \mapsto \begin{pmatrix} 0 & \frac{1}{2} \\ \frac{1}{2} & 0 \end{pmatrix}. \]
    When rewritten with respect to the basis $(\omega_1, \omega_2), (\omega_1, -\omega_2)$ this becomes
    \[ \frac{1 + \sqrt{-m}}{2} \mapsto \begin{pmatrix} \frac{1}{2} + \i & 0 \\ 0 & \frac{1}{2} -\i \end{pmatrix}, \qquad \sqrt{d} \mapsto \begin{pmatrix} \frac{1}{2} & 0 \\ 0 & -\frac{1}{2} \end{pmatrix}, \]      
    showing that $\rho_{(A, \lambda, \iota)}$ arises as a direct sum of two $1$-dimensional representations, mapping
    \[   x_1 + x_2\frac{1 + \sqrt{-m}}{2} + x_3 \sqrt{d} + x_4 \sqrt{d} \frac{1 + \sqrt{-m}}{2} \quad \mapsto \]
    \begin{equation} \label{eq:rho_for_iota}  x_1 +  (\frac{1}{2} \pm \i) x_2 \pm \frac{1}{2}x_3 + \frac{1}{2} (\i \pm \frac{1}{2}) x_4
    \end{equation}
    for all $x_i \in \Z$.
    A similar reasoning for $\rho_{(A, \lambda, \iota')}$
    gives 
    \[ \frac{1 + \sqrt{-m}}{2} \mapsto \begin{pmatrix} \frac{1}{2} + \i & 0 \\ 0 & \frac{1}{2} + \i \end{pmatrix}, \qquad \sqrt{d} \mapsto \begin{pmatrix} \frac{1}{2} & 0 \\ 0 & -\frac{1}{2} \end{pmatrix} \] 
    when expressing the images with respect to $(b \omega_1, (1/2-a)\omega_2), (b\omega_1, (-1/2-a)\omega_2)$. Thus we again have a direct sum of two $1$-dimensional representations, namely
    \begin{equation} \label{eq:rho_for_iotaprime}  x_1 +  (\frac{1}{2} + \i) x_2 \pm \frac{1}{2}x_3 \pm \frac{1}{2} (\frac{1}{2} + \i) x_4
    \end{equation}
    From~\eqref{eq:rho_for_iota} and~\eqref{eq:rho_for_iotaprime} we see that $\iota, \iota'$ give rise to non-isomorphic $R$-module structures on $D(A)$. So by Lemma~\ref{lem:CMtypeobstruction} we must be in different orbits under the $\frakC(R)$-action.
\end{example}

We emphasize that there is an important difference between the obstructions from Example~\ref{ex:decidingorbits} and Example~\ref{ex:example_orientation_revisited}. In the former case, precomposition of $\iota'$ with complex conjugation removes the obstruction. For the second example this is not true, nor is it possible to achieve this using another automorphism of $R$: the obstruction to transitivity is really more fundamental here. This is related to the notion of ``inequivalent" CM types discussed in the next section.

\begin{example} \label{ex:jacexample_revisited} Revisiting the $\Z[\zeta_5]$-orientation from Example~\ref{ex:jacexample}, we have $D(A) \cong H^0(C, \Omega^1_C) = \langle dx/y, xdx/y \rangle$. Through
\[ \iota(\zeta_5)^\ast \frac{dx}{y} = \frac{d (\xi_5 x)}{y} = \xi_5 \frac{dx}{y}, \qquad \iota(\zeta_5)^\ast x\frac{dx}{y} =  (\xi_5 x) \frac{ d (\xi_5 x)}{y} = \xi_5^2 x \frac{dx}{y} \]
we see that, here too, $\rho_{(A, \lambda, \iota)}$ decomposes into two $1$-dimensional representations, mapping $\zeta_5$ to multiplication-by-$\xi_5$ and multiplication-by-$\xi_5^2$, respectively.
\end{example}

\subsection{CM types in characteristic $p = 0$} \label{subsec:CMchar0}

In the above examples, the decomposability of $\rho_{(A, \lambda, \iota)}$ into $1$-dimensional subrepresentations $\rho_i : R \to k$, $i = 1, \ldots, g$, is not a coincidence. Moreover, these $1$-dimensional representations are not arbitrary. 
In characteristic $p = 0$ this is classical and captured by the notion of CM type. In this case the $\rho_i$ naturally extend their domains to all of $K$ and take values in some subfield $L \subset k$ that is isomorphic to a Galois closure of $K$. It can be argued~\cite{StrengThesis} that the resulting size-$g$ subset of $\Hom(K, L)$
is a CM type, in the following sense:

\begin{definition}
   Let $K$ be a CM field and let $L \subset \C$ be isomorphic to a Galois closure of $K$. A \emph{CM type} of $K$ is a subset of $\Hom(K, L)$ containing exactly one element from each pair of complex conjugate embeddings $\varphi, \bar \varphi : K \to L$.
\end{definition}

\noindent E.g., this pattern is clearly visible in Example~\ref{ex:jacexample_revisited}. 

Famously, it was proven by Shimura and Taniyama that for $p = 0$ and under mild assumptions on $R$, the CM type is the only obstruction to transitivity. 
One sufficient condition is that $R = \Z_K$,\footnote{The ``right" sufficient condition reads that $R$ is Gorenstein, which is implied, e.g., by the existence of an index-$2$ subfield $K' \subset K$ such that $R \cap K'$ is a maximal order~\cite{bass} (e.g., maximal real multiplication is a sufficient condition).} see~\cite[Thm.\,5.2]{StrengThesis}:

\begin{theorem}[Shimura--Taniyama, 1961] \label{thm:ST}
 Consider a CM field $K$ and let $(A, \lambda, \iota)$, $(A', \lambda', \iota') \in \M_k(\Z_K)$ have the same CM types; here $k$ denotes an algebraically closed field of characteristic $0$. There exists a unique $[(\fraka, \mu)] \in \frakC(\Z_K)$ 
 such that $(A', \lambda', \iota') = [(\fraka, \mu)] \star (A, \lambda, \iota)$.
\end{theorem}

The number of CM types of $K$ is easily seen to be $2^g$, so under the aforementioned mild assumptions on $R$ this is our best guess for the number of orbits in characteristic $0$. However we note that, very exceptionally, the number of orbits may be smaller because certain CM types may not occur, even when $R = \Z_K$.\footnote{This can only happen if the last map in~\eqref{eq:shimura_exact_sequence} is not surjective, see the proof of~\cite[Lem.\,19]{IONICA2020385}. In dimension $g = 2$ all such examples come from the $C_2 \times C_2$ case. } 
Another important difference with the elliptic curve case  is that it may be impossible to jump between orbits by precomposing the orientation with an automorphism of $K$. This comes from the existence of inequivalent CM types:
any automorphism $\tau$ of $K$ commutes with complex conjugation (this is because $K_0$ can be characterized as the maximal totally real subfield of $K$) and therefore, if $\Phi$ is a CM type, then so is $\Phi' = \Phi \circ \tau$, in which case $\Phi, \Phi'$ are said to be \defi{equivalent}.
In our main case of interest $g = 2$ we have from~\cite[Lem.\,3.4]{StrengThesis}:

\begin{lemma}[Inequivalent CM types for $g = 2$] \label{lem:inequivalentCM}
  Let $K$ be a CM field of degree $4$ with Galois closure $L$ and let $G = \Gal(L, \Q)$. Let $\Hom(K, L) = \{ \id, \bar \id, \varphi, \bar \varphi \}$. Write $\Phi = \{ \id, \varphi \}$ and $\Phi' = \{ \id, \bar \varphi \}$. Then:
  \begin{itemize}
      \item if $G \cong C_2 \times C_2$ or $G \cong D_4$ then $\Phi$ and $\Phi'$ are inequivalent,
      \item if $G \cong C_4$ then $\Phi, \Phi'$ are equivalent.
  \end{itemize}
  So  the number of inequivalent CM types in these cases is $2$, $2$, $1$, respectively. 
\end{lemma}

\subsection{Back to positive characteristic}

We now move to our main case of interest $k = \overline{\mathbb{F}}_p$. We expect that under mild assumptions on $R$, the obstruction to transitivity from Lemma~\ref{lem:CMtypeobstruction} should be the only one. Again we formulate this for maximal orders only:\footnote{Here too, it is likely that this assumption can be relaxed significantly, perhaps to $R$ Gorenstein and $p \nmid [\Z_K : R]$ but we did not find references supporting this.}

\begin{theorem} \label{thm:transitivitymodp}
  Let $k = \overline{\mathbb{F}}_p$ and let $K$ be a CM field. Let $(A, \lambda, \iota), (A', \lambda', \iota') \in \M_k(\Z_K)$ be such that $\rho_{(A, \lambda, \iota)} \cong \rho_{(A', \lambda', \iota' )}$. Then there exists a unique $[(\fraka, \mu)] \in \frakC(\Z_K)$ 
  such that $(A', \lambda', \iota') = [(\fraka, \mu)] \star (A, \lambda, \iota)$.
\end{theorem}

\begin{proof}
  It was shown in~\cite[Lem.\,4.4.1]{Goren_Lauter} that every $(A, \lambda, \iota) \in \M_k(\Z_K)$ admits a lift to characteristic $0$.
  As explained in their proof, the possible CM types of such a lift are controlled by \defi{Kottwitz' $\Phi$-determinant condition}: explicitly, fixing an integral basis $\alpha_1, \alpha_2, \ldots, \alpha_{2g}$ of $\Z_K$, for a CM type $\Phi$ one can consider the so-called Kottwitz polynomial
\[ f_\Phi \ = \ \prod_{\varphi \in \Phi} \left( x_1 \varphi(\alpha_1) + x_2 \varphi(\alpha_2) + \ldots + x_{2g} \varphi(\alpha_{2g}) \right) \ \in \Z_{K^\ast}[x_1, x_2, \ldots, x_{2g}]. \]
Here $K^\ast \subset L$ is the fixed field of the group of automorphisms $\tau \in \Gal(L, \Q)$ for which $\tau \circ \Phi = \Phi$ (this is known as the ``reflex field" of $K$ with respect to $\Phi$).
The $\Phi$-determinant condition then states that a lift with CM type $\Phi$ can be found if and only if there exists a prime $\frakp^\ast$ of $K^\ast$ above $p$ such that
$\det\left( \iota(x_1 \alpha_1 + x_2 \alpha_2 + \ldots + x_{2g} \alpha_{2g})^\ast  \right) \equiv f_\Phi \bmod \frakp^\ast$.
In particular, if $\rho_{(A, \lambda, \iota)} \cong \rho_{(A', \lambda', \iota' )}$ then it is possible to find lifts of $(A, \lambda, \iota)$ and $(A', \lambda', \iota')$ having the same CM type, and the connecting Shimura ideal class, whose existence is predicted by Theorem~\ref{thm:ST}, will also connect $(A, \lambda, \iota)$ and $(A', \lambda', \iota')$. \qed
\end{proof}

It is not hard to check that if $p$ is unramified then the $\Phi$-Kottwitz condition cannot be satisfied with respect to more than one CM type $\Phi$. In this case, at least under the assumption $R = \Z_K$, the number of orbits in characteristic $p$ equals the number of orbits in characteristic $0$, which is $2^g$ in most cases. But if $p$ ramifies then there may be multiple options for $\Phi$, meaning that the number of orbits decreases (i.e., several CM types ``coincide" mod $p$).

\begin{example}[CSIDH]
  Assume $g = 1$ and let $K = \Q(\sqrt{-p})$ for some prime number $p$. For simplicity we assume $p \equiv 1 \bmod 4$ so that $\Z_K = \Z[\sqrt{-p}]$ but this does not matter for our conclusions. The two possible Kottwitz polynomials are 
 $x_1 \pm \sqrt{-p} x_2 \in \Z_K[x_1,x_2] $
  (in both cases $K^\ast = K$). Modulo $\frakp^\ast = (\sqrt{-p})$ both polynomials collide to $x_1$, which equals
  $\det(\iota(x_1 + \sqrt{-p} x_2)^\ast) \in \F_p[x_1, x_2]$
  for the Frobenius orientation $\iota : \sqrt{-p} \mapsto \pi_p$ on any supersingular elliptic curve over $\F_p$.
  So in this case the CM types coincide, i.e., there is only one orbit, in accordance with case \emph{(ii)} at the beginning of this section.
\end{example}

\begin{example} \label{ex:example_orientation_revisited_revisited}
Revisiting Example~\ref{ex:example_orientation_revisited}, one easily checks that both $(A, \lambda, \iota)$ and $(A, \lambda, \iota')$ satisfy the Kottwitz condition for a unique CM type, namely $\{ \id, \varphi \}$ and $\{ \id, \bar \varphi \}$, respectively, where
\[ \varphi : \frac{1 + \sqrt{-m}}{2} \mapsto \frac{1 - \sqrt{-m}}{2}, \ \sqrt{d} \mapsto -\sqrt{d}  \]
(and these CM types are inequivalent). So here we expect $2^2 = 4$ orbits.
\end{example}

\begin{example}[R2-D2] \label{ex:R2D2orbits}
For the R2-D2 orientations described in Example~\ref{ex:R2D2}, 
one checks that
$\rho_{(A, \lambda, \iota)}$ maps
\[ \sqrt{d} \mapsto \begin{pmatrix} 0 & d/c_\tau \\ c_\tau & 0 \end{pmatrix}, \qquad \sqrt{-p} \mapsto 0 \]
when expressed with respect to the basis $(\omega, 0), (0, \omega')$, where $\omega, \omega'$ denote regular $1$-forms on $E, E^{(p)}$, respectively. One checks that  
\[ \det \left(\iota\left(x_1 + x_2\sqrt{d} + x_3\sqrt{-p} + x_4\sqrt{d}\sqrt{-p}\right)^\ast\right) = x_1^2 - dx_2^2. \]
This satisfies the Kottwitz condition with respect to all CM types of $\Q(\sqrt{d}, \sqrt{-p})$. So here we expect a unique orbit, coming from a $4$-to-$1$ collision (this is modulo the disclaimer that we have naively applied Theorem~\ref{thm:transitivitymodp} to the non-maximal order $\Z[\sqrt{d}, \sqrt{-p}]$). The resulting transitivity of the Shimura class group action in this setting plays an important role in the attack from~\cite{C2C2attack}.
\end{example}

\begin{remark} In general, for $C_2 \times C_2$ fields, one checks that all $4$ Kottwitz polynomials have the same reduction mod $p$ if and only if $p$ ramifies in both imaginary quadratic subfields. If $p$ ramifies in exactly one imaginary quadratic subfield then we are left with $3$ distinct reductions. As mentioned, in the unramified case the $4$ Kottwitz polynomials have distinct reductions.
\end{remark}

In summary, for a given CM order $R$ in a CM field $K$, a rough roadmap to understanding the orbits of $\frakC(R)$ on the set of primitively $R$-oriented superspecial PPAVs is as follows. Let us assume for simplicity that $R = \Z_K$. By Theorem~\ref{thm:ST} the number of orbits is at most the number $2^g$ of CM types of $K$. But the actual number of orbits is usually smaller:
\begin{itemize}
    \item We first recall that it may (exceptionally) happen that there are no primitively $\Z_K$-oriented PPAVs in characteristic $0$ with a certain given CM type $\Phi$. In view of the proof of Theorem~\ref{thm:transitivitymodp} this then also reduces the number of orbits over $\overline{\F}_p$. Note that this property is necessarily shared by the entire equivalence class of $\Phi$.
    \item The most important decrease comes from the fact that the abelian varieties in certain orbits may not be superspecial. E.g., recall from~\emph{(i)} that if $g = 1$ then we end up with ordinary elliptic curves as soon as $p$ splits in $K$. In higher dimension this is more subtle and a full characterization is lacking, but
    Appendix~\ref{app:reduction} contains the details for $g = 2$ under the assumption that $p$ does not ramify in $K$; this (slightly) extends results by Goren~\cite{Goren1997}.
    \item As we have just discussed, if $p$ ramifies in $K$ then several CM types may ``coincide" mod $p$. This is controlled by Kottwitz' determinant condition.
\end{itemize}
A detailed understanding of all this, as well as of the precise assumptions on $R$ under which Theorems~\ref{thm:ST} and~\ref{thm:transitivitymodp} remain true, is what prevents us from giving a complete orbit description as we did for $g = 1$ at the beginning of the section.

\begin{remark} \label{rmk:frob_conj_orbit}
  Any oriented superspecial PPAV $(A, \lambda, \iota)$ can be defined over $\F_{p^2}$, so we can consider its Frobenius conjugate 
  \[ (A, \lambda, \iota)^{(p)} = (A^{(p)}, \lambda^{(p)}, \iota^{(p)}).\] We believe that $(A, \lambda, \iota)^{(p)}$ always lies in the orbit of $(A, \lambda, \iota \circ \sigma)$, perhaps under mild assumptions on $R$. This would extend Onuki's observation for $g = 1$ and we believe that it should follow from Ekedahl's product criterion~\cite[Prop.\,2.4]{ekedahl-supersingular}; see~\cite[\S2]{Goren1997} for an accessible statement. But we have not analyzed this in detail. 
 \end{remark} 

\section{Security considerations} \label{sec:security}

We conclude with some security aspects
of cryptosystems built from Shimura class group actions. 
Most threats we consider are analogs of existing attack scenarios in the case $g = 1$, that we expect to carry over to higher dimension. But the necessary tools are often  insufficiently developed, due to limited attention and greater technical difficulty. The first 4 aspects have guided us in the parameter selection for the R2-D4 key exchange protocol from Appendix~\ref{app:R2D4_key_exch}:

 \subsubsection*{Endomorphism ring computation.} It is to be expected that the vectorization problem for Shimura class group actions on oriented superspecial PPAVs reduces to the problem of computing their endomorphism rings; this would be in complete analogy with the elliptic curve case~\cite{rational_from_irrational,wesolowski_oriented}. We note that the endomorphism ring computation problem has several interpretations (e.g., computing an explicit basis of $4g^2$ self-maps, determining its structure as an abstract order with involution, computing an IKO-style matrix, \ldots) but they are most likely all equivalent: see~\cite{montessinos2026equivalent} for the surface case, while in arbitrary dimension some aspects are treated in~\cite{page2026conversion}. 
 Up to our knowledge, the fastest known algorithms run in time 
 \[ p^{1/3 + o(1)} \text{ if $g = 1$}, \qquad   p^{4/5 + o(1)} \text{ if $g = 2$}, \qquad  (\log p)^{O(g^2)} p^{g-1} \text{ if $g > 2$}, \]
 see~\cite{wesolowskip13},~\cite{pham_houben} and~\cite[Thm.\,6.15]{mariacrs_endring}, respectively. In practice this attack scenario is largely overruled by
 Kuperberg's quantum attack, which we discuss next.

 \subsubsection{Generic attacks on the abelian hidden shift problem.} The vectorization problem for commutative group actions
   can be modeled as an instance of the ``abelian hidden shift problem"~\cite{childs_jao_soukharev}, for which several quantum sub-exponential time algorithms were devised by Kuperberg~\cite{kuperberg,kuperberg2}. These run in time $L_{\#G}(1/2)$ with $G$ the acting group; so in our case $\#G = \frakh(R)$.
   The concrete cost of Kuperberg's algorithm is not well-understood, even in the elliptic curve case; one of the  fuzzy factors is the cost of an oracle query~\cite{pegasis_quantum,sqale}, and this becomes even more intricate as the dimension grows. In Appendix~\ref{app:R2D4_key_exch} we work with $\log_2 \frakh(R) \approx 2048$, which is on par with CSIDH-4096, commonly believed to reach NIST-I security. In the $\mathbf{g}$-REGA setting, it is customary to sample the secret keys from a strict subset of the acting group, but classical meet-in-the-middle approaches allow to solve the vectorization problem in time $O(\sqrt{\# \mathcal{K}})$ where $\mathcal{K}$ denotes the key space, i.e., the set of expressions of the form~\eqref{eq:keys}. We fix this at $256$ bits for NIST-I.

\subsubsection{Using non-maximal orders.} In the elliptic curve case, the vectorization problem associated with a non-maximal order $R$ reduces to any overorder in which $R$ has a smooth index, by ``walking up the volcano". This lowers the size of the class group and therefore weakens the problem, see e.g.~\cite[\S6.3]{orientexpress}. In higher dimension this picture is more complicated, but it is certain that in many scenarios (e.g., if the order is Bass), such ascending walks are again possible; see~\cite{arpin-marseglia-springer-isogeny-graphs} and Appendix~\ref{app:bass} for more context. So it seems wise to work with maximal orders, or at least with orders $R$ for which $[\Z_K : R]$ is very small.\footnote{It may also be secure to work with CM orders $R$ whose index in $\Z_K$ is large, but then it should be free of small factors so that ascension is infeasible, as in SCALLOP~\cite{scallop}.}
This is also a protective measure against the active ``overorder attack" described below.

\subsubsection{What about ``imprimitive" CM types?} 
This discussion is specific to higher dimension. We mentioned that key exchange from the R2-D2 orientation discussed in Example~\ref{ex:R2D2} got broken in~\cite{C2C2attack}, 
in the sense that the vectorization problem for $\frakC(R)$ reduces to that for $\Cl(R_1), \Cl(R_2)$ with $R_1, R_2$ the two (largest possible) imaginary quadratic subrings of $R$. The attack relies on the colliding inequivalent CM types discussed in Example~\ref{ex:R2D2orbits} and does not obviously generalize to arbitrary $C_2 \times C_2$ orientations: this is a compelling open question.
More generally, one might wonder whether there is a general security issue with mod-$p$ reductions of oriented PPAVs $(A, \lambda, \iota)$  whose CM type $\Phi$ is imprimitive, meaning there is a CM subfield $K' \subset K$ and a CM type $\Phi'$
of $K'$ such that $\Phi$ consists of those embeddings whose restriction to $K'$ is an element of $\Phi'$. It is known that this is equivalent (in characteristic $0$) to $A$ being isogenous to a product of lower-dimensional abelian varieties~\cite[Thm.\,5.2(3)]{StrengThesis}. All CM types of $C_2 \times C_2$ fields are imprimitive, while they are always primitive in the cases $C_4, D_4$.

\subsubsection{Genus theory attacks.} In the elliptic curve case, the decisional Diffie--Hellman assumption for class group actions can often be broken using genus theory~\cite{genus_theory_attack,genus_oriented}. This attack strategy generalizes to higher dimension, as was illustrated in~\cite[\S6]{genus_theory_attack} and worked out in full generality in~\cite{C2C2attack}. It can be countered by making sure one acts with elements from $\frakC(R)^2$ (which is automatic if $\#\frakC(R)$ is odd).

\subsubsection{Backdoors and active attacks.} One realistic threat is that a malicious authority or another adversary could lure Alice or Bob into acting with their secret key on a dishonestly generated PPAV $(A, \lambda, \iota)$, primitively oriented only by an overorder $R'\supsetneq R$ rather than $R$; then in this case the key could be recovered along the lines of Remark~\ref{rmk:sanon_break}. 
This attack can be prevented by public key validation, i.e., by checking that $A$ is superspecial, which can be done efficiently thanks to recent work~\cite{supersingularitytesting}, and that $\iota$ is a primitive $R$-orientation.
If $R$ has a large index in $\Z_K$, however, the non-primitivity of $\iota$ might be hard to detect; see our corresponding discussion at the end of Section~\ref{sec:superspecial}. 
Thus, for this reason as well, it is safer to consider orientations by orders that are near-maximal. 

\bibliographystyle{abbrv}
\bibliography{references}

\appendix

\section{Background on the background} \label{app:background}

\subsection{Cryptographic group actions} \label{subsec:groupactions}

Let $X$ be a finite set and $G$ a finite commutative group in which one can compute efficiently. A \defi{cryptographic group action} $\star: G \times X \to X$ is an action of $G$ on $X$ meeting some hardness properties, most notably one-wayness: given $x_1, x_2 \in X$ it is (computationally) hard to retrieve $g 
\in G$, if it exists, such that $g \star x_1 = x_2$. This problem is often referred to as \defi{vectorization} in $G$. 
An \defi{effective group action} (EGA) is one where 
the group action can be evaluated efficiently at all inputs.
Only fairly recently, it was discovered how to build effective cryptographic group actions from isogenies between elliptic curves, see \cite{pegasis,qtpegasis,clapoti,klapoti}. 

For certain applications such as non-interactive key exchange it is enough to have a so-called \defi{restricted effective group action} (REGA). 
More explicitly, the action is called a \textbf{g}-REGA for some generating subset
$\textbf{g} = \{g_1, \ldots, g_n\}$ of $G$ if $n$ is polylogarithmic in $|G|$ and there exist efficient algorithms to compute $g_i \star x$, $g_i^{-1} \star x$  for all $g_i \in \mathbf{g}, x \in X$. Most isogeny-based group actions in the literature fall into this category,
see e.g.\ \cite{csidh,osidh,ChenuSmith,orientexpress,Houben,coral}. 
This will also be the case for the Shimura class group actions studied in this paper, where potential generalizations of~\cite{qtpegasis} to construct an EGA from this are left as future work. We refer to \cite{egarega} for more details on REGAs and EGAs and their range of applications.

\subsection{Polarized isogenies} \label{subsec:isogenies}

In the next few paragraphs we recall some basic facts and terminology needed for studying the Shimura class group action. All statements below are well-known to specialists and kept concise; we refer the reader to~\cite{StrengThesis,milne_CM_coursenotes,milneAV} and the references included below for more details. 

\subsubsection{Abelian varieties and isogenies.} Let $k$ be an algebraically closed field of characteristic $p$, where $p$ is either a prime or $0$. An \defi{abelian variety} over $k$ is a smooth projective algebraic variety, equipped with a base point $0 \in A$ and geometric morphisms $A \times A \to A : (P, Q) \mapsto P + Q$, $A \to A : P \mapsto -P$ turning $A$ into an algebraic group; elliptic curves are abelian varieties of dimension $1$. A \defi{homomorphism} $\varphi : A_1 \to A_2$ between abelian varieties $A_1, A_2$ is a morphism such that $\varphi(0) = 0$; this is automatically a group homomorphism at the level of points; if $\dim A_1 = \dim A_2$ and $\varphi$ is surjective then it is called an \defi{isogeny}. The set of \defi{endomorphisms} (= self-homomorphisms) of an abelian variety $A$ becomes a ring when equipped with point-wise addition and composition; it is denoted $\End(A)$. We identify $\Z$ with the subring of $\End(A)$ consisting of scalar multiplications and write $\End^0(A) = \End(A) \otimes_\Z \Q$ for the endomorphism algebra.
Assuming $p > 0$, an abelian variety $A / k$ of dimension $g$ is called \defi{superspecial} if it is isomorphic to $E^g$ for some supersingular elliptic curve $E / k$.
If $g \geq 2$ then the choice of $E$ does not matter: by a theorem of Deligne~\cite{goldstein_schacher} any two products of $g$ supersingular elliptic curves are isomorphic. 

\subsubsection{Principal polarizations.} A \defi{polarization} on an abelian variety $A$ is
the isogeny 
\[ \lambda : A\to A^{\vee} \cong \Pic^0(A) : P \mapsto [t_{-P}D - D] \]
attached to an ample divisor $D$ on $A$ (where $t_{-P}$ denotes point-wise translation by $-P$). The reader not familiar with the language of divisors can just think of a polarization as an isogeny to a certain companion abelian variety, called the dual abelian variety, and satisfying a certain ``positivity" condition; this is good enough for almost all that follows. If this isogeny is of degree $1$ then it is called a \defi{principal polarization} and we call $(A, \lambda)$ a principally polarized abelian variety, or PPAV in short (in the specific case of surfaces we will often write PPAS).
While elliptic curves have a unique, canonical, principal polarization, this is not true in general for abelian varieties. 
\begin{example}[Jacobians]
  Consider a smooth projective curve $C$ of genus $g \geq 1$ along with the Abel--Jacobi map $u : C \hookrightarrow \Jac(C) \cong \Pic^0(C) : P \mapsto [P - P_0]$, where $P_0 \in C$ is some fixed base point. The codimension $1$ subvariety of $\Jac(C)$ consisting of all sums of $g-1$ points in the image of $u$ can be seen as a divisor $D$ on $\Jac(C)$. The associated polarization $\lambda : \Jac(C) \to \Jac(C)^\vee$ is a principal polarization, known as the \defi{canonical principal polarization} on $\Jac(C)$.
\end{example}

\begin{example}[products] \label{ex:products}
    If $(A_1, \lambda_1), (A_2, \lambda_2)$ are PPAVs, with principal polarizations coming from divisors $D_1, D_2$, then the product abelian variety $A_1 \times A_2$ comes with the \defi{product polarization} $\lambda_1 \times \lambda_2$, attached to the divisor $D_1 \times A_2 + A_1 \times D_2$, which is again principal.
\end{example}

\subsubsection*{Rosati involution.}
With every homomorphism $\varphi : A_1 \to A_2$ comes a \defi{dual homomorphism} $\hat \varphi : A_2^\vee \to A_1^\vee$ (which, in the case of a surjective homomorphism, is obtained by pulling back divisors on $A_2$ under $\varphi$). If $\lambda_1, \lambda_2$ are principal polarizations on $A_1, A_2$ then this gives rise to an \defi{adjoint homomorphism}
\[ \varphi^\dagger =  \lambda_1^{-1} \circ \hat \varphi \circ \lambda_2 : A_2 \to A_1.  \]
When applied to endomorphisms of a PPAV $(A, \lambda)$, this defines an involution on $\End(A)$
called the \defi{Rosati involution}, which extends by linearity to the entire endomorphism algebra $\End^0(A)$. The aforementioned positivity condition on $\lambda$ is reflected in the fact that the quadratic form $\alpha \mapsto \Tr(\alpha \alpha^\dagger)$ on $\End^0(A)$, where $\Tr(x)$ denotes the trace of the multiplication-by-$x$ map, viewed as a linear map on the $\Q$-vector space $\End^0(A)$, is positive-definite. 

\begin{example} \label{ex:conjugate_transpose}
  Any homomorphism $\varphi : A_1 \times A_2 \to B_1 \times B_2$ between two products can be seen as a matrix
  \[  \begin{pmatrix} \varphi_{11} & \varphi_{21} \\ \varphi_{12} & \varphi_{22}  \end{pmatrix}  \]
  of homomorphisms $\varphi_{ij} : A_i \to B_j$. If $A_1, A_2, B_1, B_2$ come equipped with principal polarizations, then with respect to the respective product polarizations we have
  \begin{equation} \label{eq:rosati_is_conj_transp} \varphi^\dagger = \begin{pmatrix} \varphi_{11}^\dagger & \varphi_{12}^\dagger \\ \varphi_{21}^\dagger & \varphi_{22}^\dagger  \end{pmatrix},
  \end{equation}
  which can be thought of as a ``conjugate-transpose". This naturally generalizes to homomorphisms between products of arbitrarily many abelian varieties.
\end{example}

\subsubsection{Polarized isogenies.}
If $\varphi : A_1 \to A_2$ is an isogeny and $\lambda_2$ is a principal polarization on $A_2$ then $\hat{\varphi} \circ \lambda_2 \circ \varphi$ is a polarization on $A_1$, which cannot be principal in general, if only for degree reasons. According to a theorem of Mumford, if $\lambda_1$ is a principal polarization on $A_1$ then the polarizations on $A_1$ are exactly those isogenies of the form $\lambda_1 \circ \alpha$ with $\alpha$ a \defi{totally positive} endomorphism on $A_1$: by this we mean that $\alpha = \alpha^\dagger$ and the (necessarily real) eigenvalues of the multiplication-by-$\alpha$ map are positive.
Applying this to the foregoing, this means that for $\hat{\varphi} \circ \lambda_2 \circ \varphi$ to be a polarization, 
there must exist a totally positive $\alpha \in \End(A_1)$ such that the diagram
\begin{center}
\begin{tikzcd}[row sep=normal,column sep=normal]
    A_1 \arrow[rd, swap, "\lambda_1" ] & A_1 \arrow[l, swap, "\alpha"] \arrow[r, "\phi"]  & A_2 \arrow[d,"\lambda_2"] \\
    & A_1^\vee & A_2^\vee \arrow[l, "\widehat{\phi}"]
\end{tikzcd}
\end{center}
commutes. We then say that $\phi: (A_1, \lambda_1) \to (A_2, \lambda_2)$ is a \defi{polarized isogeny} with respect to $\alpha$, or an $\alpha$-isogeny for short. An important special case is $\alpha = N \in \Z_{>0}$. We note that the degree of an $N$-isogeny equals $N^g$; the integer $N$ itself is then called its \defi{reduced degree}, denoted $\degrd(\phi)$. An \defi{isomorphism} between PPAVs is a $1$-isogeny.
We state the next result for later use~\cite[Prop.\,13.8]{milneAV}:
\begin{proposition} \label{prop:max_iso}
  Let $(A_1, \lambda_1)$ be a PPAV, let $H \subset A_1$ be a finite subgroup whose order is not a multiple of $p$, and let $\alpha \in \End(A_1)$ be totally positive. Let $\varphi : A_1 \to A_2$ be a separable isogeny with kernel $H$. There exists a principal polarization $\lambda_2$ on $A_2$ such that $\varphi : (A_1, \lambda_1) \to (A_2, \lambda_2)$ is an $\alpha$-isogeny if and only if $H$ is maximal with the property that
  \[ e_\alpha(P, \lambda_1(Q)) = 1, \quad \forall P, Q \in H, \]
  where $e_\alpha$ denotes the $\alpha$-Weil pairing on $A_1$. Moreover $\lambda_2$ is unique in this case.
\end{proposition}

\begin{example}[IKO matrices]\label{ex:IKO}
  Let $E$ be a supersingular elliptic curve over $k$ and assume that $p  > 0$; fix an isomorphism $\theta : \O \to \End(E)$ from a suitable maximal order $\O \subset B_{p, \infty}$ in the unique quaternion algebra ramified at $p$ and $\infty$. Consider the superspecial abelian surface $A = E^2$ equipped with the product polarization $\lambda_1$. In this case Mumford's theorem says that the principal polarizations on $A$ are precisely the maps of the form $\lambda_1 \circ \theta(g)$ where
  \[
    g = \begin{pmatrix} s & r \\ \bar{r} & t \end{pmatrix}, \qquad s, t \in \Z_{>0}, r \in \O, st - r \bar r = 1,
  \]
  and $\theta$ is applied entry-wise; the abbreviation IKO stands for Ibukiyama, Katsura and Oort, the authors of~\cite{ibukiyama1986supersingular} where this is detailed; see Section~\ref{sec:algebra}.
\end{example}

\subsection{Omitted proofs} \label{app:missing_backgroundproofs}

\begin{proof}[of Lemma~\ref{lem:different_mus}]
  Let $\epsilon_1, \ldots, \epsilon_{g-1}$ be a set of fundamental units of $R_0$, chosen such that they are all positive under some fixed real embedding. Since $(R, \epsilon_i^2) \in P_R$ for all $i$, it is clear that all Shimura classes for which $\fraka = R$ are among expressions of the form
  \begin{equation} \label{eq:candidates_for_R} 
   [ (R, \varepsilon_1^{e_1} \cdots \varepsilon_{g-1}^{e_{g-1}})], \ e_i \in \{0, 1\}.
  \end{equation}
  The lemma follows because $\fraka \bar{\fraka} = \mu_1 R = \mu_2 R$ iff $\mu_1^{-1} \mu_2 \in R_0^+$ is a unit. \qed
\end{proof}

\begin{remark}
Not all expressions of the form~\eqref{eq:candidates_for_R} are valid Shimura classes: this is the case only if $\varepsilon_1^{e_1} \cdots \varepsilon_{g-1}^{e_{g-1}}$ is totally positive. It may well happen that this is true only when $e_1 = e_2 = \ldots = e_{g-1} = 0$; in such cases $[(\fraka, \mu)]$ is always uniquely determined by $\fraka$ (however, the existence of $\mu$ still imposes a condition on $\fraka$). Moreover, even in the case of a totally positive unit $\varepsilon \neq 1$, the class $[(R, \varepsilon)]$ may still be trivial: if $\epsilon = \alpha \bar{\alpha}$ for some $\alpha \in R$, 
then 
$(R, \epsilon) = (\alpha R, \alpha \bar{\alpha}) \in P_R$.
\end{remark}

\begin{proof}[of Lemma~\ref{lem:Shimura_integer_norm_subgroup}]
  Consider a Shimura ideal $(\fraka, \mu)$ with $\fraka \subset R$ and let $n$ be an integer such that $\frac{n}{\mu} \fraka$ is integral. Then
  \[ [(\fraka, \mu)]^2 = [(\fraka^2, \mu^2)] = [ (\frac{n}{\mu} \fraka^2, \frac{n}{\mu} \frac{\bar n}{\bar \mu} \mu^2) ] = [ (\frac{n}{\mu} \fraka^2,  n^2) ], \]
  where we have used that $n = \bar n$, $\mu = \bar \mu$. \qed
\end{proof}

\begin{remark}\label{rem:dichtomy1}
    Based on computational evidence, we expect the dichotomy \[ [\frakC(R) : \frakC(R)_\Z] \in \{1, \#\frakC(R)[2]\}\] to hold for any CM order $R \subset K$. 
\end{remark}

\begin{proof}[of Theorem~\ref{thm:action}]
Although it is hard to pinpoint a reference not putting additional assumptions on $R$ and $k$, or not working in the language of complex tori~\cite{Shimura} or Deligne modules~\cite{pila}, the arguments are standard. We highlight the constructive aspects of the proof, because these indicate how to work with the action, and are more concise about the remaining details. 
Let $(A_1, \lambda_1, \iota_1) \in \M_k(R)$ and $[(\fraka, \mu)] \in \frakC(R)$. We show how to construct
\[ (A_2, \lambda_2, \iota_2) = [(\fraka, \mu)] \star (A_1, \lambda_1, \iota_1), \]
a priori as an element of $\M_k(K)$, but we will then be able to conclude that we always land in $\M_k(R)$. By replacing $\fraka$ with an equivalent fractional ideal if needed and scaling $\mu$ accordingly, we can assume that $\fraka$ is integral of norm not divisible by $p$. We then consider the finite subgroup
\[ H = \bigcap_{x \in \fraka} \ker \iota_1(x) \]
of $A_1$, which has order $N_{K / \Q}(\fraka)$. Let $\varphi : A_1 \to A_2 := A_1/H$ be the corresponding quotient isogeny (in the language of $\fraka$-multiplications, this amounts to acting with $\fraka$ on the unpolarized abelian variety $A_1$, see~\cite[\S7]{milne_CM_coursenotes}).
Writing $\mu = \sum_i x_i \bar{y_i}$ for $x_i, y_i \in \fraka$ 
and letting $z_i = x_i \bar{y_i} / \mu \in R$ so that $\sum_i z_i = 1$, for any $P, Q \in H$ we have
\begin{align*} e_{\iota(\mu)}(P, \lambda_1(Q)) & =
\prod_i e_{\iota(\mu)}(\iota(z_i)( P), \lambda_1(Q)) = \prod_i e_{\iota(x_i)\iota(y_i)^\dagger}(P, \lambda_1(Q)) \\ 
& = \prod_i e_{\iota(y_i)^\dagger}(\iota(x_i)(P), \lambda_1(Q)) = \prod_i e_{\iota(y_i)^\dagger}(0, \lambda_1(Q)) = 1
\end{align*}
where we have used some standard properties of the Weil pairing which, e.g., can be found in~\cite[\S3.1]{robert_pairing}, and the fact that $\lambda_1(Q)$ lies in the kernel of the dual of $\iota(y_i)^\dagger$. So by Proposition~\ref{prop:max_iso} there is a unique principal polarization $\lambda_2$ on $A_2$ turning $\varphi$ into an $\iota_1(\mu)$-isogeny. We then let $\iota_2 : K \to \End^0(A_2)$ be the push-forward orientation, i.e., 
\[ \iota_2 : \alpha \mapsto (\varphi \circ \iota_1(\alpha) \circ \varphi^\dagger)/(\varphi \circ \varphi^\dagger). \] 
From the multiplicativity of the unpolarized  action~\cite[Prop.\,7.28]{milne_CM_coursenotes} it is not hard to deduce that acting in this way with $(\fraka \frakb, \mu \nu)$ amounts to the same thing as first acting with $(\fraka, \mu)$, then with $(\frakb, \nu)$. Since it is straightforward that Shimura ideals of the form $(\mu R, \mu \bar \mu)$ with $\mu \in R$ act trivially, this defines an action of $\frakC(R)$ as wanted, and freeness is easily established. To see that this lands in $\M_k(R) \subset \M_k(K)$ it is enough to observe that every ideal class admits two representing ideals with coprime norms, see e.g.~\cite[Proof of Lem.\,27]{lauter_robert} for more details on this step.  \qed
\end{proof}

\section{Endomorphism rings of some well-known supersingular elliptic curves} \label{sec:endrings}

We include explicit descriptions of the endomorphism rings of certain ``extremal" supersingular elliptic curves: 
\begin{itemize}
    \item $y^2 = x^3 + x$ in characteristic $p \equiv 3 \bmod 4$, of $j$-invariant $1728$: being the cryptographers' default example of a supersingular elliptic curve, this is of course very well-known and included for completeness only,
    \item $y^2 = x^3 + 1$ in characteristic $p \equiv 5 \bmod 6$, of $j$-invariant $0$: an explicit description of its endomorphism ring can be found in~\cite[Prop.\,3.2]{mcmurdy} but this description is wrongly  attributed to $y^2 = x^3 - 1$, see Remark~\ref{rmk:5mod6},
    \item $y^2 = x^3 + 12x^2 + 18x$ in characteristic $p \equiv 5 \bmod 8$, which has $j$-invariant $8000$: we are unaware of prior work describing the endomorphism ring as explicitly as in Lemma~\ref{lem:endring5mod8} below; in~\cite[Lem.\,3]{KLPT} it was studied as an abstract maximal order in $B_{p, \infty}$  but, unfortunately, that statement contains a typo, see Remark~\ref{rmk:5mod8} below.
\end{itemize}
We do not cover other congruence classes of $p$, but we note that also the maximal orders from~\cite[Lem.\,4]{KLPT} for the cases $p \equiv 1 \bmod 4$ contain typos, as was pointed out in~\cite[Ftn.\,4]{PIP}.
Since~\cite{mcmurdy,KLPT} are frequently cited references and the typos have been copied into the existing literature, we hope that this appendix helps in cleaning up the landscape. 

\begin{lemma} \label{lem:endring3mod4}
    Let $p \equiv 3 \bmod 4$ be a prime number and consider the quaternion algebra
    \[ B_{p, \infty} = \left( \frac{-1, -p}{\Q} \right) \]
  ramified at $p$ and $\infty$,
  with the usual generators $1, i, j, k$ (i.e., we have $i^2 = -1$, $j^2 = -p$ and $k = ij = -ji$) over $\Q$. Fix a square root $\i \in \F_{p^2}$ of $-1$. The elliptic curve $E : y^2 = x^3 + x$ over $\overline{\F}_p$ with $j$-invariant $1728$ admits the endomorphism
  \[ \iota : (x, y) \mapsto (-x, \i y) \]
  satisfying $\iota^2 = [-1]$ and 
  \begin{equation} \label{eq:map3mod4} 1 \mapsto [1], \quad i \mapsto \iota, \quad j \mapsto \pi_p, \quad k \mapsto \iota \pi_p
  \end{equation}
    defines a ring isomorphism
  \begin{equation*} 
    \Z\left\langle 1, i, \frac{i + j}{2}, \frac{1 + k}{2} \right\rangle \to \End(E),
  \end{equation*}
  where $\pi_p$ denotes the $p$-th power Frobenius endomorphism on $E$.
\end{lemma}

\begin{proof} As mentioned, this is standard, and easier than the lemmata below, so we will be concise. It is immediate that $\iota^2 = [-1]$ and, using $p \equiv 3 \bmod 4$, it is easy to check that $\iota \pi_p = - \pi_p \iota$. Thus the assignment~\eqref{eq:map3mod4} defines an isomorphism $B_{p, \infty} \to \End^0(E)$ at the level of algebras. To complete the proof, it suffices to observe that $(i + j)/2$, $(1 + k)/2$ map to endomorphisms on $E$, which in view of~\cite[Cor.\,III.4.11]{silverman} is equivalent to
\[ E[2] \subset \ker(\iota + \pi_p), \ker([1] + \iota \pi_p), \]
but that is immediate from $E[2] = \{ 0, (0, 0), (\i, 0), (-\i, 0)\}$. This is indeed enough, because an explicit calculation shows that $\Z\langle 1, i, (i + j)/2, (1 + k)/2\rangle$ is a subring of $B_{p, \infty}$ and a discriminant computation shows that it concerns a maximal order.
\end{proof}

\begin{lemma} \label{lem:endring5mod6}
    Let $p \equiv 5 \bmod 6$ be a prime number and consider the quaternion algebra
    \[  B_{p, \infty} = \left( \frac{-3, -p}{\Q} \right) \]
    ramified at $p$ and $\infty$, with the usual generators $1, i, j, k$ (i.e., we have $i^2 = -3$, $j^2 = -p$ and $k = ij = -ji$) over $\Q$. Fix a primitive $3$rd root of unity $z \in \F_{p^2}$. The elliptic curve $E : y^2 = x^3 + 1$ over $\overline{\F}_p$ with $j$-invariant $0$ admits the endomorphism
    \[ \omega : (x, y) \mapsto (zx, y) \]
    satisfying $\omega^2 + \omega + [1] = 0$. Writing $\Omega = [2]\omega + [1]$ and $\pi_p$ for the $p$-th power Frobenius endomorphism on $E$, the assignment
    \begin{equation*} 1 \mapsto [1], \quad i \mapsto \Omega, \quad j \mapsto \pi_p, \quad k \mapsto \Omega \pi_p
  \end{equation*}
    defines an isomorphism
  \begin{equation} \label{eq:order5mod6} 
    \Z\left\langle 1, \frac{-1 + i}{2}, j, \frac{3 + i + 3j + k}{6} \right\rangle \to \End(E)
  \end{equation} 
    of rings.
\end{lemma}

\begin{proof}
  It is easy to check that the (horizontal) line connecting two points of the form $(zx_0, y_0)$, $(z^2x_0, y_0)$ has $(x_0, y_0)$ as its third point of intersection with $E$. This implies that $\omega^2 + \omega + [1] = 0$, and as a special case we find that $(0, 1)$ and $(0, -1)$ are points of order $3$. 

  An easy calculation shows $\Omega^2 = [-3]$. Since  $\pi_p^2 = [-p]$ and $\pi_p \omega = \omega^2 \pi_p$ so that $\Omega \pi_p = -\pi_p \Omega$, the assignment from the statement of the lemma is seen to define an isomorphism between $B_{p, \infty}$ and $\End^0(E)$. Clearly $(-1 + i)/2 \mapsto \omega$, so it suffices to show that 
  \begin{equation} \label{eq:divby6} 
    [3] + \Omega + [3]\pi_p + \Omega \pi_p = [4] + [2]\omega + [4] \pi_p + [2] \omega \pi_p 
  \end{equation}
  annihilates $E[6]$: then the left-hand side of~\eqref{eq:order5mod6} lands inside $\End(E)$, and this left-hand side is a maximal order, as can be checked explicitly.
  
  It obviously annihilates $E[2]$, so it suffices to concentrate on $E[3]$.
  An analysis of the $3$-division polynomial $3x(x^3 + 4)$ of $E$ shows that, besides $(0, \pm 1)$, the points of order $3$ are
  \[ (t, \pm \sqrt{-3}), \ (z t, \pm \sqrt{-3}), \ (z^2 t, \pm \sqrt{-3}) \]
  with $t = (-4)^{-(p-2)/3} \in \F_p$. We now do a case-by-case analysis of how $\pi_p$ acts on these points, keeping in mind that they are of order $3$:
  \begin{itemize}
      \item on $(0, \pm 1)$ it acts as the identity, so expression~\eqref{eq:divby6} acts as $\omega - [1]$ which clearly kills these points,
      \item on $(t, \pm \sqrt{-3})$ it acts by negation, easily seen to make expression~\eqref{eq:divby6} vanish,
      \item on $(zt, \pm \sqrt{-3})$ it acts as $-\omega$, again making expression~\eqref{eq:divby6} vanish after reduction modulo $[3]$,
      \item similarly, on $(z^2t, \pm \sqrt{-3})$ it acts as $-\omega^2$, which similarly makes expression~\eqref{eq:divby6} vanish modulo $[3]$.
  \end{itemize}
This concludes the proof.
\end{proof}

\begin{remark} \label{rmk:5mod6}
In contrast with what is claimed in~\cite[Prop.\,3.2]{mcmurdy} the statement fails for $E : y^2 = x^3 - 1$ as soon as $p \equiv 3 \bmod 4$. In this case $3$ is a square in $\F_p$ and one finds that the points $(t, \pm \sqrt{3}) \in E[3]$ with $t = 4^{-(p-2)/3} \in \F_p$ are fixed by Frobenius. Expression~\eqref{eq:divby6} acts on these points as $\omega - [1]$ which fails to annihilate them.
\end{remark}

\begin{lemma} \label{lem:endring5mod8}
  Let $p \equiv 5 \bmod 8$ be a prime number and consider the quaternion algebra
  \[ B_{p, \infty} = \left( \frac{-2, -p}{\Q} \right) \]
  ramified at $p$ and $\infty$,
  with the usual generators $1, i, j, k$ (i.e., we have $i^2 = -2$, $j^2 = -p$ and $k = ij = -ji$) over $\Q$. Fix a square root $z \in \F_{p^2}$ of $-2$. The elliptic curve $E : y^2 = x^3 + 12x^2 + 18x$ over $\overline{\F}_p$ with $j$-invariant $8000$ admits the endomorphism
  \[ \theta : (x, y) \mapsto \left( \frac{-y^2}{2x^2}, \frac{z}{4} y (1 - \frac{18}{x^2}) \right) \]
  satisfying $\theta^2 = [-2]$ and
  \begin{equation} \label{eq:iso5mod8} 
    1 \mapsto [1], \quad i \mapsto \theta, \quad j \mapsto \pi_p, \quad k \mapsto \theta \pi_p
  \end{equation}
  defines a ring isomorphism
  \begin{equation} \label{eq:ring5mod8} 
    \Z\left\langle 1, i, \frac{1 + i + j}{2}, \frac{i + 2j \pm k}{4} \right\rangle \to \End(E)
  \end{equation}
  where the sign $\pm$ should be read as $+$ if $p \equiv 5, 13 \bmod 48$ and as $-$ if $p \equiv 29, 37 \bmod 48$, and $\pi_p$ denotes the $p$-th power Frobenius endomorphism.
\end{lemma}

\begin{proof}
  In the proof all occurrences of $\pm$ should be read as $+$ if $p \equiv 5, 13 \bmod 48$ and as $-$ if $p \equiv 29, 37 \bmod 48$, and conversely for $\mp$. Throughout we fix a square root $\i \in \F_p$ of $-1$.
  
  It is easy to check that $\theta$ is a well-defined endomorphism of degree $2$ satisfying $\theta^2 = [-2]$.
  It is also easy to see that $\pi_p^2 = [-p]$ and $\theta \pi_p = -\pi_p \theta$, showing that~\eqref{eq:iso5mod8} defines an isomorphism between $B_{p, \infty}$ and $\End^0(E)$. It thus suffices to show that 
  \[ E[2] \subset \ker([1] + \theta + \pi_p), \qquad E[4] \subset \ker(\theta + [2]\pi_p \pm \theta \pi_p), \] 
  because this implies that the left-hand side of~\eqref{eq:ring5mod8} indeed lands inside $\End(E)$, and this left-hand side is a maximal order, regardless of the sign chosen, as can be checked via an explicit discriminant calculation. 

  We do this in an artisanal way.
  First, write $O = (0, 0)$, $T_0 = (-6 + 3\i z, 0)$, $T_1 = (-6 -3\i z, 0)$
  for the points of order $2$ on $E$. Using
  \[ \theta(O) = 0, \quad \theta(T_0) = \theta(T_1) = O \]
  and
  \[ \pi_p(O) = O, \quad \pi_p(T_\ell) = T_{1-\ell} \text{ for $\ell=0, 1$} \]
  one indeed checks that $[1] + \theta + \pi_p$ annihilates $E[2]$ as wanted. 
  
  The remainder of the proof is a bit more tedious. By explicit calculation, the points of order $4$ are seen to be groupable as follows:
  \begin{itemize}
      \item the halves of $O$, which have $x$-coordinates $3 \i z$, $-3 \i z$,
      \item the halves of $T_0$, with $x$-coordinates 
      \[ -6 + 3\i z + 6\sqrt{1 - \i z}, \ -6 + 3\i z - 6\sqrt{1 - \i z}, \]
      \item the halves of $T_1$, with $x$-coordinates 
      \[ -6 - 3\i z + 6\sqrt{1 + \i z}, \ -6 - 3\i z - 6\sqrt{1 + \i z}. \]
  \end{itemize}
  These $x$-coordinates all live in $\F_{p^2}$, i.e., $1 + \i z$ and $1 - \i z$ are squares in $\F_{p^2}$. This follows, e.g., from $\pi_{p^2} = [-p]$, implying that for any point $P \in E[4]$ it holds that $\pi_{p^2}(P) = [-p] P = - P$ where we have used $p \equiv 1 \bmod 4$. For the same reason none of the corresponding $y$-coordinates is defined over $\F_{p^2}$, i.e., the points of order $4$ all have field of definition $\F_{p^4}$.

  It is immediate to see how $[2] \pi_p = \pi_p [2]$ acts on these halves: the respective results are $O$, $T_1$, $T_0$, independent of which half was chosen. Thus it suffices to show that
  \[ \theta (1 \pm \pi_p) = (1 \mp \pi_p) \theta \]
  also maps the halves of $O$ to $O$, the halves of $T_0$ to $T_1$ and the halves of $T_1$ to $T_0$, with the sign depending on $p \bmod 48$ as in the statement of the lemma.

  Let us first concentrate on the halves of $O$. Since $O \in \ker \theta$ and since $\theta^2 = [-2]$, these halves must be mapped inside $\{T_0, T_1\}$ under $\theta$. But this means that $(1 \pm \pi_p) \theta$ maps every half of $O$ to $T_0 + T_1 = O$ as wanted.
  
  Now consider a half of $T_\ell$ for $\ell \in \{0, 1\}$. Since $\theta(T_\ell) = O$ this half must be mapped to a half of $O$ under $\theta$.  
  It follows that $(1 \mp \pi_p)\theta$ maps the half of $T_\ell$ to a sum of two halves of $O$ with different $x$-coordinates. It is clear that this is indeed a point of order $2$ different from $O$, hence an element of $\{ T_0, T_1 \}$, but we need to show that it concerns $T_{1-\ell}$. 
  
  We concentrate on a half $H$ of $T_0$ (the other case is analogous). An explicit calculation shows that 
  \[ \theta(H) = (-3\i z, y_0) \]
  for some $y$-coordinate $y_0 \in \F_{p^4}$. Then
  \[ \pi_p \theta(H) = (3\i z, y_0^p) \]
  and the usual addition formula shows that the $x$-coordinate of $(1 \mp \pi_p) \theta(H)$ is given by
  \[ \left( \frac{\mp y_0^p - y_0}{6 \i z} \right)^2 - 12 = \frac{1}{72}(y_0^p \pm y_0)^2 - 12 =  \frac{1}{72}( (y_0^2)^p \pm 2 y_0^{p+1} + y_0^2) -12. \]
  Since $y_0^2 = (-3\i z)^3 + 12(-3\i z)^2 + 18(-3\i z) = 108 (2 - \i z )$ this equals
  \[ -6 \pm 3 \cdot 108^{(p-1)/2}  (2 - \i z)^{(p+1)/2} = -6 \pm 3 \left(\frac{3}{p} \right) (2 - \i z)^{(p+1)/2},  \]
  with $(\frac{\cdot}{p} )$ the Legendre symbol. So upon setting this equal to the $x$-coordinate of $T_1$, it suffices to show
  \[ \pm \left( \frac{3}{p} \right) (2 - \i z)^{(p+1)/2} =  -\i z. \]
  Note, as a sanity check, that both sides square to $2$, but we have to argue that we have the same square root of $2$ on both sides. For this we rely on the congruence
  \[ (2 + \sqrt{2})^{(p+1)/2} \equiv \sqrt{2} \bmod p \]
  in $\Z[\sqrt{2}]$, which holds if and only if $p \equiv 3, 13 \bmod 16$ assuming $p$ is inert: to prove this, the main trick is to rewrite $2 + \sqrt{2} = \zeta^{-1}(1 + \zeta)^2$ for a primitive $8$th root of unity $\zeta$. So using $p \equiv 5 \bmod 8$ we find that the equality holds with $+$ if and only if $p \equiv 5, 13 \bmod 48$ and with $-$ if and only if $p \equiv 29, 37 \bmod 48$, as wanted.
\end{proof}

\begin{remark} \label{rmk:5mod8}
  The two orders in~\eqref{eq:ring5mod8}  are isomorphic via $j \mapsto -j$, so also if $p \equiv 29, 37 \bmod 48$ then it is possible to give an isomorphism from the maximal order with a $+$ sign, but the isomorphism changes. 
  These orders are essentially taken from~\cite[Lem.\,3]{KLPT} but, there, the third basis element reads $(1 + j + k)/2$ rather than $(1 + i + j)/2$. This is an error: for the sign $-$ the resulting order is not maximal but of index $3$ (for the sign $+$ it is).
\end{remark}

\section{R2: three families of $\F_p$-rational orientations} \label{app:R2families}

In this section we study families of oriented PPASs defined over $\F_p$, where the orientations are such that the $p$-th power Frobenius endomorphism  is among the orienting endomorphisms; in this context, we will denote this Frobenius endomorphism by $\Pi_p$ to avoid confusion with $p$-th power Frobenius maps between elliptic curves, for which we keep writing $\pi_p$. We have already encountered R2-D2 in Example~\ref{ex:R2D2}. The other orientations are called R2-C4 and R2-D4. Besides paying tribute to a series of astromech droids, these names make reference to the Galois group (of the Galois closure) of the underlying CM field $K$, which is either isomorphic to $D_2 \cong C_2 \times C_2$, $C_4$ or $D_4$ (indeed, $A_4$ and $S_4$ are impossible in the CM case). The R2 stands for ``rational and $2$-dimensional".\footnote{We invite the reader to complete our series with a key exchange protocol for which the name C-3PO would be well-suited.} Our main motivation for including $\Pi_p$ in the ring of orienting endomorphisms is that its representation comes ``for free", i.e., we do not need to carry around corresponding interpolation data; see Section~\ref{subsec:representation}.
This comes with important gains, 
similar to why, in the elliptic curve case, CSIDH stands out in terms of efficiency. 

For each of the  families, the starting surface is the Weil restriction
\[ A = W_{\F_{p^2} / \F_p}(E) \cong_{\bar \F_p} E \times E^{(p)} \]
of a well-chosen supersingular elliptic curve $E / \F_{p^2}$, equipped with the product polarization $\lambda$. Recall: the $p$-th power Frobenius endomorphism on $A$ is given by
\[ \Pi_p = \begin{pmatrix} 0 & \pi_p \\ \pi_p & 0 \end{pmatrix} \]
with $\pi_p : E \to E^{(p)}$ the Frobenius isogeny. 
In R2-D2 we use a curve $E$ for which $\pi_{p^2} = -p$, from which it follows that $\Pi_p^2 = -p$. In particular this requires that our orienting order $R$ contains $\sqrt{-p}$. For this reason, 
for R2-C4 and R2-D4 one must proceed differently: indeed, $C_4$-fields and $D_4$-fields that are CM cannot contain any imaginary quadratic subfields. Instead we choose $E / \F_{p^2}$ such that $\pi_{p^2} = p$, so that $\Pi_p^2 = p$.
To our knowledge, the use of PPAS over $\F_p$ on which the $p$-th power Frobenius endomorphism does not act as $\sqrt{-p}$ has not been considered in a cryptographic context before, so these constructions may be found interesting for independent reasons. 

\begin{remark} \label{rmk:props_of_Fp_orientations}
Since $\Pi_p$ commutes with all other orienting endomorphisms, all oriented PPAS in the orbit of $(A, \lambda, \iota)$ under the action of $\frakC(R)$ can again be assumed to be defined over $\F_p$.
For use in Appendix~\ref{app:compress}, let us also note that the assumptions $\pi_{p^2} = \pm p$ give rise, respectively, to isomorphisms
  \[ A(\F_p) \cong E(\F_{p^2}) \cong \left( \frac{\Z}{(p \mp 1)\Z} \right)^2 \hspace{-0.1cm}, \  A(\F_{p^2}) \cong E(\F_{p^2}) \times E(\F_{p^2}) \cong  \left( \frac{\Z}{(p \mp 1)\Z} \right)^4  \]
  of abelian groups.
\end{remark}

More precisely, we consider the following method for generating Frobenius-oriented PPASs over $\F_p$, which produces all possible Galois groups.
As we have mentioned, we must start from a supersingular elliptic curve $E / \F_{p^2}$ on which $\pi_p^2 = p$, with known endomorphism ring, and let $(A, \lambda)$ be its Weil restriction. 
Our strategy is to let
\[ K = \Q \left( \sqrt{a + b\sqrt{p}} \right) \] 
for $a, b \in \Z$ such that $a + |b| \sqrt{p} < 0$ (so that it concerns a quartic CM field) that will be determined later. We
look for an orientation of the form
\[ \iota : \sqrt{p} \mapsto \Pi_p = \begin{pmatrix} 0 & \pi_p \\ \pi_p & 0 \end{pmatrix},  \sqrt{a + b\sqrt{p}} \mapsto \begin{pmatrix} \alpha & \beta \pi_p \\ \pi_p \beta & \frac{1}{p}\pi_p \alpha \pi_p \end{pmatrix} \]
for suitable 
\[ \alpha \in \End(E), \ \beta \in \frac{1}{p} \Hom(E^{(p)}, E) \pi_p  \supsetneq \End(E) \] of trace $\tr(\alpha) = \tr(\beta) = 0$.
An easy calculation shows that 
the right-most matrix commutes with $\Pi_p$ and is negated when taking the conjugate-transpose. Let us compute
\begin{align*} 
 \begin{pmatrix} \alpha & \beta \pi_p \\ \pi_p \beta &  \frac{1}{p}\pi_p \alpha \pi_p \end{pmatrix}^2 & =  \begin{pmatrix} p \beta^2  + \alpha^2 & 0 \\ 0 & \pi_p \beta^2 \pi_p + \frac{1}{p}\pi_p \alpha^2 \pi_p \end{pmatrix} + \begin{pmatrix} 0 & ( \alpha \beta + \beta \alpha) \pi_p  \\ \pi_p ( \alpha \beta + \beta \alpha)  & 0 \end{pmatrix}  \\ 
& = (-p \nrm(\beta) -\nrm(\alpha)) \iota(1) +  \tr(\alpha \beta) \iota(\sqrt{p})
\end{align*}
where we have used that $\tr(\beta) = \tr(\alpha) = 0$.  Thus for $a = - p \nrm(\beta) - \nrm(\alpha)$ and $b = \tr(\alpha \beta)$ the map $\iota : K \to \End^0(A)$ is well-defined.
Our key tool to control the Galois group of $K$ is the following lemma:

\begin{lemma} \label{lem:conradlemma}
  Under the above assumptions on $a, b$, the Galois group (of the Galois closure) of $K$ is isomorphic to
  \begin{enumerate}
      \item $D_2 \cong C_2 \times C_2$ if $a^2 - b^2p$ is a square,
      \item $C_4$ if $a^2 - b^2p$ is not a square but $p(a^2 - pb^2)$ is,
      \item $D_4$ if neither $a^2 - pb^2$ nor $p(a^2 - pb^2)$ is a square.
  \end{enumerate}
In particular, if $p \equiv 3 \bmod 4$ then $K$ is either a $C_2 \times C_2$-field or a $D_4$-field.
\end{lemma}
\begin{proof} The classification can be found in~\cite[Lem.\,4.7]{conrad_cubicquartic}. As for the last claim, the classification shows that $\Gal(K , \Q) \cong C_4$ if and only if $a^2 - b^2p = pk^2$ for some $k \in \Z$. In particular $a = pa'$ for some integer $a'$ such that $p a'^2 = b^2 + k^2$. Since $p \equiv 3 \bmod 4$ this is only possible if $a', b, k$ are all even. But then we can take out a factor $4$ and the argument repeats, leading to a contradiction by infinite descent. \qed
\end{proof}

In the concrete orientations below, the curve $E$ will be constructed as a twist of some well-known supersingular elliptic curve $E_0 : y^2 = f(x)$ defined over $\F_p$:
\[ E : v y^2 = f(x), \qquad v \in \F_{p^2} \text{ non-square}.\]
Note that $\pi_{p^2} = -p$ on $E_0$ (assuming $p > 3$) and therefore we indeed have $\pi_{p^2} = p$ on $E$.
The twisting isomorphisms
\[ \tau : E \to E_0 : (x, y) \mapsto (x, u y), \qquad \tau^{(p)} : E^{(p)} \to E_0 : (x, y) \mapsto (x, u^p y),  \]
where $u \in \F_{p^4} \setminus \F_{p^2}$ satisfies $u^2 = v$, combine (component-wise) into an isomorphism of PPAVs
\[ T : W_{\F_{p^2} / \F_p}(E) \to E_0^2 \] 
(for the product polarizations). Using $\pi_p \tau = \tau^{(p)} \pi_p$ and $\pi_p \tau^{(p)} = - \tau \pi_p$ one checks
\[ T \circ \Pi_p \circ T^{-1} = \begin{pmatrix} 0 & -\pi_p \\
\pi_p & 0 \end{pmatrix}, \quad   T \circ \begin{pmatrix}  \alpha & \beta \pi_p \\ \pi_p \beta & \frac{1}{p}\pi_p \alpha \pi_p \end{pmatrix}   \circ T^{-1} = \begin{pmatrix} \tilde{\alpha} & -\tilde{\beta} \pi_p \\ \pi_p \tilde{\beta} & -\frac{1}{p}\pi_p \tilde{\alpha} \pi_p \end{pmatrix},  \]
where the (quasi-)endomorphisms $\tilde{\alpha} = \tau \alpha \tau^{-1}$ and $\tilde{\beta} = \tau \beta \tau^{-1}$ denote the ``lollipops" of $\alpha, \beta$.
Thus, one can also think of $\iota$ as an orientation of $(E_0^2, \lambda_0)$, where for any
\[ \tilde{\alpha} \in \End(E_0), \quad \tilde{\beta} \in  \End(E_0)  \frac{1}{\pi_p} = \frac{1}{\pi_p} \End(E_0)  \]
with trace $0$, we obtain an orientation by 
\begin{equation} \label{eq:orientwithZ}
\Z[\sqrt{a + b\sqrt{p}}]
\end{equation} with $a = -p \nrm(\tilde{\beta}) - \nrm(\tilde{\alpha})$ and $b = \tr(\tilde{\alpha} \tilde{\beta})$; the advantage of this viewpoint is that it allows us to work directly in $\End(E_0)$, but it becomes  less apparent that the orienting endomorphisms on the Weil restriction are $\F_p$-rational. 

\begin{remark}
  This viewpoint also makes for a clearer link with the machinery from Section~\ref{sec:constructing_orientations}: applying Lemma~\ref{lem: cent} to $g = \Id_2$ and the symmetric element
  \[ P = \begin{pmatrix} 0 & -\pi_p \\ \pi_p & 0 \end{pmatrix} \in M_2(\End(E_0)), \]
  we are looking for quadratic subrings of $\Q(P) \otimes_{\Q} \End(E_0)$ of the form~\eqref{eq:orientwithZ}.
\end{remark}

\begin{example}[R2-D4] \label{ex:R2D4}
In view of Lemma~\ref{lem:conradlemma}, let us assume $p \equiv 3 \bmod 4$. We start from $E_0 : y^2 = x^3 + x$,
whose endomorphism ring is of course very well known, see Lemma~\ref{lem:endring3mod4}: it is isomorphic to the maximal order
\[ 
  \O_0 = \Z \langle 1, i, \frac{i + j}{2}, \frac{1 + k}{2} \rangle \subset B_{p, \infty} = \left( \frac{-1, -p}{\Q} \right) 
\]
where an isomorphism $\O_0 \to \End(E_0)$ maps $i$ to $(x,y) \mapsto (-x, \i y)$ and $j$ to $p$-th power Frobenius $\pi_p$. 

We proceed as above and look for suitable trace-$0$ elements $\alpha \in \O_0$, $\beta \in  \O_0 \frac{1}{j}$, where we have dropped the $\tilde{\, \cdot \,}$'s to lighten notation.
A $\Z$-basis of the trace-$0$ sublattice of $\O_0$ is given by $i$, $\frac{i + j}{2}$, $k$, so a general trace-$0$ element is
\[
  \alpha = a_1 i + a_2 \frac{i + j}{2} + a_3 k.
\]
For $\beta$, a direct computation of the reduced trace on $\frac{1}{j}\O_0$ shows that $\frac{1}{j}$, $\frac{i}{j}$, $\frac{1 + k}{2j}$ form a $\Z$-basis of its trace-$0$ sublattice, so
\[
  \beta = \frac{1}{j} \Bigl( b_1 + b_2 i + b_3 \frac{1 + k}{2} \Bigr)
\]
for integers $a_1, a_2, a_3, b_1, b_2, b_3$. This yields
\begin{align*}
  a & = -a_1^2 - a_1 a_2 - \frac{p + 1}{4} a_2^2 - p a_3^2 - b_1^2 - b_1 b_3 - \frac{p + 1}{4} b_3^2 - b_2^2, \\
  b & = a_2 b_1 + (a_1 + a_2) b_3 - 2 a_3 b_2.
\end{align*}
For each choice of $a_i, b_i$ such that $a^2 - b^2p$ is a non-square (generically the case) we obtain an orientation of type $D_4$ by Lemma~\ref{lem:conradlemma}. The condition $a + |b| \sqrt{p} < 0$ is automatically satisfied because every quartic number field contained in $M_2(B_{p, \infty})$ is necessarily totally imaginary.

As for the primitivity, 
an easy calculation shows
  \begin{equation} \label{eq:discD4} \disc \Z [\sqrt{ a + b\sqrt{p}}] = 256b^4p^2(a^2 - b^2p). 
  \end{equation}
  From~\cite[Thm.\,1 and Tbl.\,B]{huard_spearman_williams} we see that for $p \equiv 3 \bmod 4$ this equals $\disc K$ as soon as $b = \pm 1$ and $a^2 - p$ is square-free.
  So in these cases the order is maximal and our orientation is automatically primitive. A near-maximal example where $|b| = 4$ will be studied in Appendix~\ref{app:key_exch}.

\end{example}

\begin{example}[R2-C4] \label{ex:R2C4}
  We follow the same recipe to realize $C_4$. 
  Lemma~\ref{lem:conradlemma} tells us that we must move away from characteristics $p \equiv 3 \bmod 4$. We therefore start from the elliptic curve $E_0 : y^2 = x^3 + 12x^2 + 18x$ over $\F_{p^2}$, where $p \equiv 5 \bmod 8$. We refer to Lemma~\ref{lem:endring5mod8} for an explicit description of its endomorphism ring, which  is isomorphic to the maximal order
  \[ \Z\left\langle 1, i, \frac{1 + i + j}{2}, \frac{i + 2j \pm k}{4} \right\rangle \subset B_{p, \infty} = \left( \frac{-2, -p}{\mathbb{Q}} \right), \]
  with the sign depending on whether $p \equiv 5, 13 \bmod 48$ or $p \equiv 29, 37 \bmod 48$. For concreteness, we proceed under the assumption that $p \equiv 5, 13 \bmod 48$, corresponding to the $+$-sign. The trace-$0$ condition imposes
  \begin{align*}
    \alpha & =  a_1 i + a_2(i + j) + a_3 \frac{i + 2j + k}{4},  \\
    \beta & = \frac{1}{j} \left( b_1 + b_2 i + b_3 \frac{1 + i + j}{2} - b_3 \frac{i + 2j + k}{4}  \right)
  \end{align*}
  for integers $a_1, a_2, a_3, b_1, b_2, b_3$, so we end up with expressions
\begin{align*}
  a & =  - 2a_1^2 -(p+2)a_2^2 - 4a_1a_2  - a_1a_3 - (p+1)a_2a_3  \\ & \qquad \qquad \qquad \qquad \qquad \qquad - \frac{3p + 1}{8}a_3^2 - b_1^2 - 2b_2^2 - b_1b_3 - b_2b_3 - \frac{p + 3}{8} b_3^2, \\
  b & = - a_1b_3 + 2a_2b_1 + a_3b_1 - a_3b_2.
\end{align*}
Again this will generically result in orientations of type $D_4$. To end up in the $C_4$ case, we write $p = x^2 + y^2$ as a sum of two squares (possible because $p \equiv 1 \bmod 4$) and look for a non-trivial zero of the quadratic form
\[ x a - p b \in \Z[a_1, a_2, a_3, b_1, b_2, b_3]. \]
Then $0 = (xa - pb)(xa + pb) = (p - y^2)a^2 - p^2b^2 = p(a^2 - pb^2) - a^2y^2$ so that $p(a^2 - pb^2)$ is a square as wanted.

Controlling primitivity for R2-C4 is harder than for R2-D4. Indeed, by construction the factor $a^2 - pb^2$ is a square away from factors $p$, and following~\cite[Thm.\,1 and Tbl.\,C]{huard_spearman_williams} this makes our order non-maximal, typically. An example family of solutions is where $p = 1 + 4m^2$ for some integer $m \equiv 5 \bmod 6$ (so that $p \equiv 5 \bmod 8$ and $p \not \equiv 29, 37 \bmod 48$): we can take
\begin{multline*}
   a_1 = -m+1, a_2 = - 1, a_3 = 2, b_1 = -1, b_2 = m, b_3 = 2 \quad \Rightarrow \quad a = -2p, b = -2 
\end{multline*}
yielding $a^2 - pb^2 = 16 p m^2$. So $\disc K$ is smaller than~\eqref{eq:discD4} by a factor $m^2$, at least.
\end{example}

\begin{remark} \label{rmk:nonempty_orbit_with_conductor_p}
Many of the examples we found satisfy $p \mid a, b$ and consequently $p^3 \mid a^2 - b^2p$ (since $p(a^2 - b^2p)$ has to be a square), in which case one checks that
\[ \gamma = \frac{\sqrt{a + b\sqrt{p}}}{\sqrt{p}}  \]
is integral. If we want our orientation to extend to $\Z[\gamma] \subset \Z_K$ then
\[ \iota(\sqrt{a + b\sqrt{p}}) = \begin{pmatrix} \alpha & \beta \pi_p \\ \pi_p \beta & \frac{1}{p}\pi_p \alpha \pi_p \end{pmatrix} = \begin{pmatrix} 0 & \pi_p \\ \pi_p & 0 \end{pmatrix} \begin{pmatrix} \beta & \frac{1}{p} \alpha \pi_p \\ \frac{1}{p} \pi_p \alpha & \frac{1}{p} \pi_p \beta \pi_p \end{pmatrix} \]
should factor through $\Pi_p$, i.e., the
right-most factor should really be an element of $\End(A)$.
However, this is not always the case. The bottleneck is the bottom-right entry, which is not always free from denominators; recall that we sample $\beta$ from $(1/p)\Hom(E^{(p)}, E) \pi_p$. E.g., again for $p = 1 + 4m^2$ we can consider
\begin{multline*}
   a_1 = -2m^2 - m + 1, a_2 = -2m - 1, a_3 = 4m + 2, \\ b_1 = 2m - 1, b_2 = -2m^2 + 3m, b_3 = -4m + 2
\end{multline*}
(leading to  $a = -2p^2, b = -2p$) and one checks that the coefficient at $k$ of $j\beta j$ equals $b_2 + b_3 / 4 \equiv 2m + 1 \not \equiv 0 \bmod p$. 
So this example shows that $\M_k(R)$ can be non-empty even if $p \mid [\Z_K : R]$: this cannot occur in dimension $g = 1$. 
\end{remark}

\section{Proofs from Section~\ref{sec:evaluating_the_action} (except for Lemma~\ref{lem:almost_iso})}\label{app:proofevalsection}

\subsection{Proof from Section \ref{subsec:genset}}
 \begin{proof}[of Lemma~\ref{lem:possible_splittings}]
   If $K = \Frac(R)$ is Galois (cases $C_2 \times C_2, C_4$) then 
   $\ell R$ necessarily factors into prime ideals of equal degree $4/r$ for $r \in \{1, 2, 4\}$ and the analysis is easy. In the $D_4$-case there is an exception to this:
   it may happen that $\ell R = \frakp_1 \frakp_2 \frakp$ for degree-$1$ primes $\frakp_1, \frakp_2$ and a degree-$2$ prime $\frakp$. But then necessarily $\frakp$ is self-conjugate and this does not lead to a suitable factorization of $\ell R$.
   \qed
 \end{proof}

\subsection{Proofs from Section \ref{sec:Kani_with_CM}}

\begin{proof}[of Theorem~\ref{thm:embedding}]
The product polarization on $A_i^u$ gives that
\[ \iota_{A_i}(\alpha)^\dg = (\iota_{i}(\alpha_{k,j})^\dg)_{1 \leq k, j \leq u} = (\iota_{i}(\bar \alpha_{k,j}))_{1 \leq j, k \leq u} = \iota_{A_i}(\alpha^{*}) \] 
for any $\alpha \in M_u(R)$.
Moreover, $\iota_{A_i}(\alpha)^\dg \circ \iota_{A_i}(\alpha) = \iota_{A_i}(n \Id_u) $ because $\iota_{A_i}: M_u(R) \to \End(A_i^u)$ is an injective ring homomorphism and $\alpha^{*}\alpha = n \Id_u \in M_u(R)$.
Then, $\Phi^{\dg}= \left( \begin{smallmatrix} \varphi^{\dg}\Id_u & -\alpha_{A_1}^{\dg} \\ \alpha_{A_2} & \varphi\, \Id_u \end{smallmatrix} \right)$. 
Using
$\alpha^{*}\alpha=\alpha\alpha^{*}=n\Id_u$,
$\varphi^{\dg} \circ \varphi=\iota_1(d)$, $\varphi \circ \varphi^{\dg}=\iota_2(d)$, the four
blocks of $\Phi^{\dg}\Phi$ are
\[
\begin{aligned}
(1,1)&: \varphi^{\dg}\varphi\,\Id_u+\alpha_{A_1}^{\dg}\alpha_{A_1}=\iota_1(d+n)\Id_u= \iota_1(N) \Id_u,\\
(1,2)&: \varphi^{\dg}\alpha_{A_2}^{\dg}-\alpha_{A_1}^{\dg}\varphi^{\dg}\Id_u
     =0,\\
(2,1)&: \alpha_{A_2}\varphi\,\Id_u-\varphi\,\alpha_{A_1}\Id_u
     =0, \\
(2,2)&: \alpha_{A_2}\alpha_{A_2}^{\dg}+\varphi\varphi^{\dg}\Id_u=\iota_2(n+d)\Id_u=\iota_2(N) \Id_u,
\end{aligned}
\]
where in $(1,2)$ and $(2,1)$ we used that $\varphi$ is oriented, hence $\varphi^{\dg} \circ \iota_2(\alpha)^\dg = \iota_1(\alpha)^\dg \circ \varphi^{\dg}$ and $\varphi \circ \iota_1(\alpha)=\iota_2(\alpha) \circ \varphi$. 
Since $\iota_i(N)=[N]$, it follows that $\Phi^{\dg}\Phi=[N]$ and $\Phi$ is an $N$-isogeny of PPAVs of dimension $2ug$,
of degree $N^{2ug}$.

(i) $\Phi(P,0)=(\varphi\,\Id_u(P),\,-\alpha_{A_1}(P))$ for $P\in A_1^u$; taking
$P=(P_1,0,\dots,0)$ the first coordinate of the image is $\varphi(P_1)$.

(ii) For $P\in A_1^u[N]$,
\begin{multline*}
  \Phi\bigl(-\alpha_{A_1}^{\dg}(P),\,\varphi\,\Id_u (P)\bigr)
  = \\ \bigl(-\varphi\,\alpha_{A_1}^{\dg}(P)+\alpha_{A_2}^{\dg}\varphi (P),\;
         \alpha_{A_1}\alpha_{A_1}^{\dg}(P)+\varphi^{\dg}\varphi (P)\bigr)
  =\bigl(0,\,[N]P\bigr)=0 .
\end{multline*}
The map $P\mapsto(-\alpha_{A_1}^{\dg}(P),\varphi\,\Id_u (P))$ is injective on $A_1^u[N]$: indeed, if both
coordinates vanish, then $P\in A_1^u[nR_0] \cap A_1^u[dR_0]$;
coprimality of $n R_0$ and $d R_0$ then implies that $xn+yd=1$ for suitable $x,y \in R_0$, hence $P=0$. Since $N$ is coprime to the base field characteristic and
$\#A_1^u[N]=N^{2ug}=\deg\Phi$, the image coincides with $\ker\Phi$. 
For the second equality, since 
$nR_0 + NR_0 = nR_0 + d R_0 =R_0$, then $nR_0$ and $NR_0$ are coprime, hence $\alpha_{A_1}$ acts as an automorphism on $A_1^u[N]$: indeed, $A_1[N]$ is a finite group and if $P \in A_1^u[N]$ and $\alpha_{A_1}(P) = 0$, then $P \in A_1^u[nR_0] \cap A_1^u[NR_0]$, and again $P = 0$ by coprimality of $n R_0$ and $N R_0$. 
Then, substituting $P\mapsto\alpha_{A_1}(P)$ and using 
$-\iota_1(n) \Id_u =\iota_1(d) \Id_u$ on $A_1^u[N]$, and $\varphi \circ \iota_1(\alpha)=\iota_2(\alpha) \circ \varphi$, we obtain the stated subgroup. \qed
\end{proof}

\begin{remark}[continuation of Remark~\ref{rmk:meet-in-the-middle}] \label{rmk:app_meet-in-the-middle}
We have that $\Phi=\Psi_2\circ\Psi_1$ with $\Psi_1,\Psi_2$ both $\sqrt N$-isogenies and
$\ker\Psi_1=\sqrt N\ker\Phi$, hence
\[
  \ker\Psi_1=\bigl\{\bigl(-\alpha_{A_1}^{\dg}(P),\ \varphi\,\Id_u(P)\bigr)
  \mid P\in A_1^{u}[\sqrt N]\bigr\}.
\]
The Rosati dual
$\Phi^{\dg}$
is again of the shape of the theorem, up to sign flips in the $\alpha_{A_i}$'s, so the same argument yields
\[
\begin{aligned}
  \ker\Psi_2^{\dg} &=\bigl\{\bigl(\alpha_{A_2}^{\dg}(Q),\, \varphi^{\dg}\Id_u(Q)\bigr) \mid Q\in A_2^{u}[\sqrt N]\bigr\} \\  &=\bigl\{\bigl(\alpha_{A_2}^{\dg}\varphi\,\Id_u(P),\ \iota_1(d)\Id_u(P)\bigr) \mid P\in A_1^{u}[\sqrt N]\bigr\}.
\end{aligned}
\]
One computes
$\Psi_1$ and $\Psi_2^{\dg}$ from their kernels and glues them along an isomorphism between
their codomains; the details of this gluing in the middle are technical, but are discussed in~\cite[App.\,F]{sqisignhd}, for instance. 
\end{remark}

\begin{proof}[of Proposition~\ref{prop:auxiliary-alpha}]
Since $\alpha^{*}=\left(\begin{smallmatrix}X^{*}&-Y^{*}\\Y&X\end{smallmatrix}\right)$, we have
\[
  \alpha^{*}\alpha=\begin{pmatrix}X^{*}X+Y^{*}Y & (YX-XY)^{*}\\[2pt]
  YX-XY & XX^{*}+YY^{*}\end{pmatrix}.
\]
Under the hypotheses of (a) we have $X^{*}X = XX^{*}= x\Id_m$ and $Y^{*}Y = YY^{*}=y\Id_m$, 
so both diagonal blocks equal $(x+y)\Id_m$ and the off-diagonal blocks vanish. 
The conclusion for $\alpha\alpha^{*}$ follows similarly.

\noindent The hypotheses of (b) are exactly the conditions under which $\alpha^*\alpha$ equals $n\,\Id_{2m}$. 

\noindent For (c), the claim follows by induction
on $k$. The case $i=0$ is proven by $\bar \theta_0 \theta_0 = \theta_0 \bar \theta_0 = N_{K/K_0}(\theta_0)$. For $i \geq 1$, a scalar matrix $\theta \Id_m$ satisfies $(\theta \Id_m)^{*}(\theta\Id_m)= (\theta\Id_m)(\theta \Id_m)^{*} =N_{K/K_0}(\theta)\,\Id_m$ and commutes with every matrix, so (a) applies at each step
with $X=\alpha_{i-1}$ and $Y=\theta_i\,\Id_{2^{i-1}}$, adding the term $N_{K/K_0}(\theta_i)$ to $n$.

\noindent For (d), the claim again follows by induction and similarly as for (c), since the case $i=0$ holds by the hypotheses on $\alpha_0$ and the hypotheses of (a) apply at every step $i \geq 1$ with $X = \alpha_{i-1}$, $x = n_1 + \sum_{j=1}^{i-1} N_{K/K_0}(\theta_j')$, and $Y = \theta_i'\Id_{2^{k+i-1}}$. The last statement follows taking $\alpha_0$ as the matrix from \cite[Lemma 4.2]{robert-note}.
\qed
\end{proof}

\section{Proof of Lemma~\ref{lem:almost_iso}} \label{app:almostiso}

This appendix is devoted to a detailed proof of Lemma \ref{lem:almost_iso}. We also prove an easy corollary. Let us first recall the statement of the Lemma.
\AlmostIso*

\begin{corollary} \label{cor:kerbound}
    Assume that $K$ is a $C_2 \times C_2$ CM field with $\Z_K^\times$ containing no roots of unity other than $\pm 1$. Let $R$ be an order in $K$ and $R_i = R \cap K_i$ for $i = 1, 2$. Then,
    \[ |\ker \kappa_R| \leq \frac{C_1 C_2}{2q(K)} \prod_{\ell \ \text{prime}} e(\ell) \]
    where $C_i = f_i \prod_{p \ \text{prime},\ p|f_i} \left( 1 - \left(\tfrac{d_{K_i}}{p}\right)\tfrac{1}{p} \right)$, with $f_i = [\Z_{K_i} : R_i]$ the conductor of $R_i$, and $d_{K_i}$ the discriminant of $K_i$, for $i = 1,2$.
\end{corollary}

\begin{proof}
If $([\fraka_1],[\fraka_2]) \in \Cl(R_1) \times \Cl(R_2)$ is in the kernel of $\kappa_R$, then \[ ([\fraka_1\Z_{K_1}], [\fraka_2\Z_{K_2}]) \in \Cl(\Z_{K_1}) \times \Cl(\Z_{K_2}) \] is in the kernel of $\kappa_{\Z_K}$. This yields a well-defined group homomorphism $\psi: \ker{\kappa_R} \to \ker{\kappa_{\Z_K}}$ whose kernel consists of the pairs in $\ker\kappa_R$ of classes of invertible ideals of $R_1$, $R_2$ that become principal when extended to $\Z_{K_1}$, $\Z_{K_2}$, respectively.
Hence, 
\[ |\ker \psi| \leq \frac{h(R_1)}{h(\Z_{K_1})}\cdot \frac{h(R_2)}{h(\Z_{K_2})} \] 
where $h$ denotes the class number, and 
\[ \frac{h(R_i)}{h(\Z_{K_i})} = f_i \prod_{\substack{p \text{ prime} \\ p \mid f_i}}\left( 1- \left( \frac{d_{K_i}}{p}\right)\frac{1}{p}\right)\]
using~\cite[Theorem 7.24]{cox2013primes}, and remembering the assumption that the only roots of unity in $\mathbb{Z}_K$ (and hence in $\mathbb{Z}_{K_i}$ and $R_i$) are $\pm 1$.
Together with Lemma \ref{lem:almost_iso}, the inequality $|\ker\kappa_R| \leq |\ker{\psi}| |\ker{\kappa_{\Z_K}}|$ concludes the proof. 
\qed
\end{proof}

To prove the lemma we will make use of the following results from \cite{lemmermeyer_kuroda}.

\begin{lemma}[{\cite[1.3]{lemmermeyer_kuroda}}] \label{lem:j_jhat}
    Let $K$ be a $C_2 \times C_2$ extension of $\Q$ and let $K_0, K_1, K_2$ be the three intermediate quadratic fields. The map
    \[ j: \Cl(\Z_{K_0}) \times \Cl(\Z_{K_1}) \times \Cl(\Z_{K_2}) \to \Cl(K) : ([\fraka_0], [\fraka_1], [\fraka_2]) \mapsto [\fraka_0\Z_K \fraka_1\Z_K \fraka_2 \Z_K] \]
    is a group homomorphism.
    Taking the restriction map 
    $\widehat{j} :\widehat{C} \to \Cl(\Z_K)$ where 
    \begin{multline*} 
    \widehat{C} = \{ ([\fraka_0], [\fraka_1], [\fraka_2]) \in \Cl(\Z_{K_0}) \times \Cl(\Z_{K_1}) \times \Cl(\Z_{K_2}) \mid \\ N_{K_0/\Q}(\fraka_0) N_{K_1/\Q}(\fraka_1) N_{K_2/\Q}(\fraka_2) = 1 \} 
    \end{multline*}
    gives a surjection $\coker{\widehat{j}} \twoheadrightarrow \coker{j}$ and 
    \[ |\coker{\widehat{j}}| = \frac{2^{d-2}  \prod_{\ell \text{ prime}} e(\ell) }{[\Z^\times : H]}, \]
    where $d$ is the number of infinite primes that ramify in $K/\Q$, 
    $e(\ell)$ is the ramification index of a prime $\ell$ in the extension $K/\Q$, 
    and $H \leq \Z^\times$ is the subgroup of integer units that are in the image of the norm map $N_{K/\Q}$. 
\end{lemma}

\begin{lemma}[{\cite[1.5]{lemmermeyer_kuroda}}] \label{lem:class_number}
    Let $K$ be a CM $C_2 \times C_2$ field, with quadratic subfields $K_0, K_1, K_2$. Then
    \[ h(\Z_K) = \frac{1}{2} q(K) h(\Z_{K_0}) h(\Z_{K_1}) h(\Z_{K_2}) \]
    where $h$ denotes the class number and $q(K) = [\Z_K^\times : \Z_{K_0}^\times \Z_{K_1}^\times \Z_{K_2}^\times]$ is the unit index.
\end{lemma}

Moreover, we will need the following formulas for the narrow class number $h^+$ and the order of the Shimura class group.

\begin{lemma} \label{lem:narrow_cl_number}
    Let $L$ be a number field, then 
    \[ h^+(\Z_{L}) = \frac{2^r}{[\Z_{L}^\times : (\Z_{L}^\times)^+]} h(\Z_{L}),\]
where $r$ is the number of real embeddings $\sigma_i:L \hookrightarrow \C$, $i \in \{ 1, \ldots r\}$.
\end{lemma}

\begin{proof}
    There is an exact sequence
   \[ 1 \to \frac{\Z_{L}^\times}{(\Z_{L}^\times)^+}  \xrightarrow{u~\mapsto~u} \frac{L^\times}{(L^\times)^+}  \xrightarrow{u~\mapsto~[u\Z_{L}]} \Cl^+(\Z_{L})  \xrightarrow{[\fraka]~\mapsto~[\fraka]} \Cl(\Z_{L}) \to 1. \]
    The formula follows from surjectivity of the sign map $\sgn : L^\times \to \{\pm1\}^r$ sending elements $x \in L$ to $( \, \sgn(\sigma_1(x)), \, \ldots \, , \, \sgn(\sigma_r(x)) \, ) \in \{\pm 1\}^r$, with kernel $(L^\times)^+$. \qed
\end{proof}

\begin{lemma} \label{lem:shimura_order}
    Let $K$ be a CM $C_2 \times C_2$ field. If $\Z_K^\times$ contains no roots of unity other than $\pm 1$, then  
    \[ |\frakC(\Z_K)| = \frac{1}{2} h(\Z_{K_1}) h(\Z_{K_2}) |\coker N_{K/K_0}| \]
    where $N_{K/K_0}: \Cl(\Z_K) \to \Cl^+(\Z_{K_0})$ is the ideal class norm map.
\end{lemma}

\begin{proof}
    From the exact sequence in Section \ref{sec:shimura_cl_group}, it follows that 
    \[ |\frakC(\Z_K)| = [(\Z_{K_0}^\times)^+ : N_{K/K_0}(\Z_K^\times)] \frac{h(\Z_K)}{|\im{N_{K/K_0}}|}. \]
    Using Lemma \ref{lem:class_number}, and Lemma \ref{lem:narrow_cl_number} applied to $K_0$, we can rewrite this as 
    \[  \frac{1}{8}\frac{h^+(\Z_{K_0})}{|\im{N_{K/K_0}}|} h(\Z_{K_1}) h(\Z_{K_2}) q(K) [\Z_{K_0}^\times : (\Z_{K_0}^\times)^+][(\Z_{K_0}^\times)^+ : N_{K/K_0}(\Z_K^\times)] .\]
    With the assumption that $\Z_K^\times$ contains no roots of unity other than $\pm 1$, then $q(K) = [\Z_K^\times : \Z_{K_0}^\times \Z_{K_1}^\times \Z_{K_2}^\times] = [\Z_K^\times : \Z_{K_0}^\times]$, yielding
    \[ |\frakC(\Z_K)| = \frac{1}{8} |\coker{N_{K/K_0}}| h(\Z_{K_1}) h(\Z_{K_2}) [\Z_{K}^\times : N_{K/K_0}(\Z_K^\times)] .\]
    Again under the assumption that $\Z_K^\times$ contains no roots of unity other than $\pm 1$, it is not difficult to show that $[\Z_{K}^\times : N_{K/K_0}(\Z_K^\times)] = 4$, which ends the proof. \qed
\end{proof}

We are now ready to prove Lemma \ref{lem:almost_iso}.

\begin{proof}
    First observe that if $([\fraka_1], [\fraka_2]) \in \Cl(\Z_{K_1}) \times \Cl(\Z_{K_2})$ is in the kernel of $\kappa_{\Z_K}$, then $([1], [\fraka_1], [\fraka_2]) \in \Cl(\Z_{K_0}) \times \Cl(\Z_{K_1}) \times \Cl(\Z_{K_2})$ is in the kernel of the map $j$ from Lemma \ref{lem:j_jhat}, namely there is an injection $\ker{\kappa_{\Z_K}} \hookrightarrow \ker{j}$, from which we have
    \[ |\ker{\kappa_{\Z_K}}| \leq |\ker{j}| = \frac{h(\Z_{K_0}) h(\Z_{K_1}) h(\Z_{K_2}) }{h(\Z_K)}|\coker{j}| .\]
    Using Lemma \ref{lem:class_number}, then 
    \[ |\ker{j}| = \frac{2}{q(K)}|\coker{j}|,\]
    where one can see that $q(K) \in \{1,2\}$, following e.g. \cite[Proposition 5.12]{milneANT}. Hence, 
    \[ |\ker{j}| \leq \frac{2}{q(K)}|\coker{\widehat{j}}| \leq \frac{1}{2q(K)} \prod_{\ell \text{ prime}} e(\ell) \]
    using Lemma \ref{lem:j_jhat} for the second inequality, with $d=1$ (as the real infinite prime of $\Q$ ramifies in $K$) and $[\Z^\times : H] = 2$ (as $-1$ is not a norm in the totally imaginary field $K$).
    This proves the bound on $|\ker{\kappa_{\Z_K}}|$. 

    On the other hand, 
    \[ |\coker{\kappa_{\Z_K}}| = \frac{|\frakC(\Z_K)|}{h(\Z_{K_1}) h(\Z_{K_2})} |\ker{\kappa_{\Z_K}}| = \frac{1}{2}|\coker{N_{K/K_0}}||\ker{\kappa_{\Z_K}}|  \]
    by Lemma \ref{lem:shimura_order}, which proves the bound on $|\coker{\kappa_{\Z_K}}|$. \qed
\end{proof} 

\begin{example}  \label{ex:shimuraC2C2}
    We consider the R2-D2 orientation introduced in Example \ref{ex:R2D2} and
    used in Section~\ref{app:R2D2_keyexchange} below.
    Here $K = \Q(\sqrt{d}, \sqrt{-p})$ for a prime $p > 3$ and a square-free integer $d$ coprime to $p$. Then $K_0 = \Q(\sqrt{d})$, and say $K_1 = \Q(\sqrt{-p})$, $K_2 = \Q(\sqrt{-dp})$. Assume that $d=q_1 q_2$, where $q_1, q_2 >3$ are prime. Assume moreover $p \equiv d \equiv 3 \bmod 4$. 
    Then $\disc{\Z_K} = 2^4 p^2 q_1^2 q_2^2$ by~\cite[Thm.\,3]{Williams_1970}. 
    There are then exactly four ramified primes in $K/\Q$, namely $2, p, q_1, q_2$, and since each of them is ramified in exactly two quadratic subfields, we have $e(2) = e(p) = e(q_1) = e(q_2) = 2$.
    From Lemma \ref{lem:almost_iso} it follows that $|\ker \kappa_{\Z_K}| \leq 8 / q(K)$. 
    Since the extension $K/K_0$ is ramified at a prime ideal of $\Q(\sqrt{d})$ above $p$, the norm map $N_{K/K_0}$ is surjective on $\Cl^+(\Z_{K_0})$, hence $|\coker{\kappa_{\Z_K}}| \leq 4 / q(K)$. 
\end{example}

\section{Key exchange from R2} \label{app:key_exch}

We now delve into the details of generating concrete parameters for a Diffie--Hellman style key exchange protocol built from Shimura class group actions, along the usual diagram depicted in Figure~\ref{fig:key-exchange-diagram_general}.
\begin{figure}[h]
\centering
\begin{tikzcd}[column sep=15pt, row sep=small]
(A, \lambda, \iota) \arrow[rrrrrr, teal, "{\frakg_1^{a_1} \frakg_2^{a_2} \cdots \frakg_r^{a_r}}"{teal}]
\arrow[dddd, teal, "{\frakg_1^{b_1} \frakg_2^{b_2} \cdots \frakg_r^{b_r}}"{teal, description}]
&  &  &  &  &  &
(A_{\mathcal{A}}, \lambda_{\mathcal{A}}, \iota_{\mathcal{A}}) \arrow[dddd, teal, "{\frakg_1^{b_1} \frakg_2^{b_2} \cdots \frakg_r^{b_r}}"{teal, description}]
\\
&  &  &  &  &  &
\\
&  &  &  &  &  &
\\
&  &  &  &  &  &
\\
(A_{\mathcal{B}}, \lambda_{\mathcal{B}}, \iota_{\mathcal{B}}) \arrow[rrrrrr, teal, "{\frakg_1^{a_1} \frakg_2^{a_2} \cdots \frakg_r^{a_r}}"'{teal}]
&  &  &  &  &  &
(A_{\mathcal{AB}}, \lambda_{\mathcal{AB}}, \iota_{\mathcal{AB}})
\end{tikzcd}
\caption{Diffie--Hellman key exchange based on Shimura class group actions, for a set of ``generators" $\mathbf{g} = \{ \frakg_1, \ldots, \frakg_r \} \subset \frakC(R)$ and exponents sampled from $\{-B, \ldots, B\}^r$.}
\label{fig:key-exchange-diagram_general}
\end{figure}
We emphasize that the construction below was done for proof-of-concept purposes only: we leave considerable room for optimization, and the security remains to be assessed more thoroughly. For the appealing reasons mentioned in Section~\ref{sec:key_exch_over_Fp}, we concentrate on the R2 orientations from Appendix~\ref{app:R2families} (see also Example~\ref{ex:R2D2}), where the entire orbit is defined over $\F_p$. 
The ``best" family to work with seems to be R2-D4. Indeed, in R2-C4 the parameters are quite constrained, hindering parameter generation, while R2-D2 has security issues~\cite{C2C2attack}. 

\subsection{Key exchange from R2-D4} \label{app:R2D4_key_exch}

Recall that R2-D4 is a (not necessarily primitive) orientation by $R = \Z[\gamma]$ with
\[ \gamma = \sqrt{a + b\sqrt{p}}, \qquad a + |b|\sqrt{p} < 0. \]
Note that
\[ \bar \gamma \gamma = d := -a - b\sqrt{p} \in R_0^+ \]
so in order to represent $\iota(\gamma)$, as explained in Section~\ref{subsec:representation} we can stay in dimension $2$ as soon as we can write $N - d = \alpha^2$ for some $\alpha = s + t\sqrt{p} \in \Z[\sqrt{p}]$. This spells out as:
\begin{itemize}
    \item $2st = b$,
    \item $s^2 + pt^2 = N + a$,
\end{itemize}
showing that we cannot let $b = \pm 1$, which is a necessary condition for $\Z[\gamma]$ to be maximal. We choose $b = 4$, and then we can assume that $s = 2$, $t = 1$ so that $a = p + 4 - N$.\footnote{The reason we skip $b = 2$ is that in this case $a = p + 1 - N$. Since we will choose $p$ such that $p-1$ and $N$ share a large prime power $M$, implying $a^2 - b^2p \equiv (2 - N)^2 - 4 \equiv 0 \bmod M$, it follows that $\disc K$ plummets in view of~\cite[Thm.\,1 and Tbl.\,B]{huard_spearman_williams}.}

\subsubsection{Enforcing complete splitting of small primes.} For an efficient evaluation of the Shimura class group action in the $\mathbf{g}$-REGA model, we want that $A[\ell]$ is ``easily accessible", for every small odd prime $\ell$ involved in the construction of $\mathbf{g}$. In view of Remark~\ref{rmk:props_of_Fp_orientations} the most natural sufficient condition is $\ell \mid p - 1$, so that $A[\ell] \subset A(\F_{p^2})$. Then $\ell$ splits completely in $R$ if and only if $a \pm 4$ are both non-zero quadratic residues modulo $\ell$. Interestingly, since
\[ a = p + 4 - N \equiv 5 - N \bmod \ell, \]
this condition does not explicitly depend on $a$ and/or $p$. So the idea is to collect small prime numbers $\ell_1, \ell_2, \ldots, \ell_r$ modulo which both $9 - N$ and $1-N$ are non-zero quadratic residues; for a fixed value of $N$ we expect this to hold for a proportion $\frac{1}{4}$ of the primes (within the wiggle room one typically has when choosing $N$, one can give preference to the value giving rise to the smallest $\ell_i$). We then choose
\[ p = M c \ell_1 \ell_2 \cdots \ell_r + 1 \]
with $M$ a large factor of $N$ and $c$ a small cofactor ensuring that $p$ is a prime number. Here, more precisely, we take $N = M^2$ so that $A[\sqrt{N}] \subset A(\F_{p^2})$ for the purpose of meet-in-the-middle. 

\subsubsection{Realization as an instance of R2-D4.}
For an incorporation into the R2-D4 family we need $|a| > p/4$ and consequently $N = p + 4 - a > 5p/4$. This comes about as follows.
Recall from Example~\ref{ex:R2D4} that
our task is to find integers $a_1, a_2, a_3, b_1, b_2, b_3$ such that
\begin{align*}
  a & = -a_1^2 - a_1 a_2 - \frac{p + 1}{4} a_2^2 - p a_3^2 - b_1^2 - b_1 b_3 - \frac{p + 1}{4} b_3^2 - b_2^2, \\
  4 = b & = a_2 b_1 + (a_1 + a_2) b_3 - 2 a_3 b_2.
\end{align*}
These formulae can be rewritten as
    \begin{align*}
  -4a & = A^2 + B^2 + 4 b_2^2 + p(a_2^2 + 4a_3^2 + b_3^2), \\
8 & = A b_3 + B a_2 - 4 a_3 b_2,
\end{align*}
where $A = 2a_1 + a_2$ and $B = 2b_1 + b_3$. So the bound $|4a| > p$ is immediate. Typically this is not  enough:
we want some extra leeway in order for the following solving strategy to succeed: iterate over positive integers $n = 1, 2, 3, \ldots$ and over all vectors $(a_2, 2a_3, b_3) \in \Z \times 2\Z \times \Z$ of squared-Euclidean length $n$ until we can find a corresponding vector $(A, B, 2b_2)$ of squared-Euclidean length $-4a - np$ in the affine plane 
\[ H : b_3 X + a_2 Y - 2a_3 Z = 8, \]
where moreover we want that $A \equiv a_2 \bmod 2$ and $B \equiv b_3 \bmod 2$, so that  $a_1 = (A - a_2)/2$ and $b_1 = (B - b_3)/2$ are integers.

This can be turned into an efficient algorithm, but for generating the concrete parameters below we have used a simplified search by restricting to $a_3 = 0$ and $\gcd(a_2, b_3) = 1$; this is good enough for our proof-of-concept purposes, although it increases the expected number of $n$'s to be tried and therefore requires a somewhat larger gap between $|4a|$ and $p$.
To proceed, we use the extended Euclidean algorithm to compute $\lambda, \mu \in \Z$ such that $1 = \lambda b_3 + \mu a_2$. Then for any $x, b_2$ the tuple $(8\lambda + x a_2, 8\mu - x b_3, 2b_2)$ lies in $H$. Expressing that it has squared-Euclidean length $-4a-np$ amounts to
\[ \left(nx + 8(\lambda a_2 - \mu b_3) \right)^2 + 4n b_2^2 = -4na - n^2p - 64. \]
This asks to represent the right-hand side by the positive definite binary quadratic form $X^2 + 4nY^2$. If this right-hand side is prime, then we expect it to be representable with probability $1 / 2h(\Z[\sqrt{-4n}]) \approx 1/\sqrt{n}$ by~\cite[Thm.\,9.12]{cox2013primes} and in that case a corresponding representation can be found efficiently using Cornacchia's algorithm. Three remarks are in order:
\begin{itemize}
  \item We model $-4na - n^2p - 64$ as a random integer of  size polynomial in $p$, so that we expect success after a number of tries that is polylogarithmic in $p$. However if $3 \mid a$ and $p \equiv 2 \bmod 3$ then using $3 \nmid n$ (indeed, $n$ is a sum of coprime squares) one sees that $-4na - n^2p - 64$ is always divisible by $3$, so in these cases the modeling is clearly flawed.\footnote{Other deviating features are that if $n$ is even then $-4na - n^2p - 64 \equiv 0 \bmod 4$, while if $n$ is odd then $-4na - n^2p - 64 \equiv 1 \bmod 4$. We did not observe a noticeable effect on the runtime of our method.} But in our parameters below we have $3 \nmid a$ and $p \equiv 1 \bmod 3$ and this is not a concern.
  \item If this integer can indeed be represented as $c^2 + 4nd^2$ for certain $c, d \in \Z$, then there is no guarantee that $c = nx + 8(\lambda a_2 - \mu b_3)$ for some $x \in \Z$. This may seem a strong additional assumption but luckily it is not: one checks that $(\lambda a_2 - \mu b_3)^2 \equiv -1 \bmod n$ and therefore $c^2 \equiv (8(\lambda a_2 - \mu b_3))^2 \bmod n$. E.g., if $n$ is prime then this implies $c \equiv \pm 8(\lambda a_2 - \mu b_3) \bmod n$ so by changing the sign of $c$ if needed we can ensure that $c$ is  of the form $nx + 8(\lambda a_2 - \mu b_3)$. If $n$ is not prime then this may fail but in any case, in view of the Erd\"os--Kac theorem, the inverse probability of success is in the order of $\log n$ rather than in the order of $n$ as one might expect at first sight.
  \item If $n$ is even then $-4na - n^2p - 64 \equiv 0 \bmod 4$ so the primality assumption forces $n$ to be odd. It is then easy to check that the requirements $A \equiv a_2 \bmod 2$ and $B \equiv b_3 \bmod 2$ are automatic.
\end{itemize}

\subsubsection{Artifact of $p \equiv 3 \bmod 4$: $N$ cannot be a power of $2$.} Ideally, we would like $N$ to be a large power of $2$, but this conflicts with $M \mid p - 1$ given the R2-D4 assumption $p \equiv 3 \bmod 4$. It should be clear to the reader that this assumption is not crucial: e.g., Example~\ref{ex:R2C4} where $p \equiv 5 \bmod 8$ also produces instances of type $D_4$. Switching to $p \equiv 1 \bmod 4$ will be more technical (by lack of a uniform description of $B_{p, \infty}$) but should be perfectly doable. We will by-pass this by letting $M = 3^f$, $N = 3^{2f}$ for some exponent $f$.

\subsubsection{Group size.} 
Recall from~\eqref{eq:discD4} that the discriminant of $\Z[\gamma]$ equals $2^{16} p^2 (a^2 - 16p)$. But $\Z[\gamma]$ is not a maximal order; in fact it has index $b^2 = 16$ in the order $\Z[\gamma, \sqrt{p}]$, by which our PPASs are still naturally oriented; consequently $\Z[\gamma, \sqrt{p}]$ has discriminant $2^8p^2(a^2 - 16p)$. This order is still not maximal, but almost: using that $a = p + 4 - 3^{2f}$ and $p \equiv 1 \bmod 3^f$ for some large $f$, one verifies
\[ a^2 - 16p \equiv 4 \bmod 8, \qquad a^2 - 16p \equiv 9 \bmod 27, \]
and then from~\cite[Thm.\,1 and Tbl.\,B]{huard_spearman_williams}
it follows that the discriminant of the maximal order is $2^6 p^2 (a^2 - 16p)/9$, if we assume that $(a^2 - 16p)/36$ is square-free (very likely true for the parameters below).
We believe that the orientation by $\Z[\gamma, \sqrt{p}]$ is typically primitive, but recall from Remark~\ref{rmk:sanon_break} that it is not crucial to know this for sure, as long as the security is based on the maximal order in view of potential hidden stabilizer attacks, as well as to avoid potential ``walking up the volcano" attacks (see Section~\ref{sec:security}). In any case the relevant discriminant is of order $p^2 a^2$ and,
since we expect $h^+(\Z[\sqrt{p}])$ to be very small (see the footnote to Corollary~\ref{cor:orientation_finding}), in view of the Brauer--Siegel theorem it is fair to estimate the size of the Shimura class group as $p|a|$.

\begin{remark}
In particular this is unbounded: by letting $|a|$ grow we can reach arbitrarily large group sizes while working exclusively with varieties defined over $\F_p$; this is impossible in dimension $g = 1$. This being said, as soon as $|a| \gtrapprox p^2$ we necessarily have $N \gtrapprox p^2$ and $p-1$ can no longer host $M = \sqrt{N}$ as a factor, which creates trouble for representing the orientation $\iota$ using meet-in-the-middle.
One can circumvent this by using smaller values of $M$ and publishing auxiliary PPASs as part of the orientation data. 
E.g., one can
let 
$N = M^3$ be a cube and choose parameters such that $M \mid p - 1$. To represent $\iota(\gamma)$, note that our $N$-isogeny $\iota(\gamma) + \alpha$ necessarily factors as a composition $\kappa_3 \circ \kappa_2 \circ \kappa_1$ of $M$-isogenies:
\[ (A, \lambda) \stackrel{\kappa_1}{\longrightarrow} (B, \lambda_B) \stackrel{\kappa_2}{\longrightarrow} (C, \lambda_C) \stackrel{\hat \kappa_3}{\longleftarrow} (A, \lambda). \]
If we provide interpolation data on $A[M]$ then this is enough for recovering the outer isogenies in this diagram. We can then append $\ker \kappa_2 \subset B[M]$ to our description of $(A, \lambda, \iota)$ as a PPAS. When evaluating the action, this whole diagram can be pushed through. This is somewhat reminiscent of the technique used in~\cite{Houben}, but a full generalization of the method from~\cite{Houben} is left for future research.
\end{remark}

\subsubsection*{Concrete parameters.} We have generated parameters aiming for a Shimura class group size of roughly $2048$ bits, which should be compared with CSIDH-$4096$, aiming for NIST-I security. 
We choose $|a| \approx N \approx p$ to maximize the room available in $p-1$ for packing small primes $\ell$ (but this also maximizes the field size, so other choices can be made here). We must in fact take $N$ slightly larger than $p$ to create leeway for the generation of corresponding R2-D4 parameters. Within our range of interest, it turns out that $N = 3^{2f}$ with $f = 333$ admits a more-than-average number of small primes $\ell$ modulo which $9-N$ and $1-N$ are non-zero squares. Concretely we can use the following $56$ primes $\ell_i$
\begin{multline*}
    23, 43, 53, 61, 67, 97, 101, 103, 173, 197, 211, 227, 233, 257, 277, 283, \\ 293, 
 307, 313, 359, 373, 383, 397, 419, 433, 443, 457, 491, 521, 563, 577, \\ 601, 613, 
 673, 683, 691, 733, 743, 769, 787, 809, 811, 859, 883, 887,  911,  \\ 977,  983, 1069, 
 1151, 1163, 1187, 1193, 1237, 1249, 1259,
\end{multline*}
let $p = 2 \cdot 3^{f + 2} \cdot 7 \cdot \ell_1 \cdots \ell_{56} + 1$, and then set $a = p + 4 - N$ as required (observe the extra factor $3^2$ which is convenient for gluing in the middle). The respective bit lengths of $p$ and $|a|, N$ are $1021$ and $1056$, making for a Shimura class group estimate of $2077$ bits. Using trial division we checked that $(a^2 - 16p)/36$ contains no squares of primes smaller than $2^{40}$, so this discriminant factor is most likely square-free, but we could not check this rigorously.\footnote{One way to get around this is by randomizing the cofactor $c$, used to construct $p$, until $a^2 - 16p$ is easy to factor (i.e., a prime number times a smooth cofactor).} Concrete R2-D4 parameters can be found with $a_2 = 1758$, $a_3 = 0$, $b_3 = 565$ so that $n = a_2^2 + b_3^2 = 3409789$, which is prime. The reader can then verify that also $-4na - n^2p - 64$ is prime and representable by the form $X^2 + 4nY^2$, eventually yielding
\begin{alignat*}{2}
a_1 &= {} & &72447886291797740848322932962051593473734748284190344 \\
    &     & &65405345610109441454835644829643664359635439561189660 \\
    &     & &15384207541540878628641645549368269105214510733042954 \\[4pt]
b_1 &= {} & -&23283876993666509430775004052081427936666742195999741 \\
    &     & &02931752144318449614324311336034511014331071303795311 \\
    &     & &71042137236046983177009402579859540412085437180984222 \\[4pt]
b_2 &= {} & &23122390115850551190499797644857286144606472585743281 \\
    &     & &92024567734772318935855806010161778248613700364778186 \\
    &     & &274082288575193881983650751959251785841361677700367
\end{alignat*}
for the remaining R2-D4 parameters. These integers are roughly of size $3^f$.

\subsubsection{Isogeny computations involved.} Recall the count from~\eqref{eq:keys}: to reach a key space size having roughly $256$ bits, we want
    \[  (2B + 1)^{2 \cdot 56} \approx 2^{256}, \]
showing that $B = 2$ suffices. This means that each arrow in Figure~\ref{fig:key-exchange-diagram_general} amounts to at most $2 \cdot 56 \cdot 2 = 224$ actions by Shimura ideals of the form $(\fraka, \ell_i)$. Each such action involves the computation of $2$ isogenies of reduced degree $3^f$ that ``meet in the middle", plus some linear algebra to determine the kernel, followed by an isogeny of reduced degree $\ell_i$. All these isogenies are between PPAVs of dimension $2$.

\subsection{Key exchange from R2-D2} \label{app:R2D2_keyexchange}

We mention once again that the vectorization problem from the R2-D2 orientation, as it was introduced in Example~\ref{ex:R2D2}, can be reduced to vectorization problems for the ideal-class groups of the underlying imaginary quadratic rings $\Z[\sqrt{-p}]$ and $\Z[\sqrt{-dp}]$ and therefore does not reach the security level one would hope~\cite{C2C2attack}. We nevertheless include a quick, high-level discussion. On one hand, we do this for the sake of context for~\cite{C2C2attack}. On the other hand, it allows us to highlight an interesting connection between R2-D2 orientations and the ``higher-degree" supersingular group actions studied by Chenu--Smith~\cite{ChenuSmith},  used in Orient Express~\cite{orientexpress}. These concern actions by $\Cl(\Z[\sqrt{-dp}])$ on the set of $\Z[\sqrt{-dp}]$-oriented elliptic curves $(E, \iota)$.

\subsubsection{Generators.} Following Section~\ref{subsec:genset}
the generating set $\mathbf{g}$ is built from the group homomorphism $\kappa_R : \Cl(\Z[\sqrt{-dp}]) \times \Cl(\Z[\sqrt{-p}]) \to \frakC(R)$.
That is, we pick a set of ``easy" ideal classes 
\[ [\frakm_1 ], [\frakm_2], \ldots, [\frakm_s] \in \Cl(\Z[\sqrt{-dp}]), \qquad [\frakl_1 ], [\frakl_2], \ldots, [\frakl_t] \in \Cl(\Z[\sqrt{-p}])  \]
and let $\mathbf{g}$ be the collection of images of $([\frakm_i], 1)$, $i = 1, \ldots, s$ and $(1, [\frakl_i])$, $i = 1, \ldots, t$. Abusing language, we will say that we act with $[\frakm_i]$ when in fact acting with $\kappa_R([\frakm_i], 1)$, and likewise for $[\frakl_i]$. 
The ideals $\frakl_i$ are chosen as in CSIDH. That is, we set things up such that $\ell_1 \ell_2 \cdots \ell_t \mid p + 1$ for distinct small odd prime numbers $\ell_i$ and we let $\frakl_i = (\ell_i, \sqrt{-p} - 1)$. Then $\bar \frakl_i = (\ell_i, \sqrt{-p} + 1)$ is a representative of $[\frakl_i]^{-1}$. For the ideals $\frakm_i$, we  consider distinct small odd prime numbers $m_i$ such that $m_1 m_2 \cdots m_s \mid p + 1$ and such that for all $i = 1, 2, \ldots, s$ there exists $a_i \in \Z$ such that $d \equiv a_i^2 \bmod m_i$. We then let $\frakm_i = (m_i, \sqrt{-dp} - a_i)$ so that $[\frakm_i]^{-1}$ is represented by $\bar \frakm_i = (m_i, \sqrt{-dp} + a_i)$. The sets $\{ \ell_1,  \ldots, \ell_t \}$ and $\{ m_1, \ldots, m_s \}$ need not be disjoint.

\begin{remark}
  Recall from Lemma~\ref{lem:almost_iso} that $\kappa_R$ is ``almost" an isomorphism. See Example~\ref{ex:shimuraC2C2} for a specialization of this lemma to R2-D2.
\end{remark}

\subsubsection{Parameters.} The data needed for setting up R2-D2 are in accordance with those needed for the ``square-free" case of Orient Express~\cite{orientexpress}. There too, one looks for a supersingular elliptic curve $E / \mathbb{F}_{p^2}$ on which $\pi_p^2 = -p$, admitting a degree-$d$ isogeny
\[\tau : E \to E^{(p)} \] such that $\hat \tau \circ \pi_p = \pi_p \circ \tau$, and such that many small odd primes $\ell$ divide $p+1$ and split in $\Z[\sqrt{-dp}]$. Therefore, we can recycle the parameters from~\cite[\S5]{orientexpress}.

\subsubsection{Acting on the starting surface.}
Interestingly, the higher-degree supersingular group action can be ``embedded" in a Shimura class group action.
Indeed, given an elliptic curve $(E, \iota)$ that is oriented by $\Z[\sqrt{-dp}]$, it is not hard to show that $\iota(\sqrt{-dp})$ factors as
   \[ E \stackrel{\tau}{\longrightarrow} E^{(p)} \stackrel{\pi_p}{\longrightarrow} E \]
   for a degree-$d$ isogeny $\tau$ satisfying $\hat \tau \circ \pi_p = \pi_p \circ \tau$~\cite[Lem.\,1]{orientexpress}, so we can build an $R$-oriented PPAS from this exactly as in R2-D2. 
   In view of the Weil restriction involved, let us denote this surface as $W(E, \iota) \in \M_{\overline{\F}_p}(R)$. It is then easy to verify the rule
   \[ \kappa_R([\frakm], 1) \star W(E, \iota) = W([\frakm] \star (E, \iota)), \]
   whence the said embedding.
   This means that during public-key generation, acting with the $[\frakm_i]$'s can be done at the cost of an Orient Express action, as claimed. 
   In particular, this evaluation only uses isogenies between elliptic curves over $\F_{p^2}$ (although Orient Express uses $2$-isogenies between abelian surfaces for interpolating $\tau$). 
   Subsequently acting with the $[\frakl_i]$'s is similar to CSIDH and also very efficient, since the orientation representation comes for free.

\subsubsection{Acting on a public key.}
   But acting with the $[\frakm_i]$'s on a public key, which is not necessarily a Weil restriction, requires us to switch to the machinery from Section~\ref{subsec:representation}. The parameters in Orient Express are such that $2^{2a} - d = \alpha^2$ for certain integers $a, \alpha$. This means that the $d$-isogeny $\iota(\sqrt{d})$ can be represented in dimension $2ug = 2 \cdot 1 \cdot 2 = 4$ using Theorem~\ref{thm:embedding}. Unfortunately, here the dimension-halving trick from Section~\ref{subsec:representation} does not seem to work. The asymmetry between both types of actions is summarized in Figure~\ref{fig:key-exchange-diagram}.

\begin{figure}[h]
\centering
\begin{tikzcd}[column sep=15pt, row sep=small]
A_0 \arrow[rrr, blue, "{\frakm_\Ac \subset \mathbb{Z}[\sqrt{-dp}]}"] 
\arrow[dd, blue, "{\frakm_\Bc \subset \mathbb{Z}[\sqrt{-dp}]}" description] 
&  &  & 
A_{\mathcal{A}}^{dp} \arrow[rrr, teal, "{\frakl_\Ac \subset \mathbb{Z}[\sqrt{-p}]}"]   
&  &  & 
A_{\mathcal{A}} \arrow[dd, orange, "{\frakm_\Bc \subset \mathbb{Z}[\sqrt{-dp}]}" description]
\\
&  &  & 
&  &  &                                         
\\
A_{\mathcal{B}}^{dp} \arrow[dd, teal, "{\frakl_\Bc \subset \mathbb{Z}[\sqrt{-p}]}" description]            
&  &  &
&  &  & 
A_{\mathcal{BA}}^{dp} \arrow[dd, teal, "{\frakl_\Bc \subset \mathbb{Z}[\sqrt{-p}]}" description] 
\\
&  &  & 
&  &  &                                         
\\
A_{\mathcal{B}} \arrow[rrr, orange, "{\frakm_\Ac \subset \mathbb{Z}[\sqrt{-dp}]}"']        
&  &  & 
A_{\mathcal{AB}}^{dp} \arrow[rrr, teal, "{\frakl_\Ac \subset \mathbb{Z}[\sqrt{-p}]}"'] 
&  &  & 
A_{\mathcal{AB}}                           
\end{tikzcd}
\caption{An illustration of the R2-D2 key exchange. 
Blue arrows represent actions computed via 1-dimensional isogenies, with a 2D representation of the orientation; green arrows represent actions computed via 2-dimensional isogenies; orange arrows represent actions computed via 2-dimensional isogenies, with a 4D representation of the orientation. Polarizations and orientations are omitted to lighten the notation.}
\label{fig:key-exchange-diagram}
\end{figure}
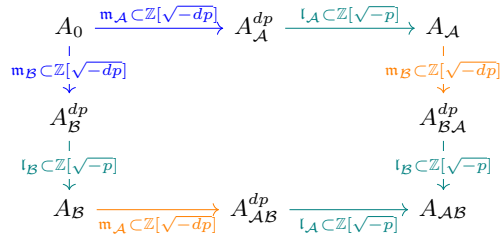

\section{Compressing interpolation data for R2} \label{app:compress}

In this appendix we discuss the problem of compressing the interpolation data $P, \iota(\gamma)(P)$, with $P$ iterating over a basis of $A[\sqrt{N}]$, in the context of the R2 orientations. To ease notation, let us write $M = \sqrt{N}$. We also assume $M \mid p \pm 1$ where the sign depends on whether we are working in the R2-D2 setting ($+$) or the R2-D4, R2-C4 settings ($-$). In view of Remark~\ref{rmk:props_of_Fp_orientations} this ensures that all interpolation points are defined over $\F_{p^2}$.

Since $\iota(\gamma)$ is defined over $\F_p$ (indeed, it commutes with $\Pi_p$), it respects eigenvectors of Frobenius. In particular it maps the $1$-eigenspace of Frobenius
\[ A[M](\F_p) \cong \left( \frac{\Z}{M\Z} \right)^2 \]
into itself, and likewise for the $-1$-eigenspace (which has the same group structure). To use this for compression purposes, let us first recall that a ``symplectic basis" of $A[M]$ is a basis $P_1, P_2, Q_1, Q_2$ such that $e_M(P_i, Q_j)$ equals $\zeta_M$ if $i = j$ and $1$ if $i \neq j$, where $\zeta_M \in \F_{p^2}$
denotes some fixed primitive $M$th root of unity, whereas $e_M(P_i, P_j) = e_M(Q_i, Q_j) = 1$ for all $i, j$.
Assuming $M$ is odd (the even case is slightly more technical; note that $M$ is odd in our protocol from Appendix~\ref{app:R2D4_key_exch}), we make the following case distinction:
\begin{itemize}
    \item If $M \mid p + 1$ as in R2-D2 then it is possible to choose a symplectic basis of $A[M]$ such that $\langle P_1, P_2 \rangle$ is the $1$-eigenspace and $\langle Q_1, Q_2 \rangle$ is the $-1$-eigenspace.
    \item If $M \mid p - 1$ as in R2-C4 and R2-D4 then one can choose it such that $\langle P_1, Q_1 \rangle$ is the $1$-eigenspace and $\langle P_2, Q_2 \rangle$ is the $-1$-eigenspace.
\end{itemize}
The second case where $\Pi^2 = p$ follows from~\cite[Lem.\,3.17 and Lem.\,5.1]{supersingularitytesting}. The first case where $\Pi^2 = -p$ is easier and well-known; e.g., it follows along the same lines as~\cite[Lem.\,3.17]{supersingularitytesting}. Henceforth, we assume that $P_1, P_2, Q_1, Q_2$ is a deterministically generated basis of the above kind. 

It now suffices to describe $\iota(\gamma)(P_i)$, $i = 1, 2$ by means of $4$ coefficients $s, t, u, v \in \Z /M\Z$. E.g., in the case of R2-D2, these indicate that
\[ \iota(\gamma)(P_1) = sP_1 + tP_2, \qquad \iota(\gamma)(P_2) = uP_1 + vP_2. \]
To expand the representation: we know that
\[ \iota(\gamma)(Q_1) = xQ_1 + yQ_2, \qquad \iota(\gamma)(Q_2) = zQ_1 + wQ_2 \]
for $x, y, z, w \in \Z / M \Z$ that can be determined through Weil pairing and discrete logarithm computations (in practice $M$ is smooth so this can be done efficiently). Concretely we find
\[  sx + ty = \log_{\zeta_M} e_M( \iota(\gamma) (P_1), \iota(\gamma)(Q_1)) = \log_{\zeta_M} e_M( P_1, \iota(d)(Q_1)),  \]
where we have used that $\iota(\gamma)^\dagger \iota(\gamma) = \iota(\bar \gamma \gamma) = \iota(d)$ using the notation from Section~\ref{subsec:representation}, and similarly obtain $ux + vy$ from which we can extract $x, y$. We similarly find $z, w$ and conclude.

\section{Proofs from Section~\ref{sec:algebra}} \label{app:algebraproofs}

\subsection{The Shimura class group action via matrix arithmetic} \label{app:direct_action}

Let $g, a, b_1, b_2 \in M_2(\O)$ be algebraic data associated with an oriented superspecial PPAS, as in Section~\ref{sec:algebra}.
Let $(\fraka, \mu)$ be a Shimura ideal of $R = \Z[a, b_1, b_2]$. For $\gamma$ a generator of the left $M_2(\O)$-ideal generated by $\fraka$, we consider
\[ g' := \gamma^{-\ast} g \mu \gamma^{-1}, \quad a' := \gamma a \gamma^{-1}, \quad b_i' := \gamma b_i \gamma^{-1}, i = 1, 2 \]
as in the proof of Theorem~\ref{thm:algN}. We now give a direct argumentation, essentially using matrix arithmetic only, for why this is a well-defined group action.

\subsubsection{Why $g'$ is a valid $g$-matrix.} In Remark~\ref{rmk:verifygmatrix} we already explained why $g'$ is positive-definite and self-conjugate of norm $1$. So in order to prove that it concerns a valid $g$-matrix, we need to show that $g'\in M_2(\mathcal{O})$. This claim is equivalent to $g^{-1}g'\in M_2(\mathcal{O})$ as $g$ is an invertible element in $M_2(\mathcal{O})$. This can be rewritten as 
$\gamma^{-\dagger}\mu\in M_2(\mathcal{O})\gamma$.

The ideal $\fraka$ can be generated by $r\in R$ and $\mu$ with the property that $r^{\dagger}r=r_0\mu$ where the norm of $r_0$ is coprime to the norm of $\fraka$. Since $\gamma$ is the generator of the extended ideal one has that $s\gamma=r$ for some $s\in M_2(\mathcal{O})$ which can be rewritten as $\gamma=s^{-1}r$. 
 Note that $\gamma^{-\dagger}\mu\in M_2(\mathcal{O})$ by the Shimura property of the ideal. So we only need to show that $\gamma^{-\dagger}\mu$ vanishes on the kernel of $\fraka$ as then by definition it belongs to the left ideal generated by $\gamma$. One has 
$$\gamma^{-\dagger}\mu=s^{\dagger}r^{-\dagger}\mu=s^{\dagger}r\mu^{-1}r_0^{-1}\mu=s^{\dagger}rr_0^{-1}=s^{\dagger}r_0^{-1}r.$$
Here we used the identity $r^{\dagger}r=r_0\mu$ as $r^{-\dagger}=r\mu^{-1}r_0^{-1}$. 
This shows that indeed the above element vanishes on the ideal kernel by the properties of $r_0$.

\subsubsection{Why $g', a', b_1', b_2'$ are valid algebraic data.} Since conjugation by any invertible element is an automorphism of $M_2(B_{p, \infty})$, one has that  $\gamma a\gamma^{-1}$, $\gamma b_1\gamma^{-1}$, $\gamma b_2\gamma^{-1}$ generate a CM order isomorphic to $R = \Z[a,b_1, b_2]$
but we also have to show that Rosati acts as conjugation on them and we also have to show that both lie in $M_2(\mathcal{O})$. We will do the calculation for $a$ but the calculation for the $b_i$'s is identical.

We need that
$(\gamma^{-*}g\mu\gamma^{-1})^{-1}(\gamma a\gamma^{-1})^*(\gamma^{-*}g\mu\gamma^{-1})=\gamma a \gamma^{-1}$. The left-hand side can be rewritten as
 $$(\gamma \mu^{-1} g^{-1}\gamma^*)(\gamma^{-*} a^*\gamma^{*})(\gamma^{-*}g\mu\gamma^{-1})=\gamma\mu^{-1}(g^{-1}a^*g)\mu\gamma^{-1}.$$
Finally, $\gamma \mu^{-1}(g^{-1}a^*g)\mu\gamma^{-1}=\gamma a\gamma^{-1}$ as $a$ was symmetric with respect to Rosati and $a$ and $\mu$ commute.

What remains is that $\gamma a\gamma^{-1}$ and $\gamma b_i\gamma^{-1}$ are indeed in $M_2(\mathcal{O})$. We show it again only for $a$ as the other cases are very similar. The condition $\gamma a\gamma^{-1}\in M_2(\mathcal{O})$ is equivalent to $\gamma a\in M_2(\mathcal{O})\gamma$ which is the left ideal generated by $\gamma$. The left ideal generated by $\gamma$ consists of exactly those endomorphisms that vanish on the kernel of $\fraka$, so we need to show that $\gamma a$ vanishes on the kernel of $\fraka$. Now $\gamma=s^{-1}r$, so $\gamma a=s^{-1}ra=s^{-1}ar$. Now since the  reduced norm of $s$ is coprime to the norm of $\fraka$ and $r$ is in $\fraka$, $\gamma a$ will indeed vanish on the kernel of $\fraka$ (the proof carries over to $b$, as we have only used the fact that $a$ and $r$ commute). 

\subsection{Omitted proofs} \label{app:omitted_algebra_proofs}  

\begin{proof}[of Lemma~\ref{lem:quaternion_cholesky}]
  Writing
  \[ g = \begin{pmatrix} s & r \\ \bar{r} & t \end{pmatrix} \]
  with $s, t \in \Q_{>0}$ and $r \in B_{p, \infty}$, we can decompose this matrix as
  \[ g = \gamma \begin{pmatrix} s & 0 \\ 0 & t' \end{pmatrix} \gamma^\ast, \qquad \gamma = \begin{pmatrix} 1 & 0 \\ s^{-1} \bar{r} & 1 \end{pmatrix}, \]
  with $t' = t - s^{-1} \nrm(r)$. We have $t' > 0$ by positive-definiteness. By the Hasse--Schilling theorem~\cite[Thm.\,14.7.4]{voight2021quaternion} we can write $s = \sigma \bar{\sigma}$ and $t' = \tau \bar{\tau}$ for certain quaternions $\sigma, \tau$. It follows that
  \[ h = \gamma \begin{pmatrix} \sigma & 0 \\ 0 & \tau \end{pmatrix} \]
  is of the requested kind. \qed
\end{proof}

\begin{proof}[of Lemma~\ref{lem: cent}]
First we compute what the centralizer of $a$ is. Since the centralizer is a purely algebraic object, it is enough to prove it for matrices of the form 
\begin{equation} \label{eq:enough} \begin{pmatrix}
    0&s\\
    1& 0
\end{pmatrix} 
\end{equation}
where $s\in \mathbb{Z}_{>0}$.
Indeed, quadratic fields can be generated by trace-$0$ elements and all pairwise isomorphic quadratic subfields of $M_2(B_{p,\infty})$ are conjugate to each other by the Noether-Skolem theorem. The centralizer of~\eqref{eq:enough} will be the set of matrices of the form
\[ \begin{pmatrix}
    x&sy\\
    y& x
\end{pmatrix} \]
where $x,y\in B_{p,\infty}$. This naturally decomposes as the tensor product of diagonal matrices (isomorphic to $B_{p,\infty}$) and $\mathbb{Q}(a)$. 

Then we observe that if $b\in M_2(B_{p,\infty})$ commutes with $a$, then so does $b^{\dagger}$ (its Rosati involution). Indeed, if $ab=ba$, then $b^{\dagger}a^{\dagger}=a^{\dagger}b^{\dagger}$ and then $b^{\dagger}a=ab^{\dagger}$ as $a^{\dagger}=a$. This shows that Rosati induces an involution on $C=\mathbb{Q}(a)\otimes B_{p,\infty}$ but we also need a stronger statement. Observe that for any element $c\in C$, the element $c^{\dagger}c$ is symmetric with respect to Rosati. Due to the positivity of the Rosati involution, $c^{\dagger}c$ will generate a totally real quadratic subfield of $C$. However, since $C$ is isomorphic to $\mathbb{Q}(a)\otimes B_{p,\infty}$, the only such subfield is the center. Thus Rosati defines a standard involution on $C$ (i.e., for every $c\in C$, the element $c^{\dagger}c$ lies in the center \cite[Rmk.\,3.2.5]{voight2021quaternion}) and for quaternion algebras there exists a unique such involution which is quaternion conjugation. \qed
\end{proof}

\begin{proof}[of Lemma~\ref{lemma:split}]
The decision version is folklore (e.g., it follows from \cite[\S6]{ivanyos1996algorithms}), hence we focus on the embedding problem. First, we compute an explicit isomorphism between $C\otimes_K K(\sqrt{d})$ and $M_2(K(\sqrt{d}))$ in quantum polynomial time~\cite{kutas_montessinos}. Using this isomorphism we can find an element $s\in C\otimes_K K(\sqrt{d}) $ for which $s^2=0$. Now if we write 
$$s=(s_1+r_1\sqrt{d})+(s_2+r_2\sqrt{d})i+(s_3+r_3\sqrt{d})j+(s_4+r_4\sqrt{d})k$$
then we get that $s_1=r_1=0$ as its trace is 0 and if we square it and look at the rational part we get
$$\nrm(s_2i+s_3j+s_4k)=-d\nrm(r_2i+r_3j+r_4k)$$
which implies that the norm of $(s_2i+s_3j+s_4k)/(r_2i+r_3j+r_4k)$ equals $-d$. Now we only have to show that the trace of this element is 0. The trace of this element is 
\[ -2\frac{s_2r_2i^2+s_3r_3j^2+s_4r_4k^2} {\nrm(r_2i+r_3j+r_4k)} \] which is 0 as the numerator equals the coefficient of $\sqrt{d}$ in $s^2$. This proves our lemma. \qed
\end{proof}

\begin{proof}[of Lemma~\ref{lem:conductor_lemma}]
  We first note that $\xi, \gamma$ indeed exist, respectively because of Lemma~\ref{lem:quaternion_cholesky} and the fact that $M_2(\O)$ is a principal ideal ring.
  Next we claim that the conductor ideal can be characterized as
  \[ \frakc = \{ \, \alpha \in K \, | \, \alpha \gamma^\ast \in \gamma^{-1} M_2(\O) \}. \]
  To see this: it is clear that $\gamma^{-1} M_2(\O)$ can be characterized as the set of
  matrices $x \in M_2(B_{p, \infty})$ such that $\Z_K x \subset M_2(\O)$. Let $\alpha \in K$. Using that $\gamma^\ast M_2(\O) = (M_2(\O) \gamma)^\ast = (M_2(\O) \Z_K)^\ast = \Z_K M_2(\O)$ along with the commutativity of $K$ we see that
  $\Z_K \alpha \gamma^\ast M_2(\O) = \alpha \Z_K^2 M_2(\O) = \alpha \Z_K M_2(\O)$,
  whence the claim.

  As a consequence we have $\kappa := \gamma \mu \gamma^\ast \in M_2(\O)$. Below we show that in fact $\kappa \in \GL_2(\O)$. Assuming this, the lemma follows fairly easily. Indeed, let $\Lambda' = \gamma^{-1} M_2(\O) \gamma$ be the right order of $M_2(\O) \gamma = M_2(\O) \Z_K$ and observe using $\Z_K^2 = \Z_K$ that this right order contains $\Z_K$. Then its conjugate $\Lambda$ contains $\xi^{-1} \Z_K \xi$
  and the computation
  \[ \Lambda^\ast = \xi^\ast \gamma^\ast M_2(\O) \gamma^{-\ast} \xi^{-\ast} = \xi^{-1} \gamma^{-1} \kappa M_2(\O) \kappa^{-1} \gamma \xi = \xi^{-1} \gamma^{-1}  M_2(\O)  \gamma \xi = \Lambda \]
  shows that it is closed under conjugate-transpose. The fact that $\iota(\bar \alpha) = \iota(\alpha)^\ast$ for all $\alpha \in \Z_K$ follows from $(\xi^{-1} \alpha \xi)^\ast = \xi^\ast \alpha^\ast \xi^{-\ast} = \xi^{-1} \alpha^\ast \xi$ for all $\alpha \in K$, where in turn the last equality follows from the fact that $\xi \xi^\ast = \mu$ commutes with $\alpha$.

  Proving that $\kappa \in \GL_2(\O)$ amounts to showing $\nrm(\gamma \mu \gamma^\ast) =  1$. 
  To establish this, it suffices to produce for every prime number $\ell \mid \nrm(\mu)$ 
  some element $\mu_\ell \in \frakc$ such that $\ell \nmid \nrm(\gamma \mu_\ell \gamma^\ast)$. It is indeed allowed to skip the primes  $\ell \nmid \nrm(\mu)$ (which include $p$ by assumption) because in those cases $\Z_K \cap M_2(\O)$ is locally maximal at $\ell$, implying that also $\nrm(\gamma)$ must be an $\ell$-adic unit, and therefore $\ell \nmid \nrm(\gamma \mu \gamma^\ast)$. 
  
  Instead of finding $\mu_\ell$ directly, we will produce an element $\nu_\ell \in K_\ell$ 
  such that $\gamma \nu_\ell \gamma^\ast \in \GL_2(\O_\ell)$, where $K_\ell = K \otimes_{\Z} \Z_\ell, \O_\ell = \O \otimes_{\Z} \Z_\ell$.
  This is good enough indeed, because if we then pick $\nu_\ell' \in K$ with the same $\frakl$-adic valuation as $\nu_\ell$ for every non-zero prime ideal $\frakl \, \mid \, \ell \Z_K$ (which is always possible by the Chinese remainder theorem), then also $\gamma \nu_\ell' \gamma^\ast \in \GL_2(\O_\ell)$.
   It then suffices to let $\mu_\ell = n\nu'_\ell$ for a suitable integer $n$; e.g., we can consider any positive integer $m$ such that $m \nu_\ell'$ is an integral multiple of $\mu$ and choose $n$ such that $n \equiv 1 \bmod \ell$ and $n \equiv m$ modulo all prime divisors of $m$ different from $\ell$.

  Recall that $\Z_K$ is contained in the right order of $M_2(\O)\gamma$, which implies that $\gamma^{-1} M_2(\O)$ is closed under left-multiplication with elements from $\Z_K$.
  It is then also true that $\gamma^{-1} M_2(\O_\ell)$ is closed under left-multiplication by $\Z_{K, \ell} = \Z_K \otimes_{\Z} \Z_{\ell}$. 
  Fix an isomorphism $B_{p, \infty} \otimes_{\Z} \Z_\ell \cong M_2(\Q_\ell)$, which through component-wise application also yields an isomorphism $M_2(B_{p, \infty} \otimes_{\Z} \Z_\ell) \cong M_4(\Q_\ell)$. We can then consider the $\Z_\ell$-lattice
  \begin{equation} \label{eq:def_latticeL}
     L = \gamma^{-1} L_0, \qquad L_0 = \Z_\ell^4
  \end{equation}
  by viewing the elements of $L_0$ as column vectors. Observe that this is again closed under left-multiplication by $\Z_{K, \ell}$. This gives $L$ the structure of a $\Z_{K, \ell}$-module. A crucial observation is that this module is free of rank $1$. To see why, consider the module
  \[ \gamma^{-1} \Q_\ell^4 = \bigoplus_{\frakl \, | \, \ell \Z_K} M_{\frakl} \qquad \text{over} \qquad K_\ell \cong \prod_{\substack{\text{non-zero prime} \\ \text{ideals } \frakl \, \mid \, \ell \Z_K}} K_{\frakl} \]
  where $K_{\frakl}$ denotes the  completion of $K$ at $\frakl$ and the decomposition on the left is obtained by acting on the whole module with subsets of $K_\ell$ of the form $\{ 0 \} \cdots \times K_{\frakl} \times \cdots \times \{ 0 \}$. Each $M_\frakl$ is a vector space over $K_\frakl$ and, since  both $K_\ell$ and $\gamma^{-1} \Q_\ell^4$ have dimension $4$ as vector spaces over $\Q_\ell$, we have $\dim_{K_\frakl} M_\frakl = 1$ for all $\frakl$, i.e., $\gamma^{-1} \Q_\ell^4$ is a free $K_\ell$-module of rank $1$. Switching to the ``integral side", namely the module
\[ L = \gamma^{-1} \Z_\ell^4 = \bigoplus_{\frakl \, | \, \ell \Z_K} L_{\frakl} \qquad \text{over} \qquad \Z_{K, \ell} \cong \prod_{ \frakl \, \mid \, \ell \Z_K} \Z_{K, \frakl}, \]
and using that the $\Z_{K, \frakl}$'s are PIDs, we can deduce the same desired statement: $L$ is a free rank-$1$ module over $\Z_{K, \ell}$.
In other words
  \begin{equation} \label{eq:lattice_as_module}  L = \Z_{K, \ell} v 
  \end{equation}
for some $v \in \Q_\ell^4$.

  We now consider a $\Q_\ell$-bilinear form $\langle \cdot, \cdot \rangle$ on $\Q_\ell^4$ such that $\langle \mu v, w \rangle = \langle v, \mu^\ast w \rangle$ for all $v, w \in \Q_\ell^4$ and all $\mu \in M_2(B_{p, \infty})$. While such a bilinear form exists for general reasons~\cite[\S1]{knus1998book}, let us make this explicit by noting that
  \[ \bar{a} = H^{-1} a^T H, \qquad H = \begin{pmatrix} 0 & 1 \\ -1 & 0 \end{pmatrix}, \]
  independently of the chosen isomorphism:
 this merely uses that $\bar{a} = \tr(a) - a$. We can then let $\langle v, w \rangle = v^T \diag(H, H) w$.
  Then using $\langle \gamma^\ast v, \gamma^{-1} w \rangle = \langle v, \gamma \gamma^{-1} w \rangle = \langle v, w \rangle$ for all $v, w \in \Q_\ell^4$
  we find that the dual lattice of $L$ with respect to this form is given by
  \begin{equation} \label{eq:Lstar1} 
   L^\ast = \{ \, v \in \Q_\ell^4 \, | \, \langle v, L \rangle \subset \Z_\ell \, \} = \gamma^\ast L_0.
  \end{equation}
  But on the other hand, through~\eqref{eq:lattice_as_module} the form also induces a bilinear map
  \[ K_\ell \times K_\ell \to \Q_\ell : (r, s) \mapsto  \langle rv, sv \rangle = \langle v, \bar{r} s v \rangle.  \]
  Note that the right-hand side is a $\Q_\ell$-linear expression in $\bar r s$. It is a standard property of the trace function that there must exist some $\delta \in K_\ell^\times$ such that 
  \[ \langle v, \bar{r} s v \rangle = \Tr_{K / \Q}(\delta \bar{r} s), \]
  where we have extended the domain of $\Tr_{K/\Q}$ to $K_\ell$ by $\Z_\ell$-linearity.
  Indeed, this only requires that the trace pairing remains non-degenerate on $K_\ell$, which is immediate.
  Thus from this perspective we can also write
  \begin{equation} \label{eq:Lstar2} L^\ast = \delta^{-1} \mathfrak{D}^{-1} v 
  \end{equation}
  with $\mathfrak{D}$ the ideal of $\Z_{K, \ell}$ generated by the different ideal of $K$. Now $\Z_{K, \ell}$ is a principal ideal ring, again because $\Z_{K, \ell}$ is a product of PIDs, so we can rewrite this as $L^\ast = \Z_{K, \ell} \nu_\ell^{-1} v$ for some generator $\nu_\ell$ of $\delta \mathfrak{D}$. As the notation suggests, this is the element we are after. Indeed, comparing expressions~\eqref{eq:def_latticeL},~\eqref{eq:lattice_as_module},~\eqref{eq:Lstar1} and~\eqref{eq:Lstar2} we see that $\gamma^{-1} L_0 = L = \nu_\ell L^\ast = \nu_\ell \gamma^\ast L_0$
  from which it follows that 
  $\gamma \nu_\ell \gamma^\ast \in \GL_2(\O_\ell)$, as wanted.
  \qed
\end{proof}

\section{On the one endomorphism problem in dimension $2$} \label{app:oneend}

The machinery from Section~\ref{sec:algebra} allows us to prove $2$-dimensional analogues of the first two statements in Section~\ref{sec:algebra_appl}, which we recall here: 
\begin{itemize}
    \item Given one non-scalar endomorphism on a supersingular elliptic curve, one can compute the rest of the endomorphism ring in subexponential time \cite{DBLP:journals/cic/MerdyW25}.
    \item Given an endomorphism of small degree (polylogarithmic in  $p$) on a supersingular elliptic curve in characteristic $p$, one can compute the full endomorphism ring in polynomial time \cite{love2020supersingular}.
\end{itemize}
Our result strongly relies on Corollary~\ref{cor:orientation_finding}, which for technical reasons was restricted to cases where $R = \Z_K$ is a maximal order; for this reason we make the same restriction here.

\begin{theorem}\label{thm:hashingprecise}
Let $(A, \lambda, \iota)$ be a superspecial PPAS oriented by $\Z_K$ for a quartic CM field $K$ and assume that $h^+(K_0)$ is polynomially bounded. 
There exists a heuristic sub-exponential quantum algorithm that computes the endomorphism ring of $A$. If the discriminant of $K$ is polynomial in the input size, then the endomorphism ring of $A$ can be computed heuristically in classical polynomial time. 
\end{theorem}
\begin{proof}
As a first step we compute  a $\Z_K$-oriented superspecial PPAV $(A', \lambda', \iota')$ along with corresponding algebraic data $g, a, b_1, b_2$ using Corollary~\ref{cor:orientation_finding}, which is where the assumption on $h^+(K_0)$ comes into play.
Then one can solve vectorization using Kuperberg's algorithm between $(A, \lambda, \iota)$ and $(A', \lambda', \iota')$. This has no solution if these oriented PPASs lie in different orbits. In that case one can retry with $\iota'$ replaced with $\iota' \circ \tau$ for some automorphism $\tau$ of $K$ (which by Remark~\ref{rmk:frob_conj_orbit} amounts to taking the Frobenius conjugate in the case of complex conjugation), but if $(A, \lambda, \iota)$ and $(A', \lambda', \iota')$ come from inequivalent CM types then this brings no solace. In those cases we retry with a new $(A', \lambda', \iota')$ generated by Corollary~\ref{cor:orientation_finding}: it is a heuristic assumption that this reaches all CM types up to equivalence. Note that this step is never needed for $C_4$-fields, because here all CM types are equivalent by Lemma~\ref{lem:inequivalentCM}. In any case, after $O(1)$ tries we find $[(\fraka, \mu)] \in \frakC(\Z_K)$ such that $(A, \lambda, \iota) = [(\fraka, \mu)] \star (A', \lambda', \iota')$. Evaluating this action algebraically as described in Section~\ref{sec:actingIKO}, we end up with a $g$-matrix for $A$, from which its endomorphism ring can be extracted using~\cite{montessinos2026equivalent}.
If the size of the Shimura class group is polynomial in the input size, then instead of running Kuperberg one can simply compute this group and apply the action of every group element to $(A', \lambda', \iota')$ until $(A, \lambda, \iota)$ is reached. \qed
\end{proof}

\begin{remark}
The reader might wonder why we have a specific statement for small discriminant CM orientations. The reason is that if one obtains an oriented superspecial PPAS by using CM reduction, then one immediately knows an orientation by a relatively small CM order. The above theorem then shows that in this case one can compute the full endomorphism ring in polynomial time, generalizing a result of Love and Boneh \cite{love2020supersingular} to dimension $2$.
\end{remark}

\section{The Gorenstein and Bass conditions} \label{app:bass}

There are two standard ring-theoretic conditions on a CM order $R \subset \Z_K$ that enter the Shimura class group action: Gorenstein and Bass. Gorenstein is the (mild) hypothesis under which the characteristic-$0$ Shimura--Taniyama transitivity theorem is expected to hold (see the footnote to Theorem~\ref{thm:ST} in Section~\ref{subsec:CMchar0}). Bass is the hypothesis under which the ascending-walk picture mentioned in Section~\ref{sec:security} is available~\cite{arpin-marseglia-springer-isogeny-graphs}. Bass implies Gorenstein, but not conversely. This appendix records both notions, two sufficient conditions for Bass, and the index bound on the integer-polarized subgroup $\frakC(R)_\ZZ \subseteq \frakC(R)$.

\begin{definition}
\label{def:gorenstein}
An order $S$ in a number field is \defi{Gorenstein} if it has Cohen--Macaulay type $1$ at every maximal ideal $\frakm \subset S$, that is, $\dim_{S/\frakm}\bigl((S : \frakm) / S\bigr) = 1$ for every $\frakm$. Equivalently, every fractional $S$-ideal $\fraka$ with $(\fraka : \fraka) = S$ is invertible.
\end{definition}

The equivalence between the Cohen--Macaulay type condition and the invertibility of every proper fractional ideal is described in Marseglia~\cite[Proposition~3.4]{Marseglia2024cmtype}, where the two formulations go back to Bass~\cite{bass}.

If $R$ is Gorenstein, every proper fractional $R$-ideal is invertible; this is the setting in which $\frakC(R)$ acts on proper-ideal classes and in which Theorem~\ref{thm:ST} describes orbits. 

If $R$ is not Gorenstein, then there exist proper fractional $R$-ideals $\fraka$ with $(\fraka : \fraka) = R$ that are not invertible. In this case, the invertible-ideal action of $\frakC(R)$ does not cover these additional proper classes. Further, lifting non-transitivity on proper classes to a statement on PPAVs requires in addition that the extra classes admit a principal polarization of the prescribed CM type; the bimodule framework of~\cite{arpin-marseglia-springer-isogeny-graphs} treats the non-Gorenstein case.

Lemma~\ref{lem:Shimura_integer_norm_subgroup} shows that for any CM order $R$, every square in $\frakC(R)$ is integer-polarized, so $[\frakC(R) : \frakC(R)_\ZZ]$ divides $\#\frakC(R)[2]$. The Bass hypothesis gives this bound a graph-theoretic meaning on the crater and underlies the volcano picture used in Section~\ref{sec:security}.

\begin{definition} \label{def:bass-order}
An order $R$ in a number field $L$ is called \defi{Bass} if every over-order of $R$ is Gorenstein. Equivalently, every fractional $R$-ideal is generated as an $R$-module by at most two elements.
\end{definition}

\begin{remark}
    The equivalence between the two formulations is treated in~\cite[Prop.\,4.6]{Marseglia2024cmtype}. This covers the classical results of Bass~\cite{bass} and Drozd--Kirichenko--Roiter \cite{DKR67}. Two elementary sufficient conditions place $R$ in the Bass-class:
\begin{itemize}
    \item \emph{Index-$2$ subfield.} If $R$ contains the maximal order of a subfield $K' \subset K$ of index $2$, then $R$ is Bass. Maximal real multiplication ($R \supset \Z_{K_0}$) is the main case we use, see~\cite[\S2]{arpin-marseglia-springer-isogeny-graphs} and~\cite[Prop.\,4.6]{Marseglia2024cmtype}.
    \item \emph{Fourth-power-free discriminant.} If $\disc(R)$ is fourth-power-free in $\Z$, then $R$ is Bass, see \cite[Thm.\,3.6]{Greither1982}.
\end{itemize}
 
\end{remark}

We give an example of a non-Bass CM order.
\begin{example}
\label{ex:non-bass-cm-order}
Let $K = \Q(\zeta_5)$, a quartic CM field with $K_0 = \Q(\sqrt{5})$. Consider the two CM orders of conductor $2 \Z_K$
\[
    R = \Z + 2 \Z_K,\  R' = \Z_{K_0} + 2 \Z_K.
\]

Complex conjugation sends $\zeta_5$ to $\zeta_5^{-1} = \zeta_5^4$, so $\Z_K = \Z[\zeta_5]$ is stable under conjugation. The real subfield $K_0 = \Q(\sqrt 5)$ is pointwise fixed. Hence both $R$ and $R'$ are closed under complex conjugation. The prime $2$ is inert in $\Z_K$, so $\Z_K / 2 \Z_K \cong \F_{16}$, $R / 2 \Z_K \cong \F_2$ and $R' / 2 \Z_K \cong \F_4$. 
\end{example}

Nakayama's lemma (see e.g. \cite{Eisenbud1995}) gives \[ \mu_R(\Z_K) = \dim_{\F_2}(\Z_K / 2 \Z_K) = 4 > 2,\ \mu_{R'}(\Z_K) = \dim_{\F_4}(\Z_K / 2 \Z_K) = 2. \] By Definition~\ref{def:bass-order}, $R$ is not Bass and $R'$ is Bass. We write  \[ 2\Z_K = R\cdot 2 + R\cdot 2\zeta_5 + R\cdot 2\zeta_5^2 + R\cdot 2\zeta_5^3\] and explain why all four generators are necessary to represent $2\Z_K$. Write \[\Z_K = \langle 1,\zeta_5,\zeta_5^2,\zeta_5^3\rangle_{\Z}, \text{ and } 2\Z_K = \langle 2,2\zeta_5,2\zeta_5^2,2\zeta_5^3\rangle_{\Z}.\] Using $R=\Z+2\Z_K$, we obtain 
\begin{align*}
   R\cdot 2+R\cdot 2\zeta_5+R\cdot 2\zeta_5^2+R\cdot 2\zeta_5^3 &= \sum_{i=0}^{3}(\Z+2\Z_K)\cdot 2\zeta_5^i \\
   &= 2\sum_{i=0}^{3}\Z\zeta_5^i
  +4\sum_{i=0}^{3}\Z_K\zeta_5^i\\
&=2\Z_K+4\Z_K\\
&=2\Z_K.
\end{align*}
Here $\sum_{i=0}^{3}\Z\zeta_5^i=\Z_K$, and $\Z_K\zeta_5^i=\Z_K$ because each $\zeta_5^i$ is a unit of $\Z_K$. Let $\frakm=2\Z_K$, the maximal ideal of $R$ above $2$. Then \[(2\Z_K)/\frakm(2\Z_K)
    = (2\Z_K)/(4\Z_K)
    \xrightarrow{\ \sim\ } \Z_K/(2\Z_K)
    \cong \F_{16}, \text{where } 2x\longmapsto x. \]

Since $R/\frakm\cong\F_2$, the field $\F_{16}$ has dimension $4$ as an $\F_2$-vector space. By Nakayama's lemma, the minimal number of generators of $(2\Z_K)_{\frakm}$ over the local ring $R_{\frakm}$ equals this dimension $\mu_{R_{\frakm}}\bigl((2\Z_K)_{\frakm}\bigr) = \dim_{\F_2}\frac{2\Z_K}{4\Z_K} = 4.$ Therefore the four displayed generators are necessary. Thus the fractional $R$-ideal $2\Z_K$ is not generated by at most two elements, and Definition~\ref{def:bass-order} shows that $R$ is not Bass. For $R'$, we have $\Z_{K_0}\subset R'$ and $[K:K_0]=2$, so the index-$2$ subfield criterion shows that $R'$ is Bass. The equality $\mu_{R'}(\Z_K)=2$ follows from $\dim_{\F_4}(\F_{16})=2$. \newline 

The structure of $\frakC(R)$ is given by the relative-norm presentation of equation~\eqref{eq:shimura_exact_sequence} in Section~\ref{sec:shimura_cl_group}. Hasse's norm theorem for the quadratic extension $K / K_0$ in \cite[Cor.\,VI.4.5]{Neukirch1999} describes the elements of $K_0^\times / N_{K/K_0}(K^\times)$ in terms of local Hasse symbols at the places of $K_0$ ramifying in $K/K_0$. This is the input that controls the integer-polarized classes on a elements level. For $R = \Z_K$, Streng~\cite[Lem.\,I.8.4, Thm.\,III.2.2]{StrengThesis} shows that the cokernel of the type norm map into $\frakC(\Z_K)$ has exponent dividing $2$; the analogous result in dimension $g = 3$ is~\cite{sijsling-dina}, where the exponent is $2$ or $4$ depending on the Galois type.

In the integer-polarized version we expect a sharper statement, for which we have computational evidence in the maximal-order case.

\begin{conjecture}
\label{conj:integer_polarization_dichotomy}
Let $K$ be a CM field and let $R \subseteq \Z_K$ be a CM order. Then
\[
    [\frakC(R) : \frakC(R)_\ZZ] \;\in\; \{1,\ \#\frakC(R)[2]\}.
\]
\end{conjecture}

\begin{remark}
\label{rmk:dichotomy_cases}
Our computational evidence is for maximal orders across the tested Galois types for $g \in \{2, 3\}$. Computations for non-maximal orders and higher dimensions are in progress and already provide further evidence for our claim.
\end{remark}

\section{Superspecial reduction} \label{app:reduction}

\subsubsection{Reduction.} 

We recall the superspecial reduction results Goren established for CM abelian surfaces in the primitive
quartic case~\cite{Goren1997}.  Goren did not consider the biquadratic $C_{2}\times C_{2}$ case because his setup only allowed for so-called ``primitive" CM types (see Section~\ref{sec:security} for a definition). His argument, however, carries over almost verbatim and gives Proposition~\ref{prop:reductionC2C2} below, which we record as a new result even though its proof is analogous to Goren's. As in~\cite{Goren1997}, we restrict throughout this appendix to $R = \Z_K$ and to primes $p$ unramified in $K$; the ramified case is left for future work (see the end of this appendix).

\begin{proposition} \label{thm:GorenC4}
Let $K$ be a cyclic quartic CM field, let $p$ be a prime number not ramifying in $K$, and let $(A, \lambda, \iota) \in \M_k(\Z_K)$ with $k = \overline{\Q}$. Let $L \subset k$ be a number field over which $A$ is defined. Let $\bar \frakp \subset \Z_L$ be a prime ideal above $p$. Assuming the reduction $A \bmod \bar \frakp$ is good, it is superspecial if and only if $p$ neither splits completely nor is inert in $K$.
\end{proposition}

\begin{proof}
    See~\cite[Thm.\,1]{Goren1997}. \qed
\end{proof}

\noindent In other words, in the $C_4$ case we have superspecial reduction if and only if $p \Z_K = \frakp_1 \frakp_2$ for two degree-$2$ prime ideals $\frakp_i$ of $K$.

In the cases $C_{2} \times C_{2}$ and $D_4$, the factorization pattern of $p$ in $K$ alone does not distinguish superspecial from non-superspecial and the CM type $\Phi$ comes into play. It should be clear, however, that the reduction type can only depend on $\Phi$ up to equivalence.

\begin{proposition} \label{prop:reductionC2C2}
We inherit the assumptions from Proposition~\ref{thm:GorenC4} but now assume that $K$ is biquadratic. We can assume, by precomposing $\iota$ with complex conjugation if needed, that the CM type of $(A, \lambda, \iota)$ is of the form $\Phi = \{ \id, \tau \}$ for some $\tau \in \Gal(K , \Q)$. Let $K_1$ be the fixed field of the subgroup $\Phi$. Assuming the reduction $A \bmod \bar \frakp$ is good, it is superspecial if and only if $p$ is inert in $K_1$.
\end{proposition}

\begin{proof} 
We can assume that $L$ contains $K$ and is Galois over $\Q$. Ekedahl's product criterion~\cite[Prop.\,2.4]{ekedahl-supersingular}, as discussed in~\cite[\S2]{Goren1997}, says that $A \bmod \bar \frakp$ is superspecial if and only if
\[ \sigma \bar{\Phi}^{\textrm{ext}} = \Phi^{\textrm{ext}} \]
where $\sigma \in \Gal(L, \Q)$ denotes a Frobenius element of $\bar \frakp$ (i.e., an automorphism such that $\sigma(\bar \frakp) = \bar \frakp$ which reduces to the $p$-th power Frobenius automorphism of $\Z_L/\bar \frakp$ over $\F_p$) and $\Phi^{\textrm{ext}}$ denotes the set of elements of $\Gal(L, \Q)$ whose restriction to $K$ is contained in $\Phi$ (and likewise for $\bar{\Phi}^{\textrm{ext}}$). Now let $\sigma^\text{res}$ be the Frobenius element of $\bar \frakp \cap K$ (unique since $p$ does not ramify). This equals the restriction of $\sigma$ to $K$, so Ekedahl's criterion becomes: $A \bmod \bar \frakp$ is superspecial if and only if $\sigma^\text{res} \notin \{ \id, \tau\}$. The result follows.
\qed
\end{proof}

\noindent In other words, in the $C_2 \times C_2$ case we have superspecial reduction if and only if $p \Z_K = \frakp_1 \frakp_2$ for two degree-$2$ prime ideals $\frakp_i$ of $K$ and the splitting is complete in $K_0$ or in the imaginary quadratic subfield different from $K_1$. As a sanity check, consider our recurring orientations by
\[ \Z \left[\frac{1 + \sqrt{-p-4}}{2}, \sqrt{\frac{p+1}{4}} \right] \]
in characteristic $p \equiv 3 \bmod 4$
from Examples~\ref{ex:example_orientation} and~\ref{ex:secondorientation}, shown to come from two inequivalent CM types in Examples~\ref{ex:example_orientation_revisited} and~\ref{ex:example_orientation_revisited_revisited}. By Proposition~\ref{prop:reductionC2C2} combined with the lifting statement from~\cite[Lem.\,4.4.1]{Goren_Lauter} (which was used in our proof of Theorem~\ref{thm:transitivitymodp}),
this must mean that $p$ splits in $K_0$ and is inert in the two imaginary quadratic subfields $K_1, K_2$. This is indeed easily checked.

\begin{proposition}
\label{prop:reductionD4}
Let $K$ be a
quartic CM field with Galois closure $M$ such that $\Gal(M, \Q) = \langle x, y \, | \, x^{2}, y^{4}, (xy)^{2} \rangle \cong D_4$ and that $K = M^{\langle x \rangle}$. Let $\Phi = \{\id, y\}$ be a fixed CM type on~$K$.  Under the remaining assumptions of Proposition~\ref{thm:GorenC4}, where we can assume $M \subset L$ w.l.o.g., and with $p$ unramified in~$M$, the reduction $A \bmod \bar\frakp$ is superspecial if and only if the 
Frobenius element of $\bar\frakp \cap M$ equals $y^2$ or $xy$.
\end{proposition}

\begin{proof} This is~\cite[Thm.\,2]{Goren1997} in cases~(3) and~(6). The two cases together are again seen to be equivalent to the product criterion of Ekedahl. \qed \end{proof}

\medskip
This closes the unramified cases for dimension $2$ (assuming $R = \Z_K$).  The ramified
case, where $p \mid \disc(K/\Q)$, is not covered by~\cite[Thm.\,1 and Thm.\,2]{Goren1997}.  In forthcoming work we treat this case by replacing
the unramified product criterion with a Lie-type Dieudonn\'e module
calculation.  The Shimura--Taniyama formula gives the Newton polygon,
and the Lie type gives the $a$-number.  Superspeciality is then the
condition that all slopes are $1/2$ and that $a=g$.

\end{document}